\documentclass[
a4paper,
11pt, 
DIV=13,
toc=bibliography, 
headsepline, 
parskip=half-,
abstract=true,
]{scrartcl}

\usepackage[english]{babel}
\usepackage[utf8]{inputenc}
\usepackage[T1]{fontenc}
\usepackage[english,cleanlook]{isodate}
\usepackage{lmodern}
\usepackage[babel=true]{microtype}
\usepackage{amsmath}
\usepackage{mathtools}
\usepackage{amssymb}
\usepackage{amsthm}
\usepackage{scrlayer-scrpage} 
\usepackage{enumitem}
\usepackage{tabularx}
\usepackage{tikz} 
\usepackage{tikz-3dplot}  

\ihead{M.\,Karpukhin, M.\,B.\,Schulz}
\ohead{\csname @title\endcsname}
\setkomafont{pageheadfoot}{\footnotesize}

\providecommand{\C}{\mathbb{C}}
\providecommand{\R}{\mathbb{R}}
\providecommand{\Z}{\mathbb{Z}}
\providecommand{\N}{\mathbb{N}}
\providecommand{\B}{\mathbb{B}}
\providecommand{\T}{\mathbb{T}}
\providecommand{\Sp}{\mathbb{S}} 
\providecommand{\dih}{\mathbb{D}}
\providecommand{\rotation}{\mathsf{R}}
\providecommand{\cA}{\mathcal{A}} 
\providecommand{\IrRep}[1]{\mathrm{IrRep}_{#1}}
\providecommand{\spec}{\mathrm{S}}
\providecommand{\ener}{\mathrm{E}}
\providecommand{\teich}{\mathrm{T}}      
\providecommand{\area}{\mathrm{A}}
\providecommand{\jacobi}{\mathrm{J}}
\providecommand{\PTr}{\mathcal{S}_{\mathrm{tr}}}
\providecommand{\Div}{\mathcal{S}_{\mathrm{div}}}
\providecommand{\DF}{\mathcal{S}_{\mathrm{tt}}}
\providecommand{\VConf}{\mathfrak{C}}
\providecommand{\triv}{\mathbf{1}}
\providecommand{\Or}{\mathbf{Or}}
\providecommand{\st}{\mathbf{St}}
\providecommand{\m}{\mathbf{m}}

\DeclarePairedDelimiter\abs{\lvert}{\rvert}
\DeclarePairedDelimiter\nm{\lVert}{\rVert}
\DeclarePairedDelimiter\brt{\langle}{\rangle} 
\DeclarePairedDelimiter\interval{]}{[}

\DeclareMathOperator{\SOgroup}{SO}  
\DeclareMathOperator{\Ogroup}{O}  
\DeclareMathOperator{\GLgroup}{GL}  
\DeclareMathOperator{\ind}{ind}
\DeclareMathOperator{\nul}{nul} 
\DeclareMathOperator{\Ind}{\mathcal{I}}
\DeclareMathOperator{\Nul}{\mathcal{N}}
\DeclareMathOperator{\tr}{tr}
\DeclareMathOperator{\Fix}{Fix}
\DeclareMathOperator{\Hom}{Hom}
\DeclareMathOperator{\End}{End}
\DeclareMathOperator{\Diff}{Diff}
\DeclareMathOperator{\divop}{div}
\DeclareMathOperator{\II}{I\hspace*{-.5pt}I}

\usetikzlibrary{plotmarks,calc} 
\pgfmathsetmacro{\picscale}{0.47}
\providecommand{\FBMS}[2][\picscale]{\draw(0,0,0)node[inner sep=0]{\includegraphics[width=#1\textwidth]{#2}};}
\pgfmathsetmacro{\thetaO}{72} 
\pgfmathsetmacro{\phiO}{90-45/2} 
\tdplotsetmaincoords{\thetaO}{\phiO}
\newenvironment{unitball}[1][\picscale]{%
\begin{tikzpicture}[scale={#1*\textwidth/2cm},tdplot_main_coords,line cap=round,line join=round,semithick,baseline={(0,0,0)}]
\useasboundingbox[tdplot_screen_coords](-1,-1)rectangle(1,1);
\pgfmathsetmacro{\laengeA}{-135-45/2+\phiO}
\pgfmathsetmacro{\laengeB}{\laengeA+45}
\begin{scope}[black!30]
\path[tdplot_screen_coords,inner color=white,outer color=cyan!5!black!20](0,0)circle(1);
\foreach\laengengrad in {\laengeA,\laengeB,...,\phiO}{
\tdplotsetthetaplanecoords{\laengengrad}
\pgfmathsetmacro{\start}{-atan( 1/(tan(\thetaO)*( cos(\laengengrad)*sin(\phiO)-cos(\phiO)*sin(\laengengrad))))}
\tdplotdrawarc[tdplot_rotated_coords]{(0,0,0)}{1}{180+\start}{360+\start}{}{}
}
\foreach\breitengrad in {-60,-30,...,60}{
\pgfmathsetmacro{\breitenkreisradius}{cos(\breitengrad)}
\pgfmathsetmacro{\hoehe}{sin(\breitengrad)}\tdplotdrawarc{(0,0,\hoehe)}{\breitenkreisradius}{0}{360}{}{}
}
\end{scope}
}{\begin{scope}[black]
\foreach\laengengrad in {\laengeA,\laengeB,...,\phiO}{
\tdplotsetthetaplanecoords{\laengengrad}
\pgfmathsetmacro{\start}{-atan( 1/(tan(\thetaO)*( cos(\laengengrad)*sin(\phiO)-cos(\phiO)*sin(\laengengrad))))}
\tdplotdrawarc[tdplot_rotated_coords]{(0,0,0)}{1}{\start}{180+\start}{}{}
} 
\foreach\breitengrad in {-60,-30,...,60}{
\pgfmathsetmacro{\breitenkreisradius}{1*cos(\breitengrad)}
\pgfmathsetmacro{\hoehe}{1*sin(\breitengrad)}
\pgfmathsetmacro{\visiblerange}{acos(-sign(\breitengrad)*min(1,abs(tan(\breitengrad)*tan(90-\thetaO)))  )}
\ifdim\visiblerange pt>0pt
\tdplotdrawarc{(0,0,\hoehe)}{\breitenkreisradius}{\phiO-90+\visiblerange}{\phiO-90-\visiblerange}{}{}
\fi
}
\draw[tdplot_screen_coords] (0,0) circle (1);
\end{scope}%
\end{tikzpicture}%
}
\makeatletter
\tikzset{
    scale plot marks/.is choice,
    scale plot marks/true/.style={},	
    scale plot marks/false/.code={
        \def\pgfuseplotmark##1{\pgftransformresetnontranslations\csname pgf@plot@mark@##1\endcsname}
    },
every mark/.append style={scale plot marks=false},
bullet/.style={mark=*,mark size=1.125pt},
}
\makeatother
\definecolor{FarbeA}{cmyk}{0,0.5,1,0}
\definecolor{FarbeB}{cmyk}{1,0.5,0,0}	

\newcolumntype{x}[1]{>{\centering\arraybackslash$}p{#1}<{$}}

\usepackage[initials]{amsrefs}
\makeatletter
\renewcommand{\amsrefs@warning}[1]{}  
\makeatother
 
\usepackage[
pdftitle={Index and nullity of minimal surfaces via representation theory},
pdfauthor={Mikhail Karpukhin and Mario B. Schulz}, 
pdfborder={0 0 0},
]{hyperref}
\usepackage{bookmark}

\theoremstyle{plain}
\newtheorem{theorem}{Theorem}[section] 
\newtheorem{lemma}[theorem]{Lemma}
\newtheorem{corollary}[theorem]{Corollary}  
\newtheorem{proposition}[theorem]{Proposition} 

\theoremstyle{definition}
\newtheorem{definition}[theorem]{Definition}
 
\theoremstyle{remark}
\newtheorem{remark}[theorem]{Remark} 
\newtheorem{example}[theorem]{Example} 

\numberwithin{equation}{section} 

\RedeclareSectionCommand[tocbeforeskip=1.5ex plus 1ex minus .5ex]{section}
\RedeclareSectionCommand[tocbeforeskip=0ex plus 1ex]{subsection}
\RedeclareSectionCommand[beforeskip=1ex plus 1.5ex minus .5ex]{paragraph}

\subject{\vspace*{-6ex}} 
\title{Index and nullity of minimal surfaces via~representation theory} 
\author{Mikhail Karpukhin and Mario B. Schulz}
\date{\vspace*{-4ex}} 

\newcommand\printaddress{{%
\setlength{\parindent}{17pt}
\small
\smallskip
\hfill\printdate{02.10.2026}
\par
{\scshape Mikhail Karpukhin}
\newline
University College London, 
Department of Mathematics, 
25 Gordon Street, London, WC1H 0AY, UK
\newline
\textit{E-mail address:} 
\texttt{m.karpukhin@ucl.ac.uk}
\par\medskip
{\scshape Mario B. Schulz}
\newline 
Università di Trento, 
Dipartimento di Matematica, 
via Sommarive 14, 
38123 Povo di Trento, 
Italy 
\newline
\textit{E-mail address:} 
\texttt{mario.schulz@unitn.it}
\par
}}   

\begin{document}

\maketitle

\begin{abstract} 
We develop a representation-theoretic approach to determine the Morse index and nullity of an embedded minimal surface based only on its topology and symmetry group.
The method combines an equivariant comparison between the second variations of area and energy with the representation theory of the dihedral group. 
In the Euclidean unit ball we compute for every $k\geq3$ the Morse index and nullity of any free boundary minimal $k$-noid with prismatic symmetry of order $4k$, 
and of any free boundary minimal surface of genus one with connected boundary and antiprismatic symmetry of order eight.
In the round sphere we obtain an alternative computation for the Morse index and nullity for the first family of the Lawson surfaces, using only half of their symmetry group. 
The results are independent of any particular construction of the surfaces.
\end{abstract}

\tableofcontents

\clearpage

\section{Introduction}

A minimal surface $\Sigma$ in a given Riemannian manifold is a critical point for the area functional.
We either assume that $\Sigma$ and the ambient manifold are both closed, i.\,e.~compact without boundary, or that $\partial\Sigma$ is constrained to the ambient boundary but otherwise ``free''.
Minimal surfaces are characterised by vanishing mean curvature.
Free boundary minimal surfaces $\Sigma$ additionally meet the ambient boundary orthogonally along $\partial\Sigma$. 
The \emph{Morse index} of a minimal surface is defined as the maximal dimension of a space of smooth normal variations on which the second variation of area is negative definite, while the \emph{nullity} is the dimension of its kernel.
The Morse index provides another measure for the complexity of $\Sigma$ besides its topology and quantifies the instability of the minimal surface.

Natural choices for the ambient manifold are the round sphere $\Sp^3$ in the closed setting and the Euclidean unit ball $\B^3$ in the free boundary setting as defined in \eqref{eqn:ball-sphere}. 
In these geometries the ambient coordinate functions restricted to $\Sigma$ solve the Laplace respectively Steklov eigenvalue problem, linking the theory of minimal surfaces to spectral geometry.
While there are many examples of closed minimal surfaces in $\Sp^3$ and free boundary minimal surfaces in $\B^3$, the value of their Morse index is known explicitly only in very few cases: 
\begin{itemize}[nosep]
\item 
The equatorial $2$-sphere in $\Sp^3$ and the equatorial disc in $\B^3$ both have Morse index $1$. 
\item 
Urbano \cite{Urbano1990} proved that the Clifford torus in $\Sp^3$ has Morse index $5$.
\item 
The critical catenoid in $\B^3$ has Morse index $4$ by Devyver \cite{Devyver2019}, Smith--Zhou \cite{SmithZhou2019}, and Tran \cite{Tran2020}.
Any example in $\B^3$ other than the equatorial disc has index at least $4$ by \cite{Devyver2019}.
\item 
Kapouleas and Wiygul \cite{KapouleasWiygul2020} computed the Morse index for the Lawson surfaces $\xi_{1,k-1}\subset\Sp^3$. 
\item 
Kapouleas and Zou \cite{KapouleasZou2025} computed the Morse index for certain minimal doublings of the $2$-sphere in $\Sp^3$, 
provided that their genus is sufficiently large.
\end{itemize} 
The main motivation for the present article is to add the first infinite family of free boundary minimal surfaces in $\B^3$
to this list of examples with known Morse index (see Theorem~\ref{thm:index-fbm-k-noid}).
While exact computations for concrete examples remain challenging, 
several methods have been developed to estimate the Morse index in terms of geometric and topological data. 

\emph{Harmonic forms and test functions:}
Building on ideas introduced by Ros \cite{Ros2006} and Savo \cite{Savo2010}, Ambrozio, Carlotto, and Sharp \cite{AmbrozioCarlottoSharp2018Index,AmbrozioCarlottoSharp2018IndexFBMS} and independently Sargent \cite{Sargent2017}  
developed a general method based on testing the second variation on coordinate components of harmonic $1$-forms:  
\cite[Theorem~C]{AmbrozioCarlottoSharp2018IndexFBMS} states that the Morse index of any compact, orientable, properly embedded free boundary minimal surface of genus $g$ with $b$ boundary components in a mean-convex domain of $\R^3$
is bounded from below by $\frac{1}{3}(2g+b-1)$. 
General upper bounds linear in the genus and area were obtained by Ejiri and Micallef \cite{EjiriMicallef2008} for closed minimal surfaces, and by Lima \cite{Lima2022} in the free boundary setting.
While capturing index growth, such bounds are generally far from sharp.

\emph{Partitioning methods} 
taking symmetry into account can improve general index estimates significantly. 
The idea is to estimate the number of negative eigenvalues of an operator by partitioning the domain suitably, while adjoining Dirichlet boundary conditions on the cuts for lower bounds and Neumann conditions for upper bounds. 
This approach was originally developed by Montiel and Ros \cite{MontielRos1991} in the context of complete minimal surfaces in $\R^3$ with finite total curvature.   
Kapouleas and Wiygul \cite[Theorem~6.21]{KapouleasWiygul2020} provided the first application in the setting of $\Sp^3$. 
The methodology was extended further to compact manifolds with boundary in \cite{CSWSpectral,CSWstackings}, yielding effective two-sided bounds for several families of free boundary minimal surfaces in $\B^3$, and also for hypersurfaces in higher dimensions \cite{CSWIndexgrowth}. 

\emph{Covering arguments} introduced by Song \cite[Theorem~1]{Song2023Index} can be used to bound the Betti numbers of closed minimal hypersurfaces linearly by their Morse index in ambient dimensions $3\leq n+1\leq 7$.
This technique was adapted to properly embedded free boundary minimal hypersurfaces by Cordero-Misteli and Franz \cite[Theorem~1.1]{CorderoMisteliFranz2024}.
While applicable without symmetry or curvature hypotheses, the resulting constants are non-explicit and depend on area bounds.

\emph{Asymptotic analysis:}
For minimal surfaces constructed by PDE gluing methods, Kapouleas and Zou \cite{KapouleasZou2025} developed an approach to compute the exact index and nullity by analysing the spectrum of the Jacobi operator on the individual building blocks and determined Morse index and nullity for a family of minimal doublings in $\Sp^3$. 
Because this approach relies on perturbation techniques around the singular configuration, the exact values are obtained only in the asymptotic regime of sufficiently large topological complexity.

\subsection{Main results}

In this paper, we develop a representation-theoretic approach to compute the Morse index and nullity of an embedded minimal surface with a discrete symmetry group. 
Our results are independent of any specific existence construction since they rely only on assumptions on topology and symmetry group.  
Our first application fully resolves \cite[Conjecture~7.8\,(iv)]{CSWnonuniqueness} and provides the Morse index and nullity for the first infinite family of free boundary minimal surfaces in the Euclidean unit ball. 

\begin{theorem}
\label{thm:index-fbm-k-noid}
Given $3\leq k\in\N$ let $\Sigma_k\subset\B^3$ be any properly embedded free boundary minimal surface of genus zero with $k$ boundary components and prismatic symmetry group $G$ of order~$4k$. 
Then $\Sigma_k$ has Morse index equal to $2k$, nullity $3$, and $G$-equivariant index $1$. 
\end{theorem}

\begin{remark}
``Prismatic symmetry'' refers to the symmetry group of a regular $k$-gonal prism and is made precise in Definition~\ref{defn:prismatic}. 
Free boundary minimal surfaces $\Sigma_k$ satisfying the assumptions of Theorem~\ref{thm:index-fbm-k-noid} have been constructed (for sufficiently large $k$) in \cite{FolhaPacardZolotareva2017} and (for all $k\geq3$) in \cite[§\,5]{FranzSchulz2023} (cf.~\cite[Remark~1.2]{FranzKetoverSchulz}) and \cite{KarpukhinKusnerMcGrathStern}. 
It is expected that these constructions yield the same surface for each $k$.  
\end{remark}

After completion of this work, we learned that Tran \cite{Tran2026} independently obtained the Morse index and nullity of the surfaces in Theorem~\ref{thm:index-fbm-k-noid} under additional assumptions on their boundary components, using circle domains instead of the Eichler trace formula stated in Theorem~\ref{thm:Eichler}.

Our theory developed for Theorem~\ref{thm:index-fbm-k-noid} applies equally well to minimal surfaces in the round sphere $\Sp^3$. 
For all $3\leq k\in\N$, the Lawson \cite[§\,6]{Lawson1970} surface $\xi_{1,k-1}\subset\Sp^3$ is a closed minimal surface with genus $k-1$ and symmetry group of order $16k$. 
(See Section \ref{sec:Lawson} for details.) 
Kapouleas and Wiygul \cite[Theorem 6.21]{KapouleasWiygul2020} 
exploited the full symmetry group to prove that $\xi_{1,k-1}$ has Morse index $2k+1$ and nullity $6$ for each $3\leq k\in\N$.  
We prove the following generalisation, which relies only on a subgroup of half the order and is independent of Lawson's \cite[§\,6]{Lawson1970} construction. 

\begin{theorem} 
\label{thm:index-Lawson}
Given $3\leq k\in\N$, any closed, embedded minimal surface of genus $k-1$ in $\Sp^3$ which is invariant under an action of $\dih_k\times\Z_2\times\Z_2$ in the representation \eqref{eqn:repR4} has Morse index equal to $2k+1$ and nullity $6$.
\end{theorem}

\begin{remark}
The group $G=\dih_k\times\Z_2\times\Z_2$ has order $8k$ and its action is made precise in Section~\ref{sec:Lawson}. 
Geometrically, the dihedral group $\dih_k$ acts on the $x_1x_2$-plane while the $\Z_2$-factors act by the reflections across $\{x_3=0\}$ and $\{x_4=0\}$.
The Lawson surfaces $\xi_{1,k-1}$ satisfy the assumptions of Theorem~\ref{thm:index-Lawson}, their full symmetry group being of order $16k$, see Remark~\ref{rem:Lawson_symmetry}. 
Kapouleas and Wiygul \cite[Theorem~4.5]{KapouleasWiygul2022} proved that any closed, embedded minimal surface of genus $k-1$ in $\Sp^3$ which is invariant under the full symmetry group of $\xi_{1,k-1}$ is congruent to $\xi_{1,k-1}$.
The same conclusion is expected under the subgroup $G$ of order $8k$, which consists of those symmetries preserving the two sides of the surface, see \cite[Remark~4.6 and Lemma~3.11\,(ii)]{KapouleasWiygul2022}.
As long as this remains open, Theorem~\ref{thm:index-Lawson} covers a potentially larger class of surfaces.  
\end{remark}

According to \cite[Table~1]{FranzKetoverSchulz}, the next simplest free boundary minimal surface in $\B^3$ beyond the equatorial disc and the critical catenoid has genus one and connected boundary. 
Our third application resolves the conjecture therein about its Morse index and, together with Theorem~\ref{thm:index-fbm-k-noid} and 
\cite{Devyver2019,SmithZhou2019,Tran2020}, reveals the first consecutive triple of known Morse indices for free boundary minimal surfaces: the critical catenoid of index $4$, the punctured torus of index $5$, and the trinoid of index $6$ (see Figure~\ref{fig:fbms}).

\begin{figure}
\begin{unitball}%
\FBMS{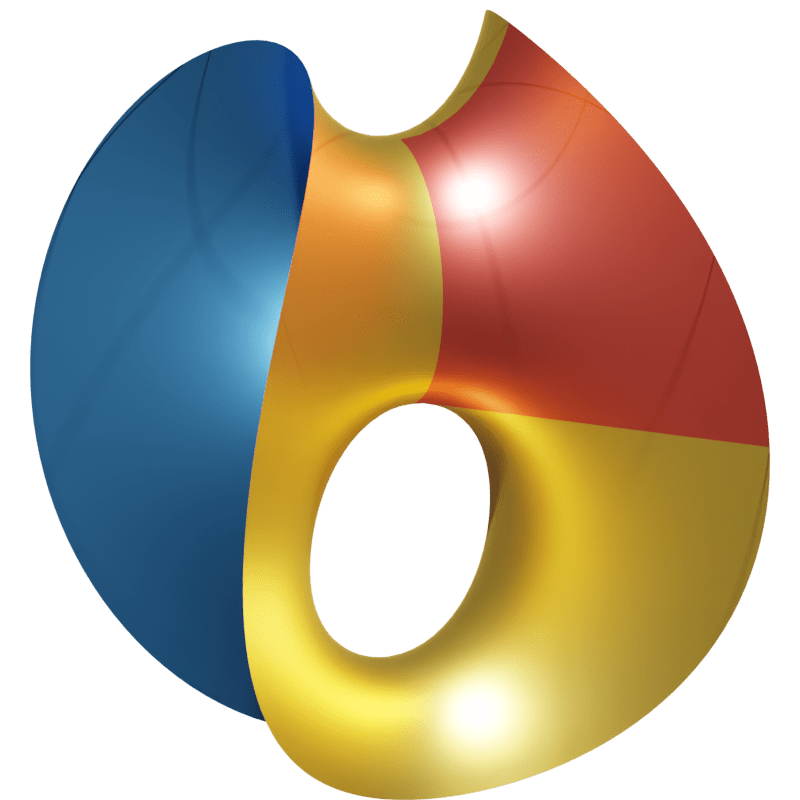}%
\end{unitball}
\hfill
\begin{unitball}%
\FBMS{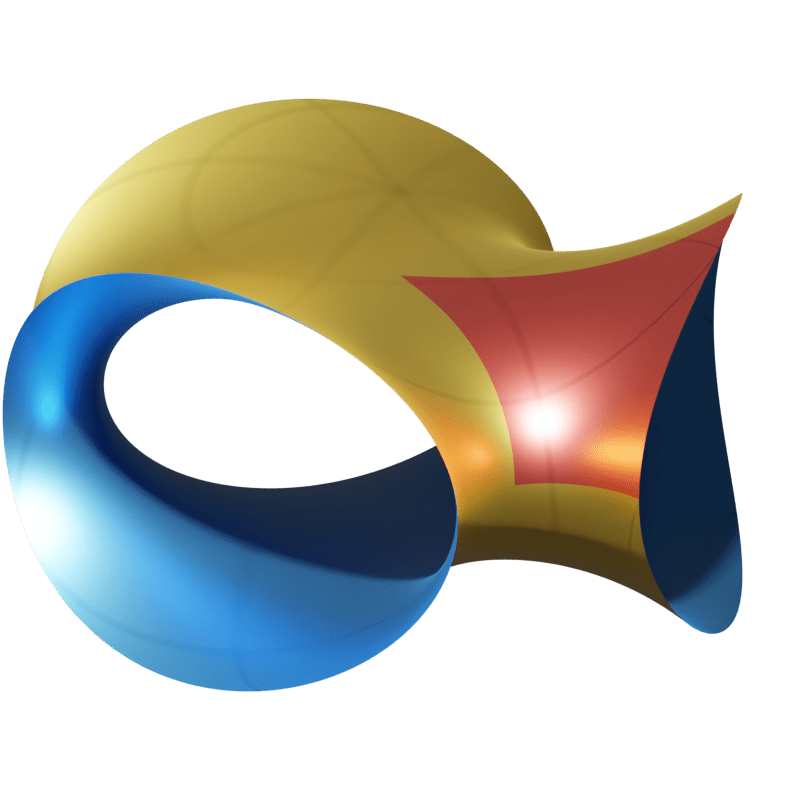}%
\end{unitball}
\caption{Left: 
Free boundary minimal surface in $\B^3$ with Morse index  $5$ 
as in Theorem~\ref{thm:index-fbm-torus}. \\
Right: Free boundary minimal trinoid in $\B^3$ with Morse index $6$ as in Theorem~\ref{thm:index-fbm-k-noid}. \\
Highlighted is a fundamental domain of the respective symmetry group. 
}%
\label{fig:fbms}%
\end{figure}

\begin{theorem}
\label{thm:index-fbm-torus}
Let $M_1\subset\B^3$ be any properly embedded free boundary minimal surface of genus~$1$ with connected boundary and antiprismatic symmetry group of order $8$. 
Then $M_1$ has Morse index equal to $5$ and nullity $3$. 
\end{theorem}

\begin{remark}
``Antiprismatic symmetry'' refers to the symmetry group of a regular antiprism and is formalised in Definition~\ref{defn:antiprismatic}. 
The existence of a free boundary minimal surface $M_1\subset\B^3$ 
satisfying the assumptions of Theorem~\ref{thm:index-fbm-torus} 
follows from \cite[Theorem~1.1]{CarlottoFranzSchulz2022} and \cite[Remark~1.2]{FranzKetoverSchulz}. 
Independently, the existence of a free boundary minimal surface $\Gamma\subset\B^3$ with Morse index equal to $5$, and the topology of either an annulus or a punctured torus, follows from \cite[§\,2]{FranzSchulz2023} and \cite{ChunPongChu2022}. 
\end{remark} 
 
\begin{remark}[No exceptional Jacobi fields]
In all three applications presented above we may conclude that the nullity is entirely accounted for by the ambient rotations.  
\end{remark}

\subsection{Strategy of the proof} 
Our method requires only three ingredients: 
the symmetry group together with its action on the ambient space, the topology of the surface, and a sharp bound on its first eigenvalue. 
Specific existence results do not play a role. 
Our approach is to control the representation of the symmetry group on the negative space of the Jacobi operator. 
The proof combines an equivariant version of the comparison between energy and area indices~\cite{EjiriMicallef2008, Moore2007, Lima2022}, 
the representation theory of the dihedral group, and, crucially, an abstract 
spectral-theoretic lemma in the spirit of~\cite{Helffer2002} that greatly limits the symmetry types of the low eigenfunctions.
We outline the argument in the case of Theorem~\ref{thm:index-fbm-k-noid}. 
The strategy for Theorems~\ref{thm:index-Lawson} and~\ref{thm:index-fbm-torus} is similar. 

We decompose the Morse index of the surface $\Sigma_k$ from Theorem~\ref{thm:index-fbm-k-noid} into even and odd parts $\ind^\pm(\Sigma_k)$ according to the $\Z_2$-factor of the symmetry group generated by the reflection $x_3\mapsto-x_3$.
The Jacobi field induced by the rotation around the vertical axis is even and gives the lower bound $\ind^+(\Sigma_k)\geq 2k-1$ by nodal domain counting. 
Similarly, considering the rotation around a horizontal axis, we obtain $\ind^-(\Sigma_k)\geq 1$. 
An approach following Kapouleas and Wiygul~\cite{KapouleasWiygul2020} would 
require graphicality of a fundamental domain with respect to rotations around those axes to establish the matching upper bounds. 
Their proof of the graphicality property for the Lawson surfaces uses the symmetries reversing the unit normal, which do not have an analogue for $\Sigma_k$. 
Instead, we prove the upper bounds by representation-theoretic means, 
avoiding graphicality altogether.

The coordinate functions on $\Sigma_k$ being first Steklov eigenfunctions~\cite{McGrath2018,KusnerMcGrath2020} bounds its energy index by $3$. 
By comparing energy and area indices~\cite{Lima2022} we obtain $\ind(\Sigma_k)\leq 3+\dim\DF(\Sigma_k)$, where $\DF(\Sigma_k)$ denotes the space of transverse-traceless bilinear forms on $\Sigma_k$. 
Theorem~\ref{thm:k-noidDF} determines the relevant isotypic components of the induced representation of the symmetry group on $\DF(\Sigma_k)$. 
Together with the energy contribution this yields an explicit representation $\cA$ of degree (i.\,e.~dimension) $3k-3$ which contains the representation on the negative space of the Jacobi operator as a subrepresentation.
The even part of $\cA$ has degree $2k-1$, which matches the lower bound and determines $\ind^+(\Sigma_k)$. 
For $\ind^-(\Sigma_k)$ we prove Proposition~\ref{prop:nodal_dihedral}, which leaves only two symmetry types for the second odd eigenfunction of the Jacobi operator. 
However, the corresponding isotypic components of $\cA$ have total degree $1$. 
Hence, the first two odd eigenfunctions cannot both have negative eigenvalues, which implies $\ind^-(\Sigma_k)=1$, completing the index computation.  

The general theory is developed in Sections~\ref{sec:equivariant-harmonic} and~\ref{sec:area}, where we carefully adapt known results about energy and area indices to the equivariant setting.
Our exposition of the comparison between the area and energy indices is purely Riemannian and avoids complex geometry, unlike the expositions \cite{EjiriMicallef2008, Moore2007, Lima2022}. 
In particular, our results apply directly to non-orientable surfaces.

\paragraph{Acknowledgments.}
This project has received funding from the European Research Council (ERC) under the European Union's Horizon 2020 research and innovation programme (grant agreement No.~947923). 

M.\,K.~would like to thank the Isaac Newton Institute for Mathematical Sciences, Cambridge, for support and hospitality during the programme Geometric spectral theory and applications, where work on this paper was undertaken. This work was supported by EPSRC grant EP/Z000580/1.

This work is supported by ERC grant (EOaMS 101219977). Funded by the European Union. Views and opinions expressed are however those of the authors only and do not necessarily reflect those of the European Union or the European Research Council. Neither the European Union nor the granting authority can be held responsible for them.
 
\section{Index and nullity of equivariant harmonic maps}
\label{sec:equivariant-harmonic}

We define the energy and spectral quadratic forms together with their index and nullity for harmonic maps, and then incorporate a finite symmetry group to obtain the isotypic refinements. 

The domain $(M,g)$ is a compact, connected Riemannian manifold with (possibly empty) boundary $\partial M$. 
Unless stated otherwise, no restriction is placed on the dimension of $M$ in this section. 
(We will specialise to $\dim M=2$ starting in Section~\ref{sec:area}.) 

Given $n\in\N$, we equip $\R^{n+1}$ with the Euclidean inner product $\brt{\cdot,\cdot}$ and Cartesian coordinates $x=(x_1,\ldots,x_{n+1})$ and denote the unit ball and the unit sphere by 
\begin{align}
\label{eqn:ball-sphere}
\begin{aligned}
\B^{n+1}&\vcentcolon=\{x\in\R^{n+1}\mid x_1^2 + \ldots + x_{n+1}^2 \leq 1\}, 
\\[.25ex]
\Sp^n&\vcentcolon=\{x\in\R^{n+1}\mid x_1^2 + \ldots + x_{n+1}^2 = 1\}=\partial\B^{n+1}.
\end{aligned}
\end{align} 

\subsection{Energy index and spectral index}
\label{sec:energy-spectral-index}

The \emph{Dirichlet energy} of a smooth map  
\(
\Phi=(\Phi_1,\ldots,\Phi_{n+1})\colon M\to\R^{n+1}   
\)
is defined as 
\begin{align}\label{eqn:energy}
E_g(\Phi)\vcentcolon=\frac{1}{2}\int_M\abs{d\Phi}^2_g.
\end{align}
Given a potential $q\in C^\infty(M)$ and a boundary potential $\sigma\in C^\infty(\partial M)$, we define the \emph{energy quadratic form} for vector-valued functions $V \in C^\infty(M, \R^{n+1})$ by 
\begin{align}\label{eqn:energy-index-form}
Q_{\ener}(V,V)&\vcentcolon= \int_M \Bigl(\abs{dV}_g^2 - q\abs{V}^2\Bigr) - \int_{\partial M} \sigma \abs{V}^2,
\intertext{and we define the \emph{spectral quadratic form} 
for scalar functions $u \in C^\infty(M)$ by 
}
\label{eqn:spectral-index-form}
Q_{\spec}(u,u)&\vcentcolon= \int_M \Bigl(\abs{du}_g^2 - qu^2\Bigr) - \int_{\partial M} \sigma u^2.
\end{align}
Depending on the geometric setting, we will restrict $Q_\ener$ to a specific space of admissible variations $\mathcal{V}_\Phi \subset C^\infty(M, \R^{n+1})$. 
We focus on the following settings. 

\begin{definition}[Harmonic maps to the sphere]
\label{defn:closed_case}
Given a closed Riemannian manifold $(M,g)$, a map $\Phi\colon M\to\Sp^n\subset\R^{n+1}$ is called \emph{harmonic} if it is a critical point of the  energy \eqref{eqn:energy}. 
The corresponding Euler--Lagrange equation is  
\begin{align}\label{eqn:energy_ELG}
-\Delta_g\Phi-\abs{d\Phi}_g^2\Phi=0,
\end{align}
which means that applying the Laplace--Beltrami operator $\Delta_g$ component-wise to $\Phi$ results in a vector field which is normal to the sphere at each point (see e.\,g.~\cite[Example~4.14 and §\,(3.8)]{EellsLemaire1978}). 
The second variation of the energy \eqref{eqn:energy} at a critical point $\Phi$ is given by the quadratic form 
\eqref{eqn:energy-index-form}
with potentials $q \vcentcolon= \abs{d\Phi}_g^2$ and $\sigma \vcentcolon= 0$, 
restricted to 
\[
\mathcal{V}_\Phi \vcentcolon= \Gamma(\Phi^* T\Sp^n) = \{ V\in C^\infty(M, \R^{n+1}) \mid \brt{V,\Phi}=0 \text{ on } M \}.
\]
\end{definition} 

\begin{definition}[Free boundary harmonic maps to the ball]
\label{defn:free_boundary_case}
Let $(M,g)$ be a compact Riemannian manifold with boundary $\partial M$. 
A map $\Phi\colon M \to\B^{n+1}$ satisfying $\Phi(\partial M) \subset \Sp^n$ is called \emph{free boundary harmonic} if it is a critical point of the energy functional \eqref{eqn:energy} with respect to variations preserving the boundary constraint. 
Denoting the outward unit conormal along $\partial M$ by $\eta$, the associated Euler--Lagrange equation is 
\begin{align}\label{eqn:fbhm_eq}
\left\{\begin{aligned}
-\Delta_g \Phi &= 0 & &\text{ in $M$,} \\
\partial_\eta \Phi &= \abs{\partial_\eta\Phi}\Phi & &\text{ on $\partial M$.}
\end{aligned}\right.
\end{align} 
In the free boundary \emph{minimal} case, $\Phi$ is conformal and $\abs{\partial_\eta\Phi}=1$ if the metric $g$ is induced by the immersion.
The second variation of the energy \eqref{eqn:energy} at a critical point $\Phi$ is given by the quadratic form 
\eqref{eqn:energy-index-form}
with potentials $q \vcentcolon= 0$ and $\sigma \vcentcolon=\abs{\partial_\eta\Phi}$, 
restricted to the space 
\[
\mathcal{V}_\Phi \vcentcolon=  \bigl\{ V \in C^\infty(M, \R^{n+1}) 
\bigm\vert \brt{V(x), \Phi(x)} = 0 \text{ for all } x \in \partial M \bigr\}
\]
consisting of vector fields which are tangent to the boundary sphere along $\partial M$. 
\end{definition} 
 
Although we focus on the two cases above, we remark that the following setting also fits the description~\eqref{eqn:energy-index-form}--\eqref{eqn:spectral-index-form}.

\begin{definition}[Free boundary harmonic maps to spherical caps]
\label{defn:spherical_cap_case}
Let $(M,g)$ be a compact Riemannian manifold with boundary, and given $\theta\in\interval{0,\frac{\pi}{2}}$ let 
$B^n_\theta\vcentcolon=\{x\in\Sp^n\mid x_1\geq\cos\theta\}$ be a spherical cap. 
We consider free boundary harmonic maps $\Phi\colon M\to B^n_\theta$. 
Setting $\Psi=(\Phi_2,\ldots,\Phi_{n+1})$, the corresponding Euler--Lagrange equation is
\begin{align}
\left\{\begin{aligned}
-\Delta_g \Phi - \abs{d\Phi}_g^2\Phi &= 0 & &\text{ in $M$,} 
\\
-\partial_\eta \Phi_1&=\tan\theta\abs{\partial_\eta\Phi} \Phi_1 & &\text{ on $\partial M,$}
\\
\partial_\eta\Psi&=\cot\theta\abs{\partial_\eta\Phi} \Psi & &\text{ on $\partial M$.}
\end{aligned}\right.
\end{align} 
The second variation of energy at $\Phi$ is given by the quadratic form \eqref{eqn:energy-index-form}
with potentials $q \vcentcolon= \abs{d\Phi}_g^2$ and $\sigma \vcentcolon= \cot\theta\abs{\partial_\eta\Phi}$, 
restricted to the space 
\[
\mathcal{V}_\Phi \vcentcolon=  \bigl\{ V \in C^\infty(M, \R^{n+1}) 
\bigm\vert \brt{V(x), \Phi(x)} = 0 \text{ for all } x \in M,~ V_1(x) = 0\text{ for all }x\in\partial M\bigr\}
\]
consisting of vector fields tangent to $\Sp^n$ which are also tangent to the boundary $\partial B^n_\theta$ along $\partial M$. 
\end{definition}

We recall (e.\,g.~from \cite[(2-7)]{CSWSpectral}) that $0\neq V\in \mathcal{V}_\Phi$ is called an \emph{eigenfunction} of $Q_{\ener}$ with \emph{eigenvalue} $\lambda\in\R$ if for all $W\in \mathcal{V}_\Phi$ 
\begin{align}\label{eqn:eigenfunction}
Q_{\ener}(V,W)&=\lambda\brt{V,W}_{L^2(M,g)}.
\end{align}

\begin{definition}[Energy index and nullity]
\label{defn:energy_index}
Let $X^-_{\ener}$ denote the (possibly trivial) linear span, in $\mathcal{V}_\Phi$, of the eigenfunctions of $Q_\ener$ with strictly negative eigenvalue,  
and let $X^0_{\ener}$ denote the kernel of $Q_{\ener}$. 
Elements in this kernel are referred to as \emph{energy Jacobi fields} along $\Phi$. 
The \emph{energy index} and the \emph{energy nullity} of $\Phi$ are defined as  
\begin{align*}
\ind_\ener(\Phi)&\vcentcolon=\dim X^-_{\ener}, &
\nul_\ener(\Phi)&\vcentcolon=\dim X^0_{\ener}.
\end{align*} 
\end{definition}
 
\begin{definition}[Spectral index and nullity]
\label{defn:spectral_index}
Let $X^-_{\spec}$ denote the (possibly trivial) linear span, in $C^\infty(M)$, of the eigenfunctions of $Q_\spec$ with strictly negative eigenvalue, and let $X^0_{\spec}$ denote the kernel of $Q_{\spec}$. 
The \emph{spectral index} and the \emph{spectral nullity} of $\Phi$ are defined as 
\begin{align*}
\ind_\spec(\Phi)&\vcentcolon=\dim X^-_{\spec}, &
\nul_\spec(\Phi)&\vcentcolon=\dim X^0_{\spec}.
\end{align*}  
\end{definition}

\begin{remark}
\label{rem:energy-spectral-inequality}
Note that the energy form \eqref{eqn:energy-index-form} is given by the sum of the spectral forms applied to the Euclidean components of $V$. 
Specifically, if $\{e_1, \dots, e_{n+1}\}$ is the standard basis of $T_x\R^{n+1}$ at each point $x$ and we write the variation as $V = \sum_{i=1}^{n+1} V^i e_i$, then
\begin{align*}
Q_{\ener}(V,V)&=\sum_{i=1}^{n+1} Q_{\spec}(V^i,V^i).
\end{align*}
This decomposition allows for a direct comparison between the indices of the two problems. Indeed, 
\cite[Proposition~1.7]{Karpukhin2021} states that if $\Phi\colon M\to\Sp^n$ is a harmonic map whose image is not contained in any proper equatorial subsphere of $\Sp^n$, 
then 
\[
\ind_\ener(\Phi)\leq(n+1)\ind_\spec(\Phi).
\]
This is not an equality in general because the (vector-valued) energy index problem is subject to the constraint $\brt{V,\Phi} = 0$, while the (scalar) spectral index problem is unconstrained.
\end{remark}

\begin{remark}[Spectral index and nullity]
\label{rem:spectral_nullity}
Any harmonic map $\Phi=(\Phi_1,\ldots,\Phi_{n+1})\colon M\to\Sp^n$ 
on a closed Riemannian manifold $(M,g)$ satisfies $Q_{\spec}(\Phi_j,\cdot)\equiv0$ for any $j\in\{1,\ldots,n+1\}$. 
This follows by integrating each component of equation \eqref{eqn:energy_ELG} by parts. 
In particular, if $\Phi$ is linearly full in the sense that the linear span of its image is all of $\R^{n+1}$, then
\[
\nul_\spec(\Phi)\geq n+1.
\]
In the special case where $\Phi\colon M\to\Sp^n$ is an isometric immersion in the sense that $g=\Phi^*g_{\R^{n+1}}$ is the induced metric on $M$, the map $\Phi$ is harmonic if and only if it is a minimal immersion and the spectral quadratic form \eqref{eqn:spectral-index-form} reads
\[
Q_\spec(u,u)=\int_M\bigl(\abs{du}_g^2 - (\dim M)u^2\bigr).
\]
In particular, the components of $\Phi$ are eigenfunctions of the Laplace--Beltrami operator $-\Delta_g$ with eigenvalue $\dim M$. 
Moreover, $\nul_\spec(\Phi)$ is the multiplicity of that eigenvalue and $\ind_\spec(\Phi)$ is the number of eigenvalues $\lambda_i(M,g)$ satisfying $\lambda_i(M,g)<\dim M$.

Similarly, any free boundary harmonic map $\Phi=(\Phi_1,\ldots,\Phi_{n+1})\colon M\to\B^{n+1}$ on a compact Riemannian manifold $(M,g)$ with boundary satisfies $Q_{\spec}(\Phi_j,\cdot)=0$ for any $j\in\{1,\ldots,n+1\}$. This follows by integrating by parts and substituting the boundary condition from equation \eqref{eqn:fbhm_eq}. 
Again, this implies $\nul_\spec(\Phi)\geq n+1$ if $\Phi$ is linearly full.
In the special case where $g=\Phi^*g_{\R^{n+1}}$, the components of $\Phi$ are Steklov eigenfunctions with eigenvalue $1$. 
Moreover, $\nul_\spec(\Phi)$ is the multiplicity of that eigenvalue and $\ind_\spec(\Phi)$ is the number of eigenvalues $\sigma_i(M,g)$ satisfying $\sigma_i(M,g)<1$.
\end{remark}

\subsection{Group actions and isotypic decomposition}
\label{sec:indE_equiv}

Let $G$ be a finite subgroup of the diffeomorphism group $\Diff(M)$ acting on $(M,g)$ by isometries in the sense that $\gamma^*g=g$ for all $\gamma\in G$. In particular, this action naturally preserves the boundary $\partial M$ as a set. 
The set of points fixed by a single group element $\gamma\in G$ is denoted by 
\begin{align*}
\Fix(\gamma)\vcentcolon=\{x\in M\mid\gamma x=x\}.
\end{align*}

\begin{definition}[Equivariance]
\label{defn:equivariance}
Let $\rho\colon G\to\Ogroup(n+1)$ be an orthogonal representation on the ambient Euclidean space $\R^{n+1}$. 
A map $\Phi\colon M\to\R^{n+1}$ is called \emph{$G$-equivariant} in the representation $\rho$ if for all $\gamma\in G$ and $x \in M$,
\[
\Phi(\gamma x) = \rho(\gamma) \Phi(x).
\]
If a submanifold $\Sigma\subset\R^{n+1}$ satisfies $\rho(\gamma)\Sigma=\Sigma$ as a set for all $\gamma\in G$, then $G$ acts on $\Sigma$ by restriction and the inclusion map is $G$-equivariant in the representation $\rho$. 
In this situation we also call $\Sigma\subset\R^{n+1}$ itself $G$-equivariant in the representation $\rho$.
\end{definition}

The action of $G$ on $M$ and its representation $\rho$ on $\R^{n+1}$ induce natural representations on the relevant function spaces. On smooth scalar functions $u \in C^\infty(M)$, we denote by $\rho_\spec$ the pullback representation given by 
\begin{align}
\label{eqn:rho_S}
(\rho_\spec(\gamma) u)(x) &\vcentcolon= u(\gamma^{-1} x) 
\intertext{for each $\gamma\in G$. 
Analogously, for vector-valued functions $V \in C^\infty(M, \R^{n+1})$, we define}
\label{eqn:rho_E}
(\rho_\ener(\gamma) V)(x) &\vcentcolon= \rho(\gamma) V(\gamma^{-1} x).
\end{align}
The following statements ensure that these representations preserve the variational structure of the problem, implying that in this setting, $X^-_{\ener}$ and $X^0_{\ener}$ are $G$-invariant subspaces. 

\begin{lemma}\label{lem:invariant_potentials}
Let $\Phi$ be as in Definition~\ref{defn:closed_case} respectively Definition~\ref{defn:free_boundary_case}. 
If $\Phi$ is $G$-equivariant then the potentials $q$ and $\sigma$ are $G$-invariant functions.  
\end{lemma}

\begin{proof}
In the setting of Definition~\ref{defn:closed_case}, $\sigma = 0$ is trivially invariant, and $q=\abs{d\Phi}_g^2$ is invariant because $\Phi$ is $G$-equivariant, $\rho(\gamma)$ is an orthogonal map, and $G$ acts by isometries.

In the setting of Definition~\ref{defn:free_boundary_case}, $q = 0$ is trivially invariant, and $\sigma = \abs{\partial_\eta\Phi}$ is invariant by the same argument as above, using in addition that the outward unit conormal $\eta$ is $G$-equivariant along $\partial M$.
\end{proof}

\begin{lemma}\label{lem:preserved}
The maps \eqref{eqn:rho_S} and \eqref{eqn:rho_E} define linear representations $\rho_\spec\colon G\to\End(C^\infty(M))$ and $\rho_\ener\colon G\to\End(C^\infty(M,\R^{n+1}))$ which act by $L^2(M,g)$-isometries. 
If $\Phi$ is $G$-equivariant as in Lemma~\ref{lem:invariant_potentials}, then $\rho_\ener(\gamma)$ preserves the subspace $\mathcal{V}_\Phi$, and $\rho_\spec(\gamma)$ and $\rho_\ener(\gamma)$ preserve the quadratic forms $Q_\spec$ and $Q_\ener$, respectively, for all $\gamma\in G$. 
\end{lemma}

\begin{proof}
The linearity of the maps is immediate. 
For any $\gamma_1,\gamma_2\in G$, any $u\in C^\infty(M)$, any $V\in C^\infty(M,\R^{n+1})$ and any $x\in M$, we have 
\begin{align*}
(\rho_\spec(\gamma_1\gamma_2)u)(x) &= u(\gamma_2^{-1}\gamma_1^{-1}x) = (\rho_\spec(\gamma_1)\rho_\spec(\gamma_2)u)(x), 
\\[1ex]
(\rho_\ener(\gamma_1\gamma_2)V)(x) &= \rho(\gamma_1)\rho(\gamma_2)V(\gamma_2^{-1}\gamma_1^{-1}x) = (\rho_\ener(\gamma_1)\rho_\ener(\gamma_2)V)(x).
\end{align*}
Since $G$ acts by isometries of $(M,g)$ and $\rho(\gamma)\in\Ogroup(n+1)$ for all $\gamma\in G$, a change of variable $x\mapsto\gamma^{-1}x$ shows that $\rho_\spec(\gamma)$ and $\rho_\ener(\gamma)$ preserve the respective $L^2(M,g)$-inner products. 

Suppose that $\Phi$ is $G$-equivariant and let $\gamma\in G$ be arbitrary. 
Then, for all $V\in C^\infty(M,\R^{n+1})$ and $x\in M$,
\begin{align}\label{eqn:20260925}
\brt[\big]{(\rho_\ener(\gamma)V)(x),\Phi(x)}
=\brt[\big]{V(\gamma^{-1} x), \rho(\gamma)^{-1}\Phi(x) } 
=\brt[\big]{V(\gamma^{-1} x), \Phi(\gamma^{-1} x)}
\end{align}
since $\rho(\gamma)\in\Ogroup(n+1)$. 
If $V\in\mathcal{V}_\Phi$, then \eqref{eqn:20260925} vanishes for all $x\in M$ in the setting of Definition~\ref{defn:closed_case}, and for all $x\in \partial M$ in the setting of Definition~\ref{defn:free_boundary_case}, since $G$ preserves $\partial M$. 
This proves that $\rho_\ener(\gamma)$ preserves the subspace $\mathcal{V}_\Phi$. 

Furthermore, since $\gamma^{-1}\in G$ acts as an isometry of $(M,g)$, the differentials satisfy
\begin{align*}
\abs[\big]{d\bigl(\rho_\spec(\gamma)u\bigr)}^2_g(x) &= \abs{du}^2_g(\gamma^{-1}x), 
&
\abs[\big]{d\bigl(\rho_\ener(\gamma)V\bigr)}^2_g(x) &= \abs{dV}^2_g(\gamma^{-1}x).
\end{align*}
Since the potentials $q$ and $\sigma$ are $G$-invariant by Lemma~\ref{lem:invariant_potentials}, we have $q(x) = q(\gamma^{-1}x)$ and $\sigma(x) = \sigma(\gamma^{-1}x)$. 
Combined with the invariance of the volume forms on $M$ and $\partial M$, a change of variable $x\mapsto \gamma^{-1}x$ yields  
\begin{align*}
Q_{\spec}\bigl(\rho_\spec(\gamma)u,\rho_\spec(\gamma)u\bigr) &= Q_{\spec}(u,u), 
&
Q_{\ener}\bigl(\rho_\ener(\gamma)V,\rho_\ener(\gamma)V\bigr) &= Q_{\ener}(V,V) 
\end{align*}
for all $u\in C^\infty(M)$ and $V\in\mathcal{V}_\Phi$. 
\end{proof}

Since the quadratic form $Q_{\ener}$ is preserved by the representation $\rho_\ener$ by Lemma \ref{lem:preserved}, the corresponding eigenspaces are $G$-invariant subspaces. 
Hence, each of them decomposes into a direct sum of $G$-invariant subspaces on which the induced representation of $G$ is irreducible \cite[Theorem~2]{Serre1977}. 

For the following definitions we recall that two representations, $\rho_1\colon G\to\End(X_1)$, and $\rho_2\colon G\to\End(X_2)$ are called \emph{isomorphic} if there exists a linear isomorphism $X_1\to X_2$ that commutes with the group action.

\begin{definition}[Isotypic index and nullity]
\label{defn:Isotypic}
Let $\IrRep{G}$ denote the set of isomorphism classes of irreducible representations of $G$ and let $r\in\IrRep{G}$ be arbitrary. 
\begin{itemize} 
\item 
The \emph{$r$-isotypic component} $\mathcal{V}_\Phi^r$ is defined as the sum of all $G$-invariant subspaces $W \subseteq \mathcal{V}_\Phi$ such that the subrepresentation on $W$ belongs to the isomorphism class $r$.
\item 
The \emph{$r$-isotypic energy index}, denoted $\ind_\ener^r(\Phi)$, is defined as 
the maximal dimension of a subspace of the $r$-isotypic component $\mathcal{V}_\Phi^r$ on which $Q_{\ener}$ is negative definite: 
\[
\ind_\ener^r(\Phi)=\dim\Bigl(X_\ener^-\cap\mathcal{V}_\Phi^r\Bigr)
\]
\item 
The \emph{$r$-isotypic energy nullity}, denoted $\nul_\ener^r(\Phi)$, is defined as 
the dimension of the kernel of the restriction of $Q_{\ener}$ to the $r$-isotypic component:
\[
\nul_\ener^r(\Phi)=\dim\Bigl(X_\ener^0\cap\mathcal{V}_\Phi^r\Bigr)
\]
\item $\Ind_\ener(\Phi)$ denotes the isomorphism class of the representation on $X_\ener^-$. 
\item $\Nul_\ener(\Phi)$ denotes the isomorphism class of the representation on $X_\ener^0$. 
\item
Similarly, we decompose the space of smooth scalar functions $C^\infty(M)$ into isotypic components $C^\infty(M)^r$ and define the $r$-isotypic spectral index $\ind_\spec^r(\Phi)$ 
and nullity $\nul_\spec^r(\Phi)$, 
and the spectral representation classes $\Ind_\spec(\Phi)$ and $\Nul_\spec(\Phi)$. 
\end{itemize}
\end{definition}
 
\begin{remark}[Isomorphism classes of representations]
\label{rem:classes_vs_spaces}
If $r\in\IrRep{G}$ occurs with multiplicity $m$ in $\Ind_\ener(\Phi)$, then $\ind_\ener^r(\Phi)=m\deg(r)$, and similarly for the nullity. 
Since $\mathcal{V}_\Phi$ is the direct sum of its isotypic components, the index and nullity of $\Phi$ are the sums 
\begin{align*}
\ind_\ener(\Phi)&=\sum_{r\in\IrRep{G}}\ind_\ener^r(\Phi), &
\nul_\ener(\Phi)&=\sum_{r\in\IrRep{G}}\nul_\ener^r(\Phi)
\end{align*}
of their isotypic counterparts (cf.~\cite[(2.14)]{CSWSpectral}). 
Summing over $\IrRep{G}$ requires its elements to be isomorphism classes.
We define $\Ind_\ener(\Phi)$ and its relatives as isomorphism classes to state bounds which apply to any surface with the given symmetry.
Between isomorphism classes we use the symbols $=$ and $\leq$ (in the sense of Definition~\ref{defn:representation_inequalities} below), while $\cong$ indicates isomorphisms between representations on concrete vector spaces. 
Relations between isomorphism classes are typically obtained by constructing equivariant maps between explicit spaces.
\end{remark}

\begin{definition}
\label{defn:representation_inequalities}
Given two representations $\rho_1\colon G\to \End(X_1)$ and $\rho_2\colon G\to\End(X_2)$ with finite degree, we write $\rho_1\leq\rho_2$ if there exists an injective linear map $X_1 \hookrightarrow X_2$ which commutes with the group action, i.\,e.~if $\rho_1$ can be realised as a subrepresentation of $\rho_2$.  
This notation extends naturally to isomorphism classes of representations.  
\end{definition}

\begin{remark}
For representations of finite groups with finite degree, complete reducibility ensures that $\rho_1 \leq \rho_2$ and $\rho_2 \leq \rho_1$ hold simultaneously if and only if $\rho_1 \cong \rho_2$. 
Indeed, the relation $\rho_1 \leq \rho_2$ is equivalent to the multiplicity of each irreducible factor in $\rho_1$ being bounded above by its corresponding multiplicity in $\rho_2$.
\end{remark}

The space $\Hom(X_1, X_2)$ of linear maps has a natural representation of $G$ with the action given by
\begin{align}
\label{eqn:action_on_Hom}
(\gamma f)(x)&\vcentcolon= \rho_2(\gamma) \Bigl( f\bigl(\rho_1(\gamma^{-1})x\bigr) \Bigr), &
&
\gamma\in G,\quad 
f \in \Hom(X_1, X_2),\quad  
x \in X_1.
\end{align}
The isomorphism class of this representation depends only on the isomorphism classes of $\rho_1$ and $\rho_2$. 
Indeed, given equivariant isomorphisms $T_i\colon X_i \to X_i'$ for $i\in\{1,2\}$, the map $f \mapsto T_2 \circ f \circ T_1^{-1}$ defines an equivariant isomorphism from $\Hom(X_1, X_2)$ to $\Hom(X'_1, X'_2)$.
Thus, $\Hom$ becomes a well-defined operation on isomorphism classes of representations.

\begin{proposition}
\label{prop:index-bound}
Let $\Phi$ be a $G$-equivariant harmonic map as in Definition~\ref{defn:closed_case} respectively Definition~\ref{defn:free_boundary_case}, and let $\cA_\rho$ denote the isomorphism class of the representation $\rho\colon G\to\Ogroup(n+1)$ on the ambient space $\R^{n+1}$. 
Then we have the following inequality of representation classes: 
\[
\Ind_\ener(\Phi)\leq\Hom\bigl(\cA_\rho, \Ind_\spec(\Phi)\bigr).
\] 
\end{proposition}

\begin{remark}
Recalling the notation from Definitions~\ref{defn:energy_index}--\ref{defn:spectral_index}, Proposition~\ref{prop:index-bound} states that the representation of $G$ on $X_\ener^-$ is isomorphic to a subrepresentation of the representation \eqref{eqn:action_on_Hom} on $\Hom(\R^{n+1},X_\spec^-)$. 
Comparing degrees, we recover the inequality $\ind_\ener(\Phi)\leq(n+1)\ind_\spec(\Phi)$ mentioned in Remark~\ref{rem:energy-spectral-inequality}.
\end{remark} 

\begin{proof}
Let $P \colon C^\infty(M) \to X_\spec^-$ denote the $L^2(M,g)$-orthogonal projection onto the spectral negative subspace (see Definition \ref{defn:spectral_index}). 
Note that $P$ commutes with the group action since $\rho_\spec$ acts by $L^2(M,g)$-isometries and preserves $X_\spec^-$ by Lemma~\ref{lem:preserved}.
Consider the linear map 
\begin{align}\label{eqn:injectionPsi}
\Psi\colon X_\ener^-\subset\mathcal{V}_\Phi \to \Hom(\R^{n+1},X_\spec^-)
\end{align}
defined by assigning to each section $V\in X_\ener^-$ the map $\Psi(V)$ given by 
\(
\Psi(V)(x)\vcentcolon=P \brt{V, x} 
\)
for each $x\in \R^{n+1}$. 
As before, $\brt{\cdot, \cdot}$ denotes the Euclidean inner product in $\R^{n+1}$. 

\emph{Claim 1.} $\Psi$ is injective.
If $\Psi(V) = 0$ then, for every $x\in\R^{n+1}$, the scalar function $u_x \vcentcolon= \brt{V, x}$ lies in the kernel of $P$, which is the $L^2(M,g)$-orthogonal complement of $X_\spec^-$. 
Consequently, $Q_\spec(u_x, u_x)\geq 0$ for all $x\in \R^{n+1}$.
Summing over an orthonormal basis $e_1,\ldots,e_{n+1}$ of $\R^{n+1}$, we obtain (cf.~Remark~\ref{rem:energy-spectral-inequality})
\[
Q_{\ener}(V,V) = \sum_{i=1}^{n+1} Q_{\spec}(\brt{V, e_i}, \brt{V, e_i}) \ge 0.
\]
However, if $V \in X_\ener^-$, we must have $Q_{\ener}(V,V) < 0$ unless $V=0$. 
Thus, the kernel of $\Psi$ is trivial. 

\emph{Claim 2.} $\Psi$ is $G$-equivariant with respect to the representation $\rho_\ener$ from Lemma~\ref{lem:preserved} and the induced representation on $\Hom(\R^{n+1}, X_\spec^-)$. 
Indeed, for any $\gamma\in G$ and $x\in \R^{n+1}$, we have 
\begin{align*}
\bigl(\gamma\Psi(V)\bigr)(x)
&=\rho_\spec(\gamma)\Bigl(\Psi(V)\bigl(\rho(\gamma^{-1})x\bigr)\Bigr)
\tag{by definition \eqref{eqn:action_on_Hom}}
\\
&=\rho_\spec(\gamma)\Bigl(P\brt[\big]{V,\rho(\gamma^{-1})x}\Bigr)
\tag{by definition of $\Psi$}
\\
&=P\brt[\big]{V\circ\gamma^{-1},\rho(\gamma^{-1})x} 
\tag{by definition \eqref{eqn:rho_S} of $\rho_\spec$}
\\
&=P\brt[\big]{\rho(\gamma)V\circ\gamma^{-1},x} 
\tag{by orthogonality of $\rho$}
\\
&=P\brt[\big]{\rho_\ener(\gamma)V,x} 
\tag{by definition \eqref{eqn:rho_E} of $\rho_\ener$}
\\
&=\Psi\bigl(\rho_\ener(\gamma)V\bigr)(x).
\end{align*} 
By Claims~1 and~2, the map \eqref{eqn:injectionPsi} is a $G$-equivariant injection, which concludes the proof.
\end{proof}

\begin{corollary}\label{cor:indE}
Let $\Phi$ be a $G$-equivariant harmonic map as in Definition~\ref{defn:closed_case} respectively Definition~\ref{defn:free_boundary_case}.  
If $\ind_\spec(\Phi)=1$, then $\Ind_\ener(\Phi)\leq \cA_\rho$.
\end{corollary}

\begin{proof}
Since $\ind_\spec(\Phi)=1$, the lowest eigenvalue of $Q_\spec$ is negative and simple, and $X_\spec^-$ is spanned by a corresponding first eigenfunction $\varphi$, which we may assume to be nonnegative.  
The $G$-invariance of $X_\spec^-$ then implies $\rho_\spec(\gamma)\varphi=\pm\varphi$ for every $\gamma\in G$, and the sign is positive because $\rho_\spec(\gamma)\varphi=\varphi\circ\gamma^{-1}\geq0$ by definition \eqref{eqn:rho_S}. 
Hence $\Ind_\spec(\Phi)$ is the isomorphism class of the trivial representation, and Proposition~\ref{prop:index-bound} implies 
\[
\Ind_\ener(\Phi)
\leq\Hom(\cA_\rho,\Ind_\spec(\Phi)) 
=\cA_\rho^* 
=\cA_\rho,
\]
where we used the fact that $\rho$ is an orthogonal representation to identify $\cA_\rho$ with its dual.
\end{proof}

In the following we specialise the setting of Definition~\ref{defn:closed_case} to minimal immersions and build upon the fact that the projections of constant vectors in $\R^{n+1}$ onto the tangent bundle of the sphere provide a natural $(n+1)$-dimensional space of unstable directions for minimal surfaces in $\Sp^n$, see e.\,g.~\cite[§\,5.1]{Simons1968}. 

\begin{proposition}
\label{prop:Arho_lb}
Let $\Phi\colon M\to\Sp^n$ be a $G$-equivariant harmonic map as in Definition~\ref{defn:closed_case} which is a conformal immersion whose image is not a totally geodesic subsphere of $\Sp^n$. 
Then $\cA_\rho\leq\Ind_\ener(\Phi)$.
In particular, if $\ind_\spec(\Phi)=1$, then $\Ind_\ener(\Phi)=\cA_\rho$.
\end{proposition}

\begin{proof}
From Definitions~\ref{defn:closed_case} and~\ref{defn:energy_index} we recall that $Q_\ener$ is defined on the space $\mathcal{V}_\Phi=\Gamma(\Phi^*T\Sp^n)$.
Given any vector $\xi\in\R^{n+1}$ we consider the variation field $V_\xi\vcentcolon=\xi-\brt{\xi,\Phi}\Phi\in\mathcal{V}_\Phi$. 
Then $V_\xi\neq0$ whenever $\xi\neq0$ because $\Phi$ is not constant. 
Since $\Phi$ is harmonic and conformal (hence minimal), $Q_\ener$ is negative definite on the $(n+1)$-dimensional space $\{V_\xi\mid\xi\in\R^{n+1}\}$ by \cite[Theorem~2.5\,(ii) and Remark~2.6]{ElSoufi1993}.  
We claim that the linear map $j\colon\R^{n+1}\to\mathcal{V}_\Phi$ given by $j(\xi)\vcentcolon=V_\xi$ is $G$-equivariant. 
Indeed, for any $\gamma \in G$ and $x \in M$, 
\begin{align*}
\bigl(\rho_\ener(\gamma)V_\xi\bigr)(x) 
&=\rho(\gamma) V_\xi(\gamma^{-1}x) 
\\
&=\rho(\gamma) \Bigl( \xi - \brt[\big]{\xi, \Phi(\gamma^{-1}x)}\Phi(\gamma^{-1}x) \Bigr) 
\\
&= \rho(\gamma)\xi - \brt[\big]{\rho(\gamma)\xi,\rho(\gamma)\Phi(\gamma^{-1}x)}
\rho(\gamma)\Phi(\gamma^{-1}x)
\tag{since $\rho(\gamma) \in \Ogroup(n+1)$} 
\\[.5ex]
&= \rho(\gamma)\xi - \brt[\big]{\rho(\gamma)\xi, \Phi(x)} \Phi(x) \tag{by equivariance of $\Phi$} 
\\[.5ex]
&=V_{\rho(\gamma)\xi}(x).
\end{align*}
Let $P^-$ denote the $L^2(M,g)$-orthogonal projection of $\mathcal{V}_\Phi$ onto $X_\ener^-$.
By Lemma~\ref{lem:preserved} the representation $\rho_\ener$ defined in \eqref{eqn:rho_E} acts by $L^2(M,g)$-isometries and preserves $X_\ener^-$. 
Hence, the projection $P^-$ commutes with $\rho_\ener$.
Since moreover $Q_\ener\geq0$ on the $L^2(M,g)$-orthogonal complement of $X_\ener^-$ in $\mathcal{V}_\Phi$, the map $P^-\circ j\colon\R^{n+1}\to X_\ener^-$ is injective.
Thus, $P^-\circ j$ is a $G$-equivariant injection and $\cA_\rho\leq\Ind_\ener(\Phi)$. 
Finally, if $\ind_\spec(\Phi) = 1$, Corollary~\ref{cor:indE} implies $\Ind_\ener(\Phi) \leq \cA_\rho$, and we conclude $\Ind_\ener(\Phi)=\cA_\rho$.
\end{proof}

\begin{remark} 
\label{rem:fb_energy_unknown}  
We are not aware of an analogue of Proposition~\ref{prop:Arho_lb} for free boundary minimal surfaces in the unit ball $\B^{n+1}$. 
However, Propositions~\ref{prop:knoid_isotypic} and~\ref{prop:torus_isotypic} show that the free boundary minimal surfaces in $\B^3$ considered in Theorems~\ref{thm:index-fbm-k-noid} and~\ref{thm:index-fbm-torus} have energy index equal to $3$.
\end{remark}

\begin{definition}
\label{defn:quadratic_independence}
The components of a map $\Phi=(\Phi_1,\ldots,\Phi_{n+1})\colon M\to\R^{n+1}$ are called \emph{quadratically independent} if the set of product functions $\{\Phi_i \Phi_j \}_{1 \leq i \leq j \leq n+1}$ is linearly independent over $\R$.
Equivalently, the image $\Phi(M)$ is not contained in any non-trivial quadratic cone in the sense that 
if $\brt{B\Phi(x), \Phi(x)} = 0$ for all $x \in M$ and some symmetric matrix $B\in\R^{(n+1)\times(n+1)}$, then  $B=0$.  
\end{definition}

\begin{remark}
\label{rem:quadratic_linear}
Quadratic independence implies linear fullness. 
Indeed, if $\Phi\colon M\to\R^{n+1}$ is not linearly full, there exists $0\neq v \in \R^{n+1}$ such that 
\[
0 = \brt{v, \Phi(x)}^2 = (\Phi(x)^\intercal v)(v^\intercal \Phi(x)) = \brt{v v^\intercal \Phi(x), \Phi(x)} 
\]
for all $x\in M$. Since $B= v v^\intercal\ne0$ is a symmetric matrix, $\Phi$ is not quadratically independent. 
\end{remark}

\begin{definition}[Conjugation representation]
\label{defn:so_A_rho}
The representation $\rho\colon G\to\Ogroup(n+1)$ induces an action of $G$ on the space $\mathfrak{so}(n+1)$ of skew-symmetric matrices by conjugation $B\mapsto\rho(\gamma)B\rho(\gamma)^{-1}$.
We denote the isomorphism class of this representation by $\mathfrak{so}(\cA_\rho)$.
\end{definition}

\begin{proposition}[Energy nullity]
\label{prop:rigidity}
Let $\Phi=(\Phi_1,\ldots,\Phi_{n+1})\colon(M,g)\to\R^{n+1}$ be a $G$-equivariant harmonic map as in Definitions~\ref{defn:closed_case}--\ref{defn:free_boundary_case}. 
Suppose that the components of $\Phi$ are quadratically independent on $M$ in the closed case, respectively on $\partial M$ in the free boundary case. 
If 
\begin{align*}
\ind_\spec(\Phi)&= 1, & 
\nul_\spec(\Phi)&= n+1, & 
\ind_\ener(\Phi)&= n+1, 
\end{align*} 
then $\Nul_\ener(\Phi)=\mathfrak{so}(\cA_\rho)$ and in particular 
\(
\nul_\ener(\Phi) = \tfrac{1}{2}n(n+1).
\) 
\end{proposition}

\begin{proof}
Let $\mathcal{K} \subset \mathcal{V}_\Phi$ denote the subspace of Killing fields induced by the action of $\SOgroup(n+1)$ on the target space. 
These are vector fields of the form $V_B(x) \vcentcolon= B\Phi(x)$ for $B \in \mathfrak{so}(n+1)$. 
Because $B$ is skew-symmetric, $\brt{V_B(x), \Phi(x)} = \brt{B\Phi(x), \Phi(x)} = 0$ for all $x \in M$, so $V_B \in \mathcal{V}_\Phi$ in both settings. 
Since target rotations preserve the Dirichlet energy and the boundary constraint, $\mathcal{K} \subseteq X_\ener^0$. 
Furthermore, since $\Phi$ is linearly full by Remark~\ref{rem:quadratic_linear}, the map $B\mapsto V_B$ is injective, so 
\[
\nul_\ener(\Phi)\geq\dim\mathcal{K}=\tfrac{1}{2}n(n+1). 
\]
It remains to show the reverse inequality which, by assumption $\ind_\ener(\Phi) = n+1$, is equivalent to the claim 
\begin{align}\label{eqn:20260211-1}
\dim(X_\ener^- \oplus X_\ener^0) \leq (n+1) + \tfrac{1}{2}n(n+1).
\end{align}
Since $\ind_\spec(\Phi)=1$, the space $X_\spec^-$ is one-dimensional, spanned by some eigenfunction $\varphi_0$.
Consider the subspace 
\[
\mathcal{V} \vcentcolon= \{ V \in X_\ener^- \oplus X_\ener^0 \mid \brt{V^j, \varphi_0}_{L^2(M,g)} = 0 \text{ for all } j=1,\dots,n+1\} 
\]
of sections whose scalar components are $L^2(M,g)$-orthogonal to $\varphi_0$. 
Then, as $\mathcal{V}$ is defined as the kernel of a linear map $X_\ener^- \oplus X_\ener^0\to\R^{n+1}$, 
\begin{align}\label{eqn:20260211-2}
\dim (X_\ener^- \oplus X_\ener^0)\leq (n+1) + 
\dim \mathcal{V}.
\end{align}
On the one hand, $Q_\ener(V,V) \leq 0$ for any $V\in\mathcal{V}$ because $\mathcal{V}\subset X_\ener^- \oplus X_\ener^0$. 
On the other hand, each component $V^j$ of any $V \in \mathcal{V}$ is orthogonal to $X_\spec^-$, which implies $Q_\spec(V^j, V^j) \geq 0$ and hence $Q_\ener(V,V) = \sum^{n+1}_{j=1} Q_\spec(V^j, V^j) \geq 0$.
Consequently, $Q_\ener(V,V) = 0$ and $Q_\spec(V^j, V^j) = 0$ for all $j\in\{1,\ldots,n+1\}$, so each $V^j$ lies in the kernel of $Q_\spec$ since $Q_\spec\geq0$ on the $L^2(M,g)$-orthogonal complement of $\varphi_0$.
 
By assumption, the spectral kernel has dimension $\nul_\spec(\Phi)=n+1$ and is thus spanned by the coordinate functions $\Phi_1, \dots, \Phi_{n+1}$ (see Remark~\ref{rem:spectral_nullity}). 
Therefore, any $V \in \mathcal{V}$ must be of the form $V(x) = B\Phi(x)$ for some constant matrix $B$. 
The condition $V \in \mathcal{V}_\Phi$ implies $\brt{B\Phi(x), \Phi(x)} = 0$ for all $x \in M$ in the closed setting, respectively for all $x \in \partial M$ in the free boundary setting. 
Since the components of $\Phi$ are quadratically independent on $M$ (respectively on $\partial M$), the symmetric part of $B$ must vanish. 
Thus $B \in \mathfrak{so}(n+1)$, which implies $\mathcal{V} \subseteq \mathcal{K}$. 
We conclude that $\dim \mathcal{V}\leq\dim\mathcal{K}=n(n+1)/2$, which proves \eqref{eqn:20260211-1} using \eqref{eqn:20260211-2}.  

Since $\dim X_\ener^-=\ind_\ener(\Phi)=n+1$, inequality \eqref{eqn:20260211-1} yields $\dim X_\ener^0\leq\dim\mathcal{K}$, and hence $X_\ener^0=\mathcal{K}$. 
Since $\rho_\ener(\gamma)V_B=V_{\rho(\gamma)B\rho(\gamma)^{-1}}$ for all $\gamma\in G$ by equivariance of $\Phi$, the map $B\mapsto V_B$ is a $G$-equivariant isomorphism from $\mathfrak{so}(\cA_\rho)$ onto $X_\ener^0$.
\end{proof}

\section{Index and nullity of equivariant minimal surfaces}
\label{sec:area} 

In this section, we assume that $(M,g)$ is a compact, connected, $2$-dimensional Riemannian manifold, possibly with boundary $\partial M$, 
and consider maps $\Phi\colon M\to\R^{n+1}$ for any $2\leq n\in\N$. 
The target will be restricted to $\Sp^n$ as in Definition~\ref{defn:closed_case} if $M$ is closed, respectively to $\B^{n+1}$ as in Definition~\ref{defn:free_boundary_case} if $\partial M\neq\emptyset$.   
In the closed case, Sections~\ref{sub:teich}--\ref{sub:area} follow \cite{Moore2007}, and we extend the theory to the free boundary case. 
Sections~\ref{sub:equivariant-area}--\ref{sub:two-sided} then adapt this material to the $G$-equivariant setting.

If $\Phi$ is the inclusion map of $M$ as a minimal submanifold of the target, we relax the notation by writing $\ind_\bullet(M)=\ind_\bullet(\Phi)$ and $\nul_\bullet(M)=\nul_\bullet(\Phi)$ for each type $\bullet$ of index and nullity.
 
Given two $\R^{n+1}$-valued $1$-forms $\alpha,\beta$ on $M$, we understand the tensor product $\alpha\otimes\beta$  
as the real-valued $(0,2)$-tensor defined on tangent vector fields $X, Y \in \Gamma(TM)$ by 
\begin{align*}
(\alpha \otimes \beta)(X,Y)&\vcentcolon=\brt[\big]{\alpha(X), \beta(Y)},  
\end{align*}
using the standard Euclidean inner product $\brt{\cdot,\cdot}$ in $\R^{n+1}$.

\subsection{Two-variable energy index}
\label{sub:teich}

The relation between the area index and the energy index is most conveniently explained using the two-variable energy functional 
\[
\widetilde{E}(\Phi,g)\vcentcolon=E_g(\Phi)
\]
defined as in \eqref{eqn:energy} but considered as a function of both the Riemannian metric $g$ and the map $\Phi$ rather than keeping the metric fixed \cite[equation~(8)]{Moore2007}. 
An elementary computation shows that the first variation of $\widetilde{E}$ at a pair $(\Phi,g)$ in the direction of a variation vector field $V \in \mathcal{V}_\Phi$ (recalling Definitions~\ref{defn:closed_case}--\ref{defn:free_boundary_case}) and a metric variation $h \in \mathcal{S}(M)$ is given by
\[
d\widetilde{E}(V,h) 
=dE_g(V)-\frac{1}{2}\int_{(M,g)} 
\brt[\Big]{ d\Phi\otimes d\Phi - \frac{1}{2}\abs{d\Phi}_g^2\, g, h }_g.
\] 
The quadratic form in the inner product is known as the \emph{Hopf differential} of the map $\Phi$. 
In particular, the critical points of $\widetilde{E}$ are pairs $(\Phi, g)$ where $\Phi$ is harmonic as well as weakly conformal with respect to $g$, or, equivalently, a branched minimal immersion.

The second variation of $\widetilde{E}$ is a quadratic form on the product space $\mathcal{V}_\Phi\times \mathcal{S}(M)$, where $\mathcal{S}(M)$ denotes the space of symmetric $(0,2)$-tensors on $M$. 
We denote this quadratic form by $Q_\teich$. 
(The subscript $\teich$ stands for ``Two-variable''.)
Following the computation in~\cite[equations (9) and (16)]{Moore2007}, for any variation $(V,h)$, we have
\begin{align}
\label{eq:QtildeE}
\begin{aligned}
Q_{\teich}\bigl((V,h), (V,h)\bigr) 
= Q_{\ener}(V,V) 
&- 2\int_{(M,g)} \brt[\Big]{ d\Phi\otimes dV - \frac{1}{2}\brt{d\Phi, dV}g, h }_g  
\\
&+\frac{1}{2}\int_{(M,g)} \abs{d\Phi}_g^2 \Bigl( \frac{1}{4}(\tr_g h)^2 - \det\nolimits_g(h) \Bigr). 
\end{aligned}
\end{align}
Compared to the single-variable energy, the two-variable energy has a large symmetry group. 
First, $\widetilde{E}$ is conformally invariant in the sense that for any $u\in C^\infty(M)$
\begin{align}
\label{eqn:conformal_invariance}
\widetilde{E}(\Phi,g)&= \widetilde{E}(\Phi,e^{2u}g).
\shortintertext{Second, $\widetilde{E}$ is invariant under the action of the diffeomorphism group $\Diff(M)$ in the sense that for any $f\in \Diff(M)$}
\label{eqn:diffeo_invariance}
\widetilde{E}(f^*\Phi, f^*g)&= \widetilde{E}(\Phi, g).
\end{align}
These symmetries lead to an infinite-dimensional kernel for $Q_{\teich}$. 
Consequently, establishing a well-defined notion of index and nullity for $\widetilde{E}$ requires restricting the domain of the quadratic form to a slice transversal to these symmetries.
Let the metric $g$ on $M$ be fixed and let 
$\eta$ denote the outward unit conormal along $\partial M$. 
We denote the space of smooth vector fields on $M$ which are tangent to the boundary by
\begin{align}
\label{eqn:partialTM}
\Gamma_\partial(TM) \vcentcolon= \{ Y \in \Gamma(TM) \mid g(Y, \eta) = 0 \text{ on } \partial M \}. 
\end{align}
(Note that the space \eqref{eqn:partialTM} contains \emph{intrinsic} vector fields on the domain manifold, while the space $\mathcal{V}_\Phi$ contains \emph{extrinsic} vector fields along the map $\Phi$.)

Instrumental in our approach is the $L^2(M,g)$-orthogonal decomposition  
\begin{align}
\label{eq:metric_decomp}
\mathcal{S}(M)=\PTr\oplus\Div\oplus\DF,
\end{align}
of the space of symmetric bilinear forms on $M$, 
where the three subspaces are defined as follows.
\begin{align*}
\PTr&\vcentcolon=\{ u g \mid u \in C^\infty(M) \}, 
\tag{pure trace} 
\\
\Div&\vcentcolon=\{ \mathcal{L}_Y g - (\divop_gY)g \mid Y\in\Gamma_\partial(TM) \},  
\tag{pure divergence} 
\\
\DF&\vcentcolon=\{ h \in \mathcal{S}(M) \mid \tr_g h = 0,~\divop_g h = 0,~h(\eta, \cdot)\vert_{T\partial M} = 0 \}.
\tag{transverse traceless} 
\end{align*}  
The decomposition \eqref{eq:metric_decomp} is classical when $M$ is closed. 
We include a proof in Appendix~\ref{sec:decomposition} which also covers the case $\partial M\neq\emptyset$. 
Note that the definition of $\Div$ uses the trace-free part of the Lie derivative, since $\tr_g(\mathcal{L}_Yg)=2\divop_gY$ in dimension $2$. 
This space corresponds to infinitesimal diffeomorphisms that change the conformal structure.

\begin{lemma}
\label{lem:L_operator}
The linear map $L\colon\Gamma_\partial(TM)\to \Div$ given by 
$L(Y)=\mathcal{L}_Y g - (\divop_g Y)g$ 
is surjective and its kernel is the finite-dimensional space $\VConf(M,g)$ of conformal vector fields tangent to the boundary.  
\end{lemma}

\begin{proof}
Surjectivity follows immediately by definition of the subspace $\Div$. 
We have $L(Y)=0$ if and only if the Lie derivative $\mathcal{L}_Y g$ is a pure trace term in the sense that $\mathcal{L}_Y g = ug$ for some function $u\in C^\infty(M)$. 
The latter is the definition of a conformal vector field.  
\end{proof}
 
\begin{remark}
\label{rmk:dim_vconf}
The dimension of $\VConf(M,g)$ depends only on the topology of the surface $M$, vanishes if $M$ has Euler characteristic $\chi(M)<0$, and equals 
\begin{itemize}[nosep]
\item $6$ for the sphere $\Sp^2$,
\item $3$ for the projective space $\mathbb{RP}^2$ and the unit disc $\B^2$,
\item $2$ for the torus $\T^2$,
\item $1$ for the Klein bottle, the annulus, and the Möbius band.
\end{itemize}
Indeed, for closed orientable surfaces the space of conformal fields can be identified with the space of (anti)holomorphic vector fields and classical computations using Riemann--Roch yield the dimensions listed above, see e.\,g.~\cite{Forster1991}. 
Bordered and non-orientable cases can be treated uniformly using the framework of Klein surfaces. 
In particular, there always exists a compact Riemann surface $\widetilde{M}$ together with an anti-holomorphic involution $\tau$ such that the Klein surface can be seen as a quotient $\widetilde{M}/\tau$, see~\cite[Theorem~1.6.1]{AllingGreenleaf1971}. 
Since $\tau$ acts anti-linearly on the space of complex vector fields on $\widetilde{M}$, one has $2\dim\VConf(M)=\dim\VConf(\widetilde{M})$.
\end{remark}

Given any metric variation $h\in\mathcal{S}(M)$, the decomposition \eqref{eq:metric_decomp} and 
Lemma~\ref{lem:L_operator}
imply that there exist a vector field $Y\in\Gamma_\partial(TM)$ and a function $u\in C^\infty(M)$ such that 
\begin{align}\label{eqn:h-L(Y)-ug}
h - L(Y) - ug \in \DF(M).  
\end{align} 
The terms $L(Y)$ and $ug$ correspond to infinitesimal diffeomorphisms and conformal changes, respectively.
Recalling the symmetries \eqref{eqn:conformal_invariance}--\eqref{eqn:diffeo_invariance}
of $\widetilde{E}$, its second variation must vanish in these directions.
This allows us to restrict the variational problem to the slice where the metric variation is transverse traceless. 

\begin{definition}[Two-variable index and nullity]
\label{defn:2index}
The \emph{two-variable index} $\ind_{\teich}(\Phi,g)$ is the maximal dimension of a subspace of \(\mathcal{V}_\Phi\times\DF(M)\) on which the quadratic form $Q_{\teich}$ defined in~\eqref{eq:QtildeE} is negative definite, and the \emph{two-variable nullity} $\nul_{\teich}(\Phi,g)$ is the dimension of the kernel of $Q_{\teich}$ restricted to \(\mathcal{V}_\Phi\times\DF(M)\). 
\end{definition}

\begin{lemma}
\label{lem:projection_Pi}
Let $\Phi\colon M\to\R^{n+1}$ be nonconstant such that $(\Phi,g)$ is a critical point of $\widetilde{E}$. 
Given $(V,h)\in\mathcal{V}_\Phi\times\mathcal{S}(M)$, let $u$ and $Y$ be as in \eqref{eqn:h-L(Y)-ug} such that $V-d\Phi(Y)$ is $L^2(M,g)$-orthogonal to the subspace $d\Phi(\VConf(M,g))$. 
Then the projection 
\begin{align}\label{eqn:projection_Pi}
\begin{aligned}
\Pi\colon\mathcal{V}_\Phi\times\mathcal{S}(M)
&\to\mathcal{V}_\Phi \times\DF(M)
\\
(V,h) &\mapsto \bigl(V - d\Phi(Y), h - L(Y) - ug\bigr)
\end{aligned}
\end{align}
is a well-defined linear map satisfying
\begin{align}\label{eqn:QTinvariance}
Q_{\teich}\bigl(\Pi(V,h),\Pi(V,h)\bigr)=Q_{\teich}\bigl((V,h), (V,h)\bigr)
\end{align}
for all $(V,h)\in\mathcal{V}_\Phi\times\mathcal{S}(M)$. Moreover
\begin{align}
\label{eq:nultildeE}
\nul_{\teich}(\Phi,g)
&= \dim\VConf(M,g) + \dim\bigl(\ker Q_{\teich} \cap \operatorname{Im}(\Pi)\bigr).
\end{align}
\end{lemma}

\begin{remark}
The first term in \eqref{eq:nultildeE} counts the kernel directions arising from conformal vector fields, while the second term quantifies the remaining ones orthogonal to the symmetries. 
One could use just the second term in \eqref{eq:nultildeE} as definition for $\nul_{\teich}(\Phi,g)$. 
However, including the dimension of $d\Phi\bigl(\VConf(M,g)\bigr)\subset\ker Q_{\ener}$ makes the connection with $\nul_{\ener}(\Phi)$ more straightforward.
\end{remark}

\begin{proof}
Since $\abs{\Phi}=1$ on $M$ in the closed case, respectively on $\partial M$ in the free boundary case, and since $Y$ is tangent to $\partial M$, we have $\brt{d\Phi(Y),\Phi}=\frac{1}{2}Y\abs{\Phi}^2=0$ wherever the constraint in the definition of $\mathcal{V}_\Phi$ is imposed, hence $d\Phi(Y)\in\mathcal{V}_\Phi$. 
Moreover, $u=\frac{1}{2}\tr_gh$ is uniquely determined, $Y$ is unique up to $\ker L=\VConf(M,g)$, and the orthogonality condition determines $d\Phi(Y)$ uniquely, so $\Pi$ is well-defined and linear. 
Given any $(V,h)\in\mathcal{V}_\Phi\times \mathcal{S}(M)$ we have, by definition, 
\begin{align*}
(V,h)-\Pi(V,h) &= \bigl(d\Phi(Y), L(Y) + ug\bigr)
\\
&=\bigl(d\Phi(Y), \mathcal{L}_Y g\bigr) + \bigl(0, (u - \divop_g Y)g\bigr)
\in\ker Q_{\teich}
\end{align*}
due to the symmetries \eqref{eqn:conformal_invariance} and \eqref{eqn:diffeo_invariance} of $\widetilde{E}$. 
(Indeed, the first summand is the tangent to the orbit of the diffeomorphism group and 
the second summand is a pure trace term corresponding to an infinitesimal conformal change.)
Since $Q_\teich$ is a symmetric bilinear form, claim \eqref{eqn:QTinvariance} follows.

To prove the nullity formula \eqref{eq:nultildeE}, let $(V,h)$ be any element of the kernel of $Q_{\teich}$ restricted to the subspace $\mathcal{V}_\Phi\times \DF(M)$.  
Since both $(V,h)$ and $\Pi(V,h)$ have metric components in $\DF(M)$, their difference $L(Y) + ug$ must also lie in $\DF(M)$. 
However, the operator $L$ defined in Lemma~\ref{lem:L_operator} maps into the subspace $\Div \oplus \PTr$, which is $L^2$-orthogonal to $\DF(M)$. 
Thus, $L(Y) + ug = 0$, which implies $Y\in\VConf(M,g)$, as in the proof of Lemma \ref{lem:L_operator}. Thus, 
\((V,h)-\Pi(V,h)=(d\Phi(Y), 0)\),
which means that $(V,h)$ can be uniquely decomposed into a summand generated by a conformal vector field and its image under $\Pi$. 
Since $\Phi$ is a branched minimal immersion and non-zero conformal vector fields vanish only at isolated points, $d\Phi$ is injective on $\VConf(M,g)$, and hence the space of variations $(d\Phi(Y), 0)$ with $Y\in\VConf(M,g)$ has dimension $\dim\VConf(M,g)$.
Furthermore, this space is disjoint from the image of $\Pi$ because $\Pi$ requires the vector field component to be $L^2$-orthogonal to $d\Phi(\VConf(M,g))$. 
The total dimension of the kernel is therefore the sum of the dimension of the space of conformal variations and the dimension of the kernel within the image of $\Pi$, as claimed.
\end{proof}

\subsection{Relation with the area index}
\label{sub:area}

In the setting of Definition~\ref{defn:closed_case}, let $\Phi\colon M\to\Sp^n$ be an unbranched minimal immersion with the induced metric $g = \Phi^*g_{\Sp^n}$. 
The second variation of area with respect to normal variations, i.\,e.~sections $V \in \Gamma(N\Phi)$ of the normal bundle along $\Phi$, is given by the area quadratic form 
\begin{align}\label{eqn:area-index-form-closed}
Q_\area(V,V)\vcentcolon= \int_{(M,g)} \Bigl( \abs{\nabla^\perp V}^2 - \abs{\brt{\II(\cdot,\cdot),V}}^2 - 2\abs{V}^2 \Bigr),
\end{align}
where $\nabla^\perp$ is the connection on the normal bundle and $\II$ denotes the second fundamental form. 
The constant $2$ in \eqref{eqn:area-index-form-closed} arises from the trace of the ambient curvature tensor of $\Sp^n$ over the two-dimensional tangent space of $M$. 

In the setting of Definition~\ref{defn:free_boundary_case}, let $\Phi\colon M\to\B^{n+1}$ be an unbranched free boundary minimal immersion with the induced metric $g = \Phi^*g_{\B^{n+1}}$. 
Since the orthogonal boundary contact implies that the outward unit conormal satisfies $\eta=\Phi$ along $\partial M$, any normal variation $V\in\Gamma(N\Phi)$ satisfies $\brt{V,\Phi}=0$ along $\partial M$ and is thus tangent to the boundary sphere. 
The second variation of area is then given by the area quadratic form 
\begin{align}\label{eqn:area-index-form-fb}
Q_\area(V,V)\vcentcolon=\int_{(M,g)}\Bigl(\abs{\nabla^\perp V}^2 - \abs{\brt{\II(\cdot,\cdot),V}}^2 \Bigr) - \int_{\partial M} \abs{V}^2.
\end{align}
The boundary integral in \eqref{eqn:area-index-form-fb} arises from the second fundamental form of $\Sp^n$.

\begin{definition}[Area index and nullity]
\label{defn:area_index}
The \emph{area index} $\ind_\area(\Phi)$ is the maximal dimension of a subspace of $\Gamma(N\Phi)$ on which $Q_\area$ is negative definite, and the \emph{area nullity} $\nul_\area(\Phi)$ is the dimension of the kernel of $Q_\area$. 
Elements of that kernel are called \emph{Jacobi fields} along $\Phi$. 
\end{definition}

Given any variation field $V\in\mathcal{V}_\Phi$, let $\Phi_t\colon M\to\R^{n+1}$ be a smooth family of maps satisfying the relevant constraints of Definition~\ref{defn:closed_case} respectively Definition~\ref{defn:free_boundary_case} with $\Phi_0 = \Phi$ and $\frac{\partial}{\partial t}\big\vert_{t=0}\Phi_t = V$. 
The variation of the induced metric $g_t \vcentcolon= \Phi_t^* g_{\R^{n+1}}$ and its trace are then given by  
\begin{align}\label{eqn:variation_induced_metric}
\tfrac{d}{dt}\big\vert_{t=0} g_t
&= dV \otimes d\Phi + d\Phi \otimes dV,
\\[1ex]\notag
\tr_g(dV \otimes d\Phi) 
&= \brt[\big]{dV(E_1), d\Phi(E_1)} + \brt[\big]{dV(E_2), d\Phi(E_2)} 
= \brt{d\Phi, dV}_g,
\end{align}
where $\{E_1, E_2\}$ denotes a local $g$-orthonormal basis for the tangent space of $M$. 
We denote the trace-free part of the metric variation \eqref{eqn:variation_induced_metric} by  
\begin{align}
\label{eq:eta_V}
h_V\vcentcolon= dV \otimes d\Phi + d\Phi \otimes dV - \brt{d\Phi, dV}_g \, g.
\end{align}  
Geometrically, $h_V$ measures the infinitesimal change in the conformal structure of $(M,g)$ induced by the variation $V\in\mathcal{V}_\Phi$. 
We denote the projection of $V$ onto the normal bundle of $M$ by $V^\perp$.

\begin{lemma}
\label{lem:QA_minimization}
For any vector field $V\in\mathcal{V}_\Phi$ and any trace-free metric variation $h\in\mathcal{S}(M)$
\begin{align}\notag
Q_\area(V^\perp, V^\perp)&\leq Q_{\teich}\bigl((V,h), (V,h)\bigr),
\intertext{with equality if and only if $h=h_V$ defined in \eqref{eq:eta_V}. 
The value at  $h=h_V$ relates to the energy index form via
}
\label{eq:QE_minus_hV}
Q_{\teich}\bigl((V,h_V), (V,h_V)\bigr) 
&= Q_{\ener}(V,V) - \smash{\frac{1}{2}\int_M \abs{h_V}_g^2.}
\end{align}   
\end{lemma}

\begin{proof}
Since $\Phi$ is a minimal immersion and $g$ the induced metric on $M$, we have $\abs{d\Phi}_g^2=2$. 
Since $h$ is trace-free, $\frac{1}{4}(\tr_g h)^2-\det_g(h)=\frac{1}{2}\abs{h}_g^2$ and the second integral in \eqref{eq:QtildeE} becomes
\begin{align*}
\frac{1}{2}\int_{(M,g)} \abs{d\Phi}_g^2 \Bigl( \frac{1}{4}(\tr_g h)^2 - \det\nolimits_g(h) \Bigr)= \frac{1}{2} \int_{(M,g)} \abs{h}_g^2
\end{align*}
while the first integral in \eqref{eq:QtildeE}, i.\,e.~the cross-term, is 
\begin{align*}
-2\int_{(M,g)} \brt[\Big]{d\Phi\otimes dV-\frac{1}{2}\brt{d\Phi,dV}g,h}_g  
=-\int_{(M,g)} \brt{h_V, h}_g.
\end{align*}
Combining these, the two-variable quadratic form is
\begin{align*}
Q_{\teich}\bigl((V,h), (V,h)\bigr) 
= Q_{\ener}(V,V) + \int_{(M,g)} \Bigl( \frac{1}{2}\abs{h}_g^2 - \brt{h_V,h}_g \Bigr). 
\end{align*}
Completing the square 
\(
\frac{1}{2}\abs{h}_g^2 - \brt{h_V, h}_g = \frac{1}{2}\abs{h - h_V}_g^2 - \frac{1}{2}\abs{h_V}_g^2
\)
in the integrand, it is evident that this quantity is minimised (for fixed $V$) by the trace-free tensor $h = h_V$, and the minimum value is \eqref{eq:QE_minus_hV}.
To prove the equality 
\[
Q_\area(V^\perp, V^\perp) = Q_{\teich}\bigl((V,h_V), (V,h_V)\bigr)
\]
we recall that the area functional can be recovered as the minimum of $E_{\tilde g}(\Phi)$ over all metrics $\tilde{g}$. 
For a critical point $(\Phi, g)$, area coincides with $E_g(\Phi)$ and the second variation $Q_\area(V,V)$ of area coincides with the second variation of the two-variable energy $\widetilde{E}$ evaluated at the pair $(V,h)$, where $h$ is a metric variation that minimises the energy quadratic form for fixed~$V$. 
As shown above, we may conclude $h=h_V$. 
Since $\widetilde{E}$ is invariant under the action of $\Diff(M)$, the result depends only on the normal component $V^\perp$. 
In the free boundary case, the boundary term $-\int_{\partial M}\abs{V}^2$ in \eqref{eqn:area-index-form-fb} is already accounted for, since $Q_\ener$ carries the same term via $\sigma=\abs{\partial_\eta\Phi}=1$ in \eqref{eqn:energy-index-form}. 
For the remaining identity we refer to \cite[equation~(28)]{Moore2007}.  
\end{proof} 

The following theorem recovers the equivalence underlying \cite[equation~(28)]{Moore2007} in the closed case, while the free boundary case is new: 
it relies on the tangency condition $g(Y,\eta)=0$ built into  \eqref{eqn:partialTM} and the boundary condition $h(\eta,\cdot)|_{T\partial M}=0$ built into $\DF$, both introduced in Section~\ref{sub:teich}.

\begin{theorem} 
\label{thm:AtildeE}
Let $\Phi$ be a minimal immersion as in Definitions~\ref{defn:closed_case}--\ref{defn:free_boundary_case}, and let $g$ be the induced metric. 
Then 
\begin{align*}
\ind_\area(\Phi)&=\ind_{\teich}(\Phi,g),  & \\
\nul_\area(\Phi)+\dim\VConf(M,g)&=\nul_{\teich}(\Phi,g).
\end{align*}
\end{theorem}

\begin{proof}  
To prove $\ind_\area(\Phi)\leq\ind_{\teich}(\Phi,g)$, 
let $X_\area^-\subset\Gamma(N\Phi)$ be any subspace on which $Q_\area$ is negative definite. 
Recalling the projection $\Pi\colon\mathcal{V}_\Phi\times\mathcal{S}(M)
\to\mathcal{V}_\Phi\times\DF$ 
from \eqref{eqn:projection_Pi} and the metric variation $h_V$ from \eqref{eq:eta_V}, 
we define a linear map $P\colon X_\area^-\to\mathcal{V}_\Phi\times\DF$ by
\begin{align}\label{eqn:20260226}
P(V)\vcentcolon=\Pi(V, h_V)
\end{align}
and claim that $P$ is injective. 
Indeed, if $P(V)=0$ then the normal vector field $V\in X_\area^-$ is equal to the tangential vector field $d\Phi(Y)$ and thus vanishes. 
By Lemmata~\ref{lem:projection_Pi} and~\ref{lem:QA_minimization}, 
\[
Q_{\teich}\bigl(P(V),P(V)\bigr) = Q_{\teich}\bigl( (V, h_V), (V, h_V) \bigr)
=Q_\area(V,V) < 0
\]
for any $0\neq V\in X_\area^-$ 
and hence $\ind_{\teich}(\Phi,g)\geq\dim X_\area^-$. 
To prove the reverse inequality $\ind_\area(\Phi)\geq\ind_{\teich}(\Phi,g)$, let $X_\teich^-\subset\mathcal{V}_\Phi\times \DF$ be any subspace on which $Q_{\teich}$ is negative definite and $0\neq(V,h)\in X_\teich^-$ arbitrary. 
By Lemma~\ref{lem:QA_minimization} 
\begin{align*}
Q_\area(V^\perp, V^\perp) &\leq Q_{\teich}\bigl((V,h), (V,h)\bigr)<0. 
\end{align*}
In particular, $V^\perp\neq0$ and we obtain $\ind_{\area}(\Phi)\geq\dim X_\teich^-$. 
Hence, $\ind_\area(\Phi)=\ind_{\teich}(\Phi,g)$ as claimed. 
For the nullity formula, by \eqref{eq:nultildeE}, it suffices to prove 
\begin{align*}
\nul_{\area}(\Phi) = \dim\bigl( \ker Q_{\teich} \cap \operatorname{Im}(\Pi) \bigr).
\end{align*}
Let $X_\area^{\leq0}\subset\Gamma(N\Phi)$ be a maximal subspace on which $Q_\area$ is nonpositive. 
We extend the map defined in \eqref{eqn:20260226} to $P\colon X_\area^{\leq0}\to\mathcal{V}_\Phi\times \DF$. 
Then $P$ is still injective by the same argument as above, and Lemmata~\ref{lem:projection_Pi} and~\ref{lem:QA_minimization} yield 
\[
Q_{\teich}\bigl(P(V),P(V)\bigr)=Q_{\teich}\bigl( (V, h_V), (V, h_V) \bigr)
=Q_\area(V,V) \leq 0
\]
for any $V\in X_\area^{\leq0}$. 
Thus, $Q_\teich$ is nonpositive on the subspace $P(X_\area^{\leq 0})\subset\operatorname{Im}(\Pi)$ which implies
\[
\ind_{\teich}(\Phi,g) + \dim\bigl( \ker Q_{\teich} \cap \operatorname{Im}(\Pi) \bigr) \geq \dim X_\area^{\leq0} = \ind_{\area}(\Phi) + \nul_{\area}(\Phi).
\]
Since we proved $\ind_{\teich}(\Phi,g) = \ind_{\area}(\Phi)$ in the first part, $\dim\bigl(\ker Q_{\teich} \cap \operatorname{Im}(\Pi) \bigr) \geq \nul_{\area}(\Phi)$.

Conversely, let $X_{\teich}^{\leq0} \subset \operatorname{Im}(\Pi)$ be a maximal subspace on which $Q_{\teich}$ restricted to $\operatorname{Im}(\Pi)$ is nonpositive. 
We recall from Lemma \ref{lem:projection_Pi} that $\mathcal{V}_\Phi\times \DF$ is the direct sum of $\operatorname{Im}(\Pi)$ and the kernel directions generated by conformal vector fields. 
Hence,   
\[
\dim X_{\teich}^{\leq0} = \ind_{\teich}(\Phi,g) + \dim\bigl( \ker Q_{\teich} \cap \operatorname{Im}(\Pi) \bigr).
\]
We claim that the projection $X_{\teich}^{\leq 0}\to\Gamma(N\Phi)$
given by $(V,h) \mapsto V^\perp$ is injective and that $Q_\area$ is nonpositive on its image.  
Indeed, by Lemma \ref{lem:QA_minimization},  
\begin{align}\label{eqn:20260226-2}
Q_{\area}(V^\perp, V^\perp) \leq Q_{\teich}\bigl((V,h), (V,h)\bigr) \leq 0
\end{align}
for any $(V,h) \in X_{\teich}^{\leq 0}$.  
If $V^\perp = 0$, then $V=d\Phi(Y)$ is a tangential vector field. 
For such $V$, the trace-free part of the induced metric variation coincides with the trace-free Lie derivative of $g$, so $h_V = L(Y)$.  
Moreover, inequality \eqref{eqn:20260226-2} implies $Q_{\teich}((V,h), (V,h)) = 0$, which by Lemma~\ref{lem:QA_minimization} requires $h = h_V$.  
Because $(V,h) \in \operatorname{Im}(\Pi)$, we have $h \in \DF$. 
Since $h = L(Y) \in \Div \oplus \PTr$, the orthogonal decomposition \eqref{eq:metric_decomp} yields $h = 0$ and $L(Y)=0$. 
By Lemma \ref{lem:L_operator}, $Y$ is a conformal vector field. 
However, since $(V,h) \in \operatorname{Im}(\Pi)$, the vector field $V = d\Phi(Y)$ is $L^2$-orthogonal to $d\Phi(\VConf(M,g))$, 
which implies $V = 0$ and thus $(V,h)=0$.
Consequently, 
\[
\ind_{\area}(\Phi) + \nul_{\area}(\Phi)\geq\dim X_{\teich}^{\leq 0} = \ind_{\teich}(\Phi,g) + \dim\bigl( \ker Q_{\teich} \cap \operatorname{Im}(\Pi) \bigr).
\]
Subtracting again $\ind_{\area}(\Phi) = \ind_{\teich}(\Phi,g)$, we obtain $\nul_{\area}(\Phi) \geq \dim\bigl( \ker Q_{\teich} \cap \operatorname{Im}(\Pi)\bigr)$. 
Combined with the reverse inequality proven above, the claim follows. 
\end{proof}

\begin{theorem}
\label{thm:energy_teich_bounds}
Let $\Phi$ be a minimal immersion as in Definitions~\ref{defn:closed_case}--\ref{defn:free_boundary_case}, and let $g$ be the induced metric. 
Then 
\begin{align*}
\ind_{\teich}(\Phi,g) 
&\leq \dim\DF(M) + \ind_\ener(\Phi), 
\\
\ind_{\teich}(\Phi,g) + \nul_{\teich}(\Phi,g) 
&\leq \dim\DF(M) + \ind_\ener(\Phi) + \nul_\ener(\Phi).
\end{align*}
\end{theorem}

\begin{proof}
By Definition~\ref{defn:2index}, the two-variable quadratic form $Q_\teich$ is evaluated on the domain $\mathcal{W}\vcentcolon=\mathcal{V}_\Phi\times \DF(M)$. 
By \eqref{eq:QtildeE}, we have 
$Q_\teich\bigl((V,0), (V,0)\bigr) = Q_\ener(V,V)$. 
The restriction of $Q_\teich$ to the subspace $\mathcal{W}_0\vcentcolon=\mathcal{V}_\Phi\times\{0\}\subset\mathcal{W}$ thus coincides with the energy quadratic form.
Let $X_{\teich}^{-}\subset\mathcal{W}$ be a subspace on which $Q_\teich$ is negative definite, chosen such that $\dim X_{\teich}^- = \ind_\teich(\Phi,g)$. 
Since $\dim \DF$ is the codimension of $\mathcal{W}_0$ in $\mathcal{W}$, we have 
\begin{align*} 
\ind_\ener(\Phi)
&\geq\dim(X_{\teich}^{-}\cap \mathcal{W}_0) \geq \dim X_{\teich}^{-} - \dim \DF(M)  
\end{align*}
which proves the first inequality. 
The second inequality follows similarly by replacing $X_{\teich}^{-}$ with $X_{\teich}^{-}\oplus X_{\teich}^{0}$, where $X_{\teich}^{0}\subset \mathcal{W}$ is the kernel of $Q_\teich$ restricted to $\mathcal{W}$.  
\end{proof}

\subsection{Equivariant area index}
\label{sub:equivariant-area}

Given a finite subgroup $G$ of $\Diff(M)$ acting on $(M,g)$ by isometries, let $\Phi\colon M\to\R^{n+1}$ be an unbranched minimal immersion as in Definitions~\ref{defn:closed_case}--\ref{defn:free_boundary_case} which is $G$-equivariant with respect to an orthogonal representation $\rho\colon G \to \Ogroup(n+1)$. 
The normal bundle $\Gamma(N\Phi)$ is a $G$-invariant subspace of the space $\mathcal{V}_\Phi$ of admissible variations. 
Indeed, the orthogonality of the representation $\rho(\gamma)$ and the equivariance of $d\Phi$ ensure that $\rho_\ener$ maps normal variations to normal variations. Thus, we obtain a linear representation $\rho_\area$ on $\Gamma(N\Phi)$ by restricting the representation $\rho_\ener$ on $\mathcal{V}_\Phi$ defined in \eqref{eqn:rho_E}.

\begin{lemma}
\label{lem:QApreserved}
The representation $\rho_\area$ preserves the area quadratic form $Q_\area$ given in \eqref{eqn:area-index-form-closed} respectively \eqref{eqn:area-index-form-fb}.
\end{lemma}

\begin{proof}
Because $G$ acts on $(M,g)$ by isometries and $\Phi$ is equivariant, the second fundamental form $\II$ and the connection $\nabla^\perp$ on the normal bundle are preserved in the sense that they commute with the group action. 
The invariance of $Q_\area$ then follows from the change of variables formula and the orthogonality of $\rho(\gamma)$. 
In the free boundary case, the invariance of the boundary integral follows additionally from the fact that the isometric action of $G$ maps $\partial M$ to itself and preserves its induced Riemannian boundary measure.
\end{proof}

The action of $G$ on $M$ also induces a natural representation on the space $\mathcal{S}(M)$ of symmetric $(0,2)$-tensors given by $h \mapsto (\gamma^{-1})^* h$ for $\gamma \in G$. 
Because $G$ acts by isometries ($\gamma^* g = g$), this pullback operation commutes with the trace and divergence operators computed with respect to $g$. 
Furthermore, recalling the identity 
\begin{align}\label{eqn:pullbackLie}
\gamma^* (\mathcal{L}_Y g) = \mathcal{L}_{(\gamma^*Y)} \bigl(\gamma^* g\bigr) = \mathcal{L}_{(\gamma^*Y)} g
\end{align}
for the pullback of a Lie derivative for any vector field $Y \in \Gamma_\partial(TM)$, we see that the pure divergence terms are mapped to pure divergence terms. 
Consequently, the $L^2(M,g)$-orthogonal decomposition \eqref{eq:metric_decomp} is preserved by the group action. 
In particular, $\DF(M)$ is a $G$-invariant subspace, and for any class $r \in \IrRep{G}$, we denote its $r$-isotypic component by $\DF^r(M)$ (cf.~Definition~\ref{defn:Isotypic}). 
We can thus define a linear representation $\rho_\teich$ of $G$ on  
$\mathcal{V}_\Phi\times \mathcal{S}(M)$ by 
\begin{align}\label{eqn:rho_T}
\rho_\teich(\gamma)(V,h) 
&\vcentcolon=\bigl( \rho_\ener(\gamma)V,\,(\gamma^{-1})^* h \bigr).
\end{align}

\begin{lemma}
\label{lem:equivariance_teich}
The linear representation $\rho_\teich$ defined in \eqref{eqn:rho_T} satisfies the following properties:
\begin{enumerate}[label={\normalfont(\roman*)}]
\item\label{lem:equivariance_teich-i} 
The quadratic form $Q_\teich$ given in \eqref{eq:QtildeE} is preserved by $\rho_\teich$ in the sense that 
\[
Q_{\teich}\bigl(\rho_\teich(\gamma)(V,h),\,\rho_\teich(\gamma)(V,h)\bigr)=Q_{\teich}\bigl((V,h),\,(V,h)\bigr).
\]
\item\label{lem:equivariance_teich-ii} 
The projection $\Pi\colon\mathcal{V}_\Phi\times\mathcal{S}(M)\to\mathcal{V}_\Phi\times\DF(M)$ defined in \eqref{eqn:projection_Pi} is $G$-equivariant in the sense that 
\(\rho_\teich(\gamma) \circ \Pi = \Pi \circ \rho_\teich(\gamma)\) for all $\gamma \in G$. 
\item\label{lem:equivariance_teich-iii} 
The map $V \mapsto h_V$ defined in \eqref{eq:eta_V} is $G$-equivariant in the sense that
\(h_{(\rho_\ener(\gamma)V)} = (\gamma^{-1})^* h_V \) for all $V\in\mathcal{V}_\Phi$ and all $\gamma\in G$.
\end{enumerate}
\end{lemma}
 
\begin{proof}
By \eqref{eq:QtildeE} the quadratic form $Q_\teich$ has three terms. 
The first term is $Q_\ener(V,V)$ which is preserved by $\rho_\ener$ (and hence $\rho_\teich$) as stated in Lemma~\ref{lem:preserved}. 
In the second and third terms we apply the change of variables $x \mapsto \gamma^{-1}x$ when integrating over $M$. 
Since $G$ acts on $(M,g)$ by isometries, the trace and determinant operators commute with the pullback by $\gamma^{-1}$, and the volume form is $G$-invariant.   
By equivariance of $\Phi$, we have $\Phi \circ \gamma = \rho(\gamma) \circ \Phi$, so differentiating at $\gamma^{-1}x$ yields 
$d\Phi\vert_x \circ d\gamma\vert_{\gamma^{-1}x} = \rho(\gamma) \circ d\Phi\vert_{\gamma^{-1}x}$ and hence 
\begin{align}\label{eqn:dPhi_pullback}
d\Phi\vert_x &= \rho(\gamma) \circ d\Phi\vert_{\gamma^{-1}x} \circ d(\gamma^{-1})\vert_x.
\intertext{Similarly, for $V\in\mathcal{V}_\Phi$,}
\label{eqn:dV_pullback}
d\bigl(\rho_\ener(\gamma)V\bigr)\vert_x &= \rho(\gamma) \circ dV\vert_{\gamma^{-1}x} \circ d(\gamma^{-1})\vert_x.
\end{align}
Because $\rho(\gamma)\in\Ogroup(n+1)$, the pointwise inner products of the differentials \eqref{eqn:dPhi_pullback}--\eqref{eqn:dV_pullback} coincide with the inner products of $d\Phi$ and $dV$ at $\gamma^{-1}x$ pulled back by $\gamma^{-1}$. 
This completes the proof of \ref{lem:equivariance_teich-i} in view of \eqref{eq:QtildeE}. 
 
To prove \ref{lem:equivariance_teich-ii}, let $(V,h) \in\mathcal{V}_\Phi\times\mathcal{S}(M)$. 
We recall the decomposition $h = u g + L(Y) + h_{0}$ with $h_{0} \in \DF$
and $Y \in \Gamma_\partial(TM)$ from \eqref{eq:metric_decomp}, as well as the maps
\(L(Y)=\mathcal{L}_Y g - (\divop_g Y)g\) and  
\(\Pi(V,h)=\bigl(V - d\Phi(Y), h_0\bigr)\) 
from Lemmata~\ref{lem:L_operator} and \ref{lem:projection_Pi}. 
Pulling back by $\gamma^{-1}$ and using identity \eqref{eqn:pullbackLie}, we obtain
\[
(\gamma^{-1})^* h = (u \circ \gamma^{-1}) g + L\bigl((\gamma^{-1})^* Y\bigr) + (\gamma^{-1})^* h_{0}. 
\]
Since $G$ acts by isometries, $(\gamma^{-1})^* Y$ is the corresponding vector field for the pure divergence part of $(\gamma^{-1})^* h$. 
Applying the projection $\Pi$ to the transformed pair yields
\[
\Pi\Bigl(\rho_\teich(\gamma)(V,h)\Bigr)
=\Pi\Bigl(\rho_\ener(\gamma)V, (\gamma^{-1})^* h\Bigr) 
=\Bigl(\rho_\ener(\gamma)V - d\Phi\bigl((\gamma^{-1})^* Y\bigr), (\gamma^{-1})^* h_{0}\Bigr).
\]
We compare this to $\rho_\teich(\gamma) \Pi(V,h) = \bigl(\rho_\ener(\gamma)(V - d\Phi(Y)), (\gamma^{-1})^* h_{0}\bigr)$. Evaluating the vector field $d\Phi((\gamma^{-1})^* Y)$ at any $x \in M$, we have
\begin{align*}
d\Phi\vert_x\bigl( ((\gamma^{-1})^* Y)(x) \bigr)
&=d\Phi\vert_x \bigl( d\gamma\vert_{\gamma^{-1}x}(Y(\gamma^{-1}x)) \bigr)
= d(\Phi \circ \gamma)\vert_{\gamma^{-1}x}(Y(\gamma^{-1}x)) 
\\
&= d(\rho(\gamma)\Phi)\vert_{\gamma^{-1}x}(Y(\gamma^{-1}x))
= \rho(\gamma) \bigl( d\Phi\vert_{\gamma^{-1}x}(Y(\gamma^{-1}x)) \bigr) 
\\
&= \bigl(\rho_\ener(\gamma)(d\Phi(Y))\bigr)(x).
\end{align*}
Thus, $d\Phi((\gamma^{-1})^* Y) = \rho_\ener(\gamma)(d\Phi(Y))$. Because conformal vector fields are preserved by isometries, the required $L^2$-orthogonality to $d\Phi(\VConf(M,g))$ is preserved as well, which proves \ref{lem:equivariance_teich-ii}.

Finally, \ref{lem:equivariance_teich-iii} follows immediately from the definition of $h_V$ in \eqref{eq:eta_V} together with the pointwise relations \eqref{eqn:dPhi_pullback} and \eqref{eqn:dV_pullback}.  
\end{proof}
 
\begin{lemma}
\label{lem:class_well_defined}
Let $Q$ be a $G$-invariant symmetric bilinear form on a $G$-invariant vector space $W$ with finite index.
Then there exists a $G$-invariant subspace of $W$ of maximal dimension on which $Q$ is negative definite, and any two such subspaces are $G$-equivariantly isomorphic. 
Moreover, every subspace of $W$ which is maximal with respect to inclusion among subspaces on which $Q$ is negative semi-definite contains the kernel of $Q$. 
\end{lemma}

\begin{proof}
Given any subspace $X\subset W$ of maximal dimension on which $Q$ is negative definite, the finite-dimensional subspace $Z\vcentcolon=\sum_{\gamma\in G}\gamma X$ is $G$-invariant (but $Q$ is not necessarily negative definite on $Z$). 
Let $f$ denote the endomorphism of $Z$ induced by $Q\vert_Z$ with respect to a $G$-invariant inner product on $Z$. 
Since $f$ commutes with the action of $G$, the sum $Z^{-}$ of the eigenspaces of $f$ with negative eigenvalues is $G$-invariant. 
Moreover, $Q$ is negative definite on $Z^{-}$, hence $\dim Z^{-}\leq \dim X$. 
Conversely, $Q$ is nonnegative on the sum $Z^{+}$ of the eigenspaces of $f$ with nonnegative eigenvalues, so $X\cap Z^{+}=\{0\}$, and $X\subset Z=Z^{-}\oplus Z^{+}$ implies $\dim X\leq\dim Z^{-}$. 
This proves existence. 

Let $X,X'\subset W$ be $G$-invariant subspaces of maximal dimension on which $Q$ is negative definite. 
Choosing a basis $b_1,\ldots,b_m$ of $X$ with $Q(b_i,b_j)=-\delta_{ij}$, the linear map $P\colon W\to X$ given by $P(w)\vcentcolon=-\sum_{i=1}^mQ(w,b_i)\,b_i$ restricts to the identity on $X$, and its kernel is the $Q$-orthogonal complement $X^{\perp}$ of $X$. 
Hence $W=X\oplus X^{\perp}$, where $X^{\perp}$ is $G$-invariant because $Q$ and $X$ are, and $Q\geq0$ on $X^{\perp}$ by maximality of $X$. 
Consequently, $P$ is $G$-equivariant, and its restriction to $X'$ is injective. 
Since $\dim X=\dim X'$, this restriction is an equivariant isomorphism. 
If $Q$ is negative semi-definite on a subspace $U\subset W$, then $Q$ is also negative semi-definite on $U+\ker Q$, hence $\ker Q\subset U$ if $U$ is maximal with respect to inclusion.
\end{proof}

By Lemmata~\ref{lem:QApreserved} and~\ref{lem:equivariance_teich}\;\ref{lem:equivariance_teich-i}, the quadratic forms $Q_\area$ on $\Gamma(N\Phi)$ and $Q_\teich$ on $\mathcal{V}_\Phi\times\DF(M)$ are $G$-invariant, and they have the same finite index by Theorem~\ref{thm:AtildeE}.
Hence, by Lemma~\ref{lem:class_well_defined}, the isomorphism class of the representation of $G$ on a $G$-invariant subspace of maximal dimension on which the respective quadratic form is negative definite is well-defined, and so is the isomorphism class of the representation on the kernel of the quadratic form. 
As in Definition~\ref{defn:Isotypic} we denote these isomorphism classes of representations by 
$\Ind_\area(\Phi)$, 
$\Ind_\teich(\Phi,g)$, 
$\Nul_\area(\Phi)$, 
$\Nul_\teich(\Phi,g)$, 
and introduce for any $r \in \IrRep{G}$ the corresponding $r$-isotypic indices 
$\ind_\area^r$, 
$\ind_\teich^r$, and nullities 
$\nul_\area^r$, 
$\nul_\teich^r$.

We denote by $\cA_{\mathrm{tt}}$ respectively $\cA_\VConf$ the isomorphism class of the representation of $G$ on the space $\DF(M)$ of transverse-traceless tensors, 
respectively the space $\VConf(M,g)$ of conformal vector fields. 
The following statement is an equivariant version of Theorems 
\ref{thm:AtildeE} and \ref{thm:energy_teich_bounds}.

\begin{theorem}
\label{thm:equivariant_area_bounds}
Let $\Phi$ be a $G$-equivariant minimal immersion as in Definitions~\ref{defn:closed_case}--\ref{defn:free_boundary_case}, and let $g$ be the induced metric. 
Then  
\begin{align}
\label{thm:equivariant_area_bounds-i}
\Ind_\area(\Phi) 
&=\Ind_{\teich}(\Phi,g), 
&
\Nul_\area(\Phi)\oplus\cA_\VConf 
&=\Nul_{\teich}(\Phi,g). 
\intertext{Furthermore, we have the following ordering of representation classes:} 
\label{thm:equivariant_area_bounds-ii}
\Ind_\area(\Phi) & \leq 
\Ind_\ener(\Phi) \oplus 
\cA_{\mathrm{tt}}, 
&
\Ind_\area(\Phi) \oplus 
\Nul_\area(\Phi) \oplus 
\cA_\VConf & \leq 
\Ind_\ener(\Phi)  \oplus 
\Nul_\ener(\Phi) \oplus 
\cA_{\mathrm{tt}}.
\end{align} 
In particular, for any $r\in\IrRep{G}$, the isotypic indices and nullities satisfy
\begin{align*}
\ind^r_\area(\Phi) &\leq \ind^r_\ener(\Phi) + \dim\DF^r(M), 
\\[1ex]
\ind^r_\area(\Phi) + \nul^r_\area(\Phi) &\leq \ind^r_\ener(\Phi) + \nul^r_\ener(\Phi) + \dim\DF^r(M).
\end{align*}
\end{theorem}

\begin{proof}
To prove \eqref{thm:equivariant_area_bounds-i}, we recall from \eqref{eqn:20260226} that 
$P\colon\Gamma(N\Phi)\to\mathcal{V}_\Phi\times\DF(M)$ defined by $P(V) = \Pi(V, h_V)$ is an injective linear map. 
By Lemma~\ref{lem:equivariance_teich}\;\ref{lem:equivariance_teich-ii}--\ref{lem:equivariance_teich-iii}, we have 
\[
P\bigl(\rho_\area(\gamma)V\bigr) = \Pi\bigl(\rho_\ener(\gamma)V, h_{\rho_\ener(\gamma)V}\bigr) = \Pi\bigl(\rho_\ener(\gamma)V, (\gamma^{-1})^* h_V\bigr) = \rho_\teich(\gamma) P(V),
\]
which shows that $P$ is also $G$-equivariant. 
Furthermore, Lemma~\ref{lem:QA_minimization} and Lemma~\ref{lem:projection_Pi} ensure that 
\[
Q_\area(V,V) = Q_\teich(P(V), P(V)).
\]
Consequently, $P$ maps a $G$-invariant subspace of maximal dimension on which $Q_\area$ is negative definite equivariantly onto a $G$-invariant subspace on which $Q_\teich$ is negative definite, which has maximal dimension by Theorem~\ref{thm:AtildeE}. 
By Lemma~\ref{lem:class_well_defined}, this proves $\Ind_\area(\Phi)=\Ind_\teich(\Phi,g)$. 

For the nullity, let $K\vcentcolon=\ker Q_\teich\cap\operatorname{Im}(\Pi)$ and let $X_\area^{\leq0}$ be the sum of $\ker Q_\area$ with a maximal negative definite subspace for $Q_\area$. 
By Theorem~\ref{thm:AtildeE}, $P(X_\area^{\leq0})\subset\operatorname{Im}(\Pi)$ has maximal dimension among subspaces of $\operatorname{Im}(\Pi)$ on which $Q_\teich$ is negative semi-definite, and hence contains $K$ by Lemma~\ref{lem:class_well_defined}. 
Since $Q_\teich(P(V),P(W))=Q_\area(V,W)$ for all $V,W\in\Gamma(N\Phi)$ by polarisation, the kernel of the restriction of $Q_\teich$ to $P(X_\area^{\leq0})$ is $P(\ker Q_\area)$, and it contains $K$ because $K\subset P(X_\area^{\leq0})$ lies in the kernel of $Q_\teich$.
Since $\dim K=\nul_\area(\Phi)$ by Theorem~\ref{thm:AtildeE}, we obtain $P(\ker Q_\area)=K$. 
The kernel of $Q_\teich$ restricted to $\mathcal{V}_\Phi\times\DF(M)$ is the direct sum of $K$ and the $G$-invariant subspace $d\Phi(\VConf(M,g))\times\{0\}$ by Lemma~\ref{lem:projection_Pi}, and $d\Phi$ is $G$-equivariant as shown in the proof of Lemma~\ref{lem:equivariance_teich}\,\ref{lem:equivariance_teich-ii}, which yields $\Nul_\area(\Phi)\oplus\cA_\VConf=\Nul_\teich(\Phi,g)$.

To prove the first claim in \eqref{thm:equivariant_area_bounds-ii}, let $\pi\colon\mathcal{V}_\Phi\times\DF(M)\to\DF(M)$ be the ($G$-equivariant) projection onto the second factor, and $X_\teich^-$ a $G$-invariant subspace of maximal dimension on which $Q_\teich$ is negative definite.
Then,
\[
X_\teich^- \cong \ker(\pi|_{X_\teich^-}) \oplus \pi(X_\teich^-).
\]
Elements in $\ker(\pi|_{X_\teich^-})$ are of the form $(V, 0)$ and satisfy $Q_\teich((V, 0), (V, 0)) = Q_\ener(V, V) < 0$. 
The $L^2(M,g)$-orthogonal projection of $\mathcal{V}_\Phi$ onto $X_\ener^-$ is $G$-equivariant by Lemma~\ref{lem:preserved} and injective on $\ker(\pi|_{X_\teich^-})$ because $Q_\ener\geq0$ on the orthogonal complement of $X_\ener^-$, so the representation on $\ker(\pi|_{X_\teich^-})$ is a subrepresentation of $\Ind_\ener(\Phi)$.
Arguing similarly with $\pi(X_\teich^-)\subset\DF(M)$, we obtain 
\[
\Ind_{\teich}(\Phi,g)\leq\Ind_\ener(\Phi) \oplus 
\cA_{\mathrm{tt}}. 
\]
The second claim in \eqref{thm:equivariant_area_bounds-ii} follows by the same argument applied to the $G$-invariant subspace $X_\teich^-\oplus X_\teich^0\subset\mathcal{V}_\Phi\times\DF(M)$, where $X_\teich^0$ denotes the kernel of $Q_\teich$ restricted to $\mathcal{V}_\Phi\times\DF(M)$, using the $L^2(M,g)$-orthogonal projection onto $X_\ener^-\oplus X_\ener^0$, which is injective on subspaces on which $Q_\ener$ is negative semi-definite.
\end{proof}

\subsection{Two-sided surfaces in codimension one}
\label{sub:two-sided}

As before, $\Phi$ denotes a $G$-equivariant minimal immersion of a surface $M$. 
This subsection specialises to the codimension-one case where the target of $\Phi$ is three-dimensional: 
either $\Sp^3$ as in Definition~\ref{defn:closed_case} if $M$ is closed, or $\B^3$ as in Definition~\ref{defn:free_boundary_case} if $M$ has boundary. 
Moreover, we assume that $\Phi$ is \emph{two-sided}, meaning that there exists a global unit normal section $N \in \Gamma(N\Phi)$ along $\Phi(M)$. Any normal variation field $V\in\Gamma(N\Phi)$ is then uniquely given by $V = fN$ for some scalar function $f\in C^\infty(M)$. 
Under this identification, the area quadratic form \eqref{eqn:area-index-form-closed} in the closed case,
respectively \eqref{eqn:area-index-form-fb} in the free boundary case, 
corresponds to the \emph{scalar Jacobi form} $Q_\jacobi(f,f)$ defined on $C^\infty(M)$ by
\begin{align}\label{eqn:jacobi_form_closed}
Q_\jacobi(f,f)&\vcentcolon=\int_{(M,g)} \Bigl( \abs{df}_g^2 - (\abs{\II}^2 + 2)f^2 \Bigr)
\shortintertext{respectively}
\label{eqn:jacobi_form_fb}
Q_\jacobi(f,f)&\vcentcolon=\int_{(M,g)} \Bigl( \abs{df}_g^2 - \abs{\II}^2 f^2 \Bigr) - \int_{\partial M} f^2.
\end{align}
In analogy with Definition~\ref{defn:Isotypic}, we denote the representation classes associated with $Q_\jacobi$ by $\Ind_\jacobi(\Phi)$ and $\Nul_\jacobi(\Phi)$.

Because $\rho(\gamma)$ is an isometry of the ambient space and $N(x)$ is the unit normal at $\Phi(x)$, the vector $\rho(\gamma)N(\gamma^{-1}x)$ is necessarily a unit normal at $\Phi(x)$. 
Since the surface is two-sided, there are only the two choices $\pm N(x)$ for the unit normal at each point. 
The following definition allows us to track whether the group action preserves or reverses the orientation of the normal bundle.

\begin{definition}[Relative orientation representation]
\label{defn:relative_orientation_representation}
In the setting above, the \emph{relative orientation representation} of $\Phi$ is the representation $\epsilon\colon G\to\Ogroup(1)\cong\{\pm1\}$ of degree $1$ defined by 
\(
\epsilon(\gamma) N =\rho_\ener(\gamma) N
\) 
for each $\gamma\in G$, recalling $\rho_\ener$ from \eqref{eqn:rho_E}. 
\end{definition}

\begin{definition} 
\label{defn:tensor_representation}
Given a representation $\rho\colon G\to\Ogroup(V)$ and a representation $\varrho\colon G\to\Ogroup(W)$, we write $\rho\otimes\varrho$ for the representation on $V\otimes W$ defined by $(\rho\otimes\varrho)(\gamma)
\vcentcolon=\rho(\gamma)\otimes\varrho(\gamma)$ 
for all $\gamma\in G$.
For any representation class $\cA = \sum_{r\in\IrRep{G}} m_r r$ and any representation $\varrho$, the tensor product $\cA\otimes\varrho$ denotes the isomorphism class of $\sum_{r\in\IrRep{G}} m_r (r \otimes \varrho)$.
\end{definition}

\begin{proposition}
\label{prop:indJ_indA}
The following identities of representation classes hold:
\begin{align*}
\Ind_\jacobi(\Phi)&=\Ind_\area(\Phi)\otimes\epsilon, &
\Nul_\jacobi(\Phi)&=\Nul_\area(\Phi)\otimes\epsilon.
\end{align*}
In particular, for any $r\in\IrRep{G}$, we have 
$\ind_\jacobi^r(\Phi)=\ind_\area^{r\otimes\epsilon}(\Phi)$ and 
$\nul_\jacobi^r(\Phi)=\nul_\area^{r\otimes\epsilon}(\Phi)$.
\end{proposition}

\begin{proof}
Consider the linear isomorphism $\mathcal{T}\colon C^\infty(M)\to\Gamma(N\Phi)$ defined by $\mathcal{T}(f) \vcentcolon= fN$. 
Recalling the respective forms \eqref{eqn:jacobi_form_closed} and \eqref{eqn:area-index-form-closed} in the closed case, respectively \eqref{eqn:jacobi_form_fb} and \eqref{eqn:area-index-form-fb} in the free boundary case, we have 
\(Q_\jacobi(f,f) = Q_\area(fN, fN)\), 
so $\mathcal{T}$ maps any subspace of maximal dimension on which $Q_\jacobi$ is negative definite to a subspace of maximal dimension on which $Q_\area$ is negative definite (and similarly for the kernels). 
To relate the representation classes, let $\gamma \in G$ and $f \in C^\infty(M)$. 
By definition \eqref{eqn:rho_E} of $\rho_\ener$, 
\begin{align*}
\bigl(\rho_\area(\gamma) \mathcal{T}(f)\bigr)(x) 
= \rho(\gamma) \bigl( f(\gamma^{-1}x) N(\gamma^{-1}x) \bigr) 
&= f(\gamma^{-1}x) \bigl( \rho_\ener(\gamma) N \bigr)(x) 
\\
&= (\rho_\spec(\gamma) f)(x) \, \epsilon(\gamma) N(x) 
= \epsilon(\gamma) \mathcal{T}\bigl(\rho_\spec(\gamma) f\bigr)(x).
\end{align*}
The identity $\rho_\area(\gamma) \circ \mathcal{T} = \epsilon(\gamma) (\mathcal{T} \circ \rho_\spec(\gamma))$ implies an isomorphism between the representation $\rho_\spec \otimes \epsilon$ on functions and the representation $\rho_\area$ on the normal bundle. 
Consequently, $\Ind_\area(\Phi)=\Ind_\jacobi(\Phi)\otimes\epsilon$. 
Since $\epsilon\otimes\epsilon$ is the trivial representation, the asserted isomorphism follows. 
The argument for the nullity is analogous.
\end{proof}

\section{Representation theory of the dihedral group}\label{sec:dihedral}

Given $2\leq k\in\N$ the dihedral group $\dih_k$ of order $2k$ is usually understood as the symmetry group of a regular $k$-gon (see Figure~\ref{fig:k-gon}).
Abstractly, $\dih_k$ admits a Coxeter presentation as the group generated by two involutions $\gamma_1, \gamma_2$ satisfying 
\begin{align*}
\gamma_1^2=\gamma_2^2=(\gamma_1 \gamma_2)^k=1.
\end{align*} 
We call $\gamma\in\dih_k$ a \emph{reflection} (respectively a \emph{rotation}) if it is the product of an odd (respectively even) number of generators. 
In this sense, the group $\dih_k$ contains $k$ reflections and $k$ rotations (including the identity $1$). 
Moreover, it contains the cyclic group $\Z_k$ of order $k$ 
as a subgroup generated by $\gamma_1\gamma_2$. 
Any representation $\rho\colon\dih_k\to\Ogroup(n)$ is uniquely determined by the two orthogonal matrices $\rho(\gamma_1)$ and $\rho(\gamma_2)$, provided that $\rho(\gamma_1)^2 = \rho(\gamma_2)^2 =(\rho(\gamma_1)\rho(\gamma_2))^k$ all coincide with the identity matrix.
We identify each irreducible representation with its isomorphism class, so that $\IrRep{\dih_k}$ consists of the following maps (see \cite[§\,5.3]{Serre1977}).
\begin{itemize} 
\item The \emph{trivial representation} $\triv\colon \dih_k\to\Ogroup(1)\cong\{\pm 1\}$ given by 
$\triv(\gamma_1)=1$ and $\triv(\gamma_2)=1$.  
\item The \emph{orientation representation} $\Or \colon \dih_k\to\Ogroup(1)$ given by 
$\Or(\gamma_1)=-1$ and $\Or(\gamma_2)=-1$, 
mapping every reflection to $-1$, and every rotation to $1$. 
\item If $k$ is even, there are two additional representations $\m_1, \m_2 \colon \dih_k\to\Ogroup(1)$ of degree $1$ given by 
$\m_1(\gamma_1)=-1$ and $\m_1(\gamma_2)= 1$, respectively 
$\m_2(\gamma_1)= 1$ and $\m_2(\gamma_2)=-1$.

\item A set of \emph{standard representations} $\st_j \colon \dih_k \to \Ogroup(2)$, of degree $2$, for $j \in \{ 1, \dots, \lfloor\frac{k-1}{2} \rfloor \}$. In terms of the standard basis of $\R^2$, these are given by
\begin{align}
\label{eqn:St-matrix}
\st_j(\gamma_1) &= \begin{pmatrix} 1 & 0 \\ 0 & -1 \end{pmatrix}, & 
\st_j(\gamma_2) &= \begin{pmatrix} \cos(2\pi j/k) & \hphantom{-}\sin(2\pi j/k) \\ \sin(2\pi j/k) & -\cos(2\pi j/k) \end{pmatrix}.
\end{align}
This choice of matrices corresponds to $\gamma_1$ acting as a reflection across the $x_1$-axis in the plane and the product $\gamma_2\gamma_1$ acting as a rotation by angle $2\pi j/k$ around the origin. 
\end{itemize}
In particular, $\IrRep{\dih_k}=\bigl\{\triv, \Or, \m_1, \m_2, \st_1, \ldots,\st_{\lfloor\frac{k-1}{2}\rfloor}\bigr\}$, omitting $\m_1$ and $\m_2$ if $k$ is odd. 
Note that if $k=2$ then $\st_1$ defined as in \eqref{eqn:St-matrix} would be reducible and equal to $\m_1\oplus\m_2$. 

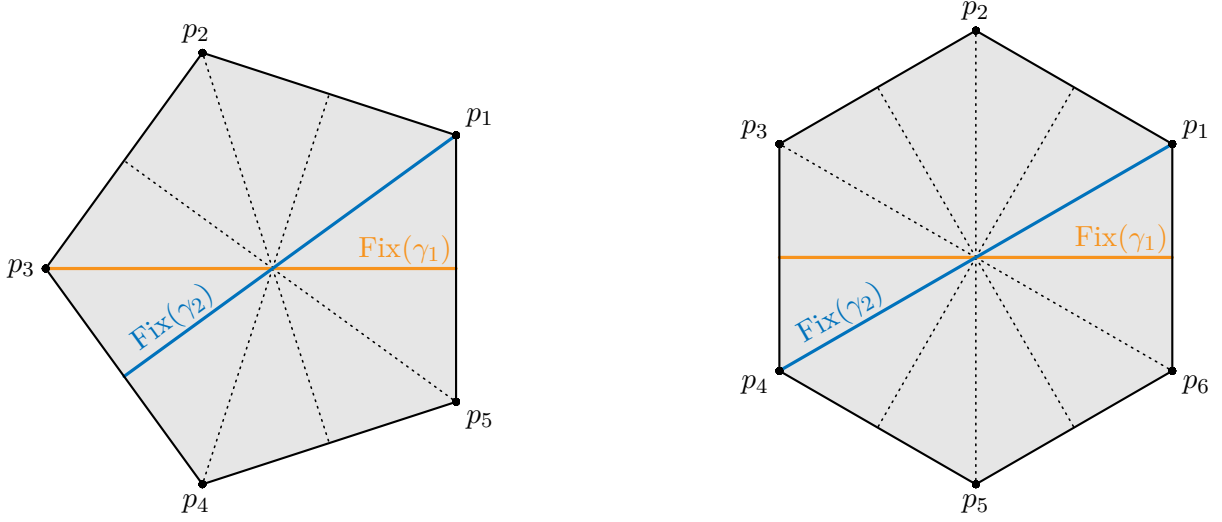
\begin{figure}
\centering
\pgfmathsetmacro{\npar}{5} 
\pgfmathsetmacro{\globalscale}{3} 
\begin{tikzpicture}[line cap=round,line join=round,semithick,scale=\globalscale]
\pgfmathsetmacro{\nminusone}{\npar-1}
\foreach[count=\i]\k in {0,...,\npar}{
\coordinate(p\k) at (-180/\npar+\k*360/\npar:1);
}  
\foreach[count=\i]\k in {0,...,\nminusone}{
\fill[black!10](0,0)--(p\k)--(p\i)--cycle;
}
\foreach[count=\i]\k in {0,...,\nminusone}{ 
\begin{scope}[rotate=\k*360/\npar]
\draw[dotted]
(0,0)--({cos(180/\npar)},0)
(0,0)--(180/\npar:1)
;
\end{scope}
}
\draw[very thick,FarbeA](-1,0)--({cos(180/\npar)},0)node[above left,inner sep=1pt]{$\Fix(\gamma_1)$};
\draw[very thick,FarbeB](180/\npar:1)--(180/\npar:{-cos(180/\npar)})node[pos=0.82,sloped,above,inner sep=1pt]{$\Fix(\gamma_2)$};
\foreach[count=\i]\k in {0,...,\nminusone}{
\draw[thick](p\k)--(p\i);
\draw plot[bullet](p\i)node[anchor={180-180/\npar+\i*360/\npar}]{$p_{\i}$}; 
} 
\end{tikzpicture}
\hfill
\pgfmathsetmacro{\npar}{6} 
\begin{tikzpicture}[line cap=round,line join=round,semithick,scale=\globalscale]
\pgfmathsetmacro{\nminusone}{\npar-1}
\foreach[count=\i]\k in {0,...,\npar}{
\coordinate(p\k) at (-180/\npar+\k*360/\npar:1);
}  
\foreach[count=\i]\k in {0,...,\nminusone}{
\fill[black!10](0,0)--(p\k)--(p\i)--cycle;
}
\foreach[count=\i]\k in {0,...,\nminusone}{ 
\begin{scope}[rotate=\k*360/\npar]
\draw[dotted]
(0,0)--({cos(180/\npar)},0)
(0,0)--(180/\npar:1)
;
\end{scope}
}
\draw[very thick,FarbeA]({-cos(180/\npar)},0)--({cos(180/\npar)},0)node[above left,inner sep=1pt]{$\Fix(\gamma_1)$};
\draw[very thick,FarbeB](180/\npar:1)--(180/\npar:-1)node[pos=0.82,sloped,above,inner sep=1pt]{$\Fix(\gamma_2)$};
\foreach[count=\i]\k in {0,...,\nminusone}{
\draw[thick](p\k)--(p\i);
\draw plot[bullet](p\i)node[anchor={180-180/\npar+\i*360/\npar}]{$p_{\i}$}; 
} 
\end{tikzpicture}
\caption{The dihedral group $\dih_k$ acting on a regular $k$-gon for $k=5$ and $k=6$.}
\label{fig:k-gon}
\end{figure}

\begin{example}[Dihedral symmetry]
\label{rem:dihedral_representations}
In the context of minimal surfaces in $\R^3$, the dihedral group often appears in the representation $\st_1\oplus \Or\colon\dih_k\to\SOgroup(3)$.
This means that $\dih_k$ acts on the $x_1x_2$-plane in the standard representation $\st_1$ and on the $x_3$-axis in the orientation representation $\Or$. 
Consequently, every reflection also flips the $x_3$-coordinate, acting as a rotation by $\pi$ around an axis in the $x_1x_2$-plane. 
This action of the dihedral group has proven to be extremely fruitful for the variational construction of equivariant minimal surfaces in various ambient manifolds 
\cite{Ketover2016FBMS,CarlottoFranzSchulz2022,BuzanoNguyenSchulz2021,FranzSchulz2023,SchulzEllipsoids,SchulzToroids}, 
not least because it acts by orientation-preserving isometries of the ambient space. 
Indeed, for every $\gamma\in\dih_k$ we have 
\(
\det(\st_1(\gamma))
=\Or(\gamma)
\)
and thus
\[
\det\bigl((\st_1\oplus\Or)(\gamma)\bigr)
=\Or(\gamma)\Or(\gamma)=1.
\]
Allowing for orientation-reversing ambient isometries, the dihedral group also appears in the context of minimal surfaces with \emph{prismatic} respectively \emph{antiprismatic} symmetry, see Definition~\ref{defn:prismatic} respectively Definition~\ref{defn:antiprismatic}, and the examples in 
\cite{FranzKetoverSchulz,FolhaPacardZolotareva2017} respectively 
\cite{CSWnonuniqueness,CSWstackings}. 
\end{example}

\paragraph{Characters.}
Given a representation $\rho\colon G\to\GLgroup(V)$ of a finite group $G$ on a finite-dimensional real vector space $V$, its \emph{character} is the function 
\begin{align}
\label{eqn:character}
\begin{aligned}
\chi_{\rho}\colon G&\to\R
\\
\gamma&\mapsto\tr\bigl(\rho(\gamma)\bigr).
\end{aligned}
\end{align}
Then $\chi_{\rho}(e)=\dim V$ for the neutral element $e\in G$ and, since the trace is invariant under conjugation, $\chi_{\rho}$ is constant on conjugacy classes of $G$. 
Characters are additive under direct sums, and $\chi_{\rho\otimes\varsigma}=\chi_{\rho}\chi_\varsigma$ for any representation $\varsigma$ as in Definition~\ref{defn:tensor_representation}. 

Given two representations $\rho,\rho'$ of $G$ we define 
\begin{align}
\label{eqn:character-pairing}
\brt{\chi_{\rho},\chi_{\rho'}}
&\vcentcolon=\frac{1}{\abs{G}}\sum_{\gamma\in G}\chi_{\rho}(\gamma)\chi_{\rho'}(\gamma).
\end{align}
Suppose that $\brt{\chi_r,\chi_r}=1$ for every $r\in\IrRep{G}$. 
Then by \cite[Theorems~4--5]{Serre1977} the multiplicity $m_r$ of any $r\in\IrRep{G}$ in any representation $\rho$ of $G$ is given by 
\begin{align}\label{eqn:20260318}
m_r=\brt{\chi_\rho,\chi_r}.
\end{align} 
The assumption is satisfied for $G=\dih_k$. 
Indeed, if $r\in\{\triv,\Or,\m_1,\m_2\}$, then $\chi_r(\gamma)\in\{\pm1\}$ for all $\gamma\in\dih_k$ and hence $\brt{\chi_r,\chi_r}=1$. 
Moreover, \eqref{eqn:St-matrix} implies $\chi_{\st_j}(\gamma)=0$ for every reflection $\gamma\in\dih_k$, while $\chi_{\st_j}\bigl((\gamma_1\gamma_2)^s\bigr)=2\cos(2\pi js/k)$ for $0\leq s<k$. 
Since $\sum_{s=0}^{k-1}\cos(4\pi js/k)=0$ for $1\leq j\leq\lfloor(k-1)/2\rfloor$, we obtain 
\[
\brt{\chi_{\st_j},\chi_{\st_j}}=\frac{1}{2k}\sum_{s=0}^{k-1}4\cos^2(2\pi js/k)=\frac{1}{2k}\sum_{s=0}^{k-1}\bigl(2+2\cos(4\pi js/k)\bigr)=1.
\]
 
The following statement is a variation on \cite[Theorem~1.3]{Helffer2002} and isolates the nodal domain argument used in the following sections. 

\begin{proposition}
\label{prop:nodal_dihedral}
Given $3\leq k\in\N$, let $\Omega$ be a compact, connected surface with an isometric action of $\dih_k$ such that the fixed point sets of the $k$ reflections in $\dih_k$ divide $\Omega$ into $2k$ copies of a fundamental domain, cyclically ordered around a common fixed point of $\dih_k$. 
Assume also that the generators $\gamma_1,\gamma_2\in\dih_k$ fix adjacent sides of a fundamental domain, see Figure~\ref{fig:20260820}. 
Let $\lambda_1<\lambda_2$ denote the first two distinct eigenvalues of the spectral quadratic form \eqref{eqn:spectral-index-form} with $\dih_k$-invariant potentials $q$ and $\sigma$, restricted to functions vanishing on a (possibly empty) $\dih_k$-invariant subset of $\partial\Omega$, and let $\mathcal E_{\lambda_i}$ be the isomorphism class of the induced representation of $\dih_k$ on the $i$th eigenspace. 
Then, for some integers $a,b\geq0$,
\begin{align*}
\mathcal E_{\lambda_1}&=\triv, & 
\mathcal E_{\lambda_2}&=a\triv\oplus b\st_1.
\end{align*}
\end{proposition}

\begin{proof}
Since the action of $\dih_k$ preserves the metric, the potentials, and the boundary conditions, each eigenspace $E_{\lambda_i}$ is a $\dih_k$-invariant subspace, so the representation classes $\mathcal{E}_{\lambda_i}$ are well-defined. 
The first eigenvalue $\lambda_1$ on the connected domain $\Omega$ is simple with an eigenfunction $u_1$ of constant sign in the interior of $\Omega$. 
In particular, 
\begin{align}\label{eqn:lambda1simple}
\dim E_{\lambda_1}=1.
\end{align}
For any $\gamma\in\dih_k$, the function $u_1\circ \gamma^{-1}$ is again an eigenfunction with eigenvalue $\lambda_1$, so simplicity implies $u_1 \circ \gamma^{-1} = u_1$, which proves $\mathcal{E}_{\lambda_1}=\triv$.

Given $\varrho\in\IrRep{\dih_k}$ with $\varrho\leq\mathcal{E}_{\lambda_2}$, let $u\in E_{\lambda_2}$ be a nonzero function in a $\dih_k$-invariant subspace of $E_{\lambda_2}$ on which the induced representation has class $\varrho$.
Being orthogonal to the function $u_1$ of constant sign, $u$ has at least two nodal domains. 
If $u$ had more than two nodal domains, \cite[Proposition~3.1\,(i)]{CSWSpectral} would imply $\dim E_{\lambda_1}\geq2$, contradicting \eqref{eqn:lambda1simple}. 
Thus, $u$ has exactly two nodal domains.

We prove the claim $\varrho \in \{\triv, \st_1\}$ by ruling  out the other classes as follows. 
If $\varrho = \Or$, then for every reflection $\gamma \in \dih_k$ and every $p\in\Fix(\gamma)$,
\[
u(p) = u(\gamma^{-1} p) = \Or(\gamma) u(p) = -u(p),
\]
so $u(p)=0$. 
The assumptions on $\Omega$ then imply that $u$ has at least $2k$ nodal domains, and \cite[Proposition~3.1\,(i)]{CSWSpectral} would imply $\dim E_{\lambda_1}\geq2k-1\geq5$ in contradiction with \eqref{eqn:lambda1simple}.

If $k$ is even and $\varrho \in \{\m_1, \m_2\}$, then $u$ is odd under the $k/2$ reflections conjugate to the corresponding generator. 
Their fixed point sets divide $\Omega$ into $k$ sectors, so $u$ has at least $k \geq 4$ nodal domains which is again a contradiction. 

If $\varrho=\st_j$ for some $j\geq2$ (which requires $k\geq5$), we may choose a basis $(v,u)$ of a $\dih_k$-invariant subspace $W\subset E_{\lambda_2}$ on which the induced representation has class $\st_j$, such that for all $\gamma\in\dih_k$ and all $x\in\Omega$ 
\begin{align}\label{eqn:Stj}
\begin{pmatrix}v(\gamma x)\\u(\gamma x)\end{pmatrix} 
&=\st_j(\gamma)
\begin{pmatrix}v(x)\\u(x)\end{pmatrix} . 
\end{align}
Then $v$ is $\gamma_1$-even and $u$ is $\gamma_1$-odd, so $u$ vanishes on $\Fix(\gamma_1)$ which separates $\Omega$ into two parts $\Omega^{\pm}$. 
Since the nodal set of $u$ contains $\Fix(\gamma_1)$ and $u$ has exactly two nodal domains, these are given by $\Omega^{+}$ and $\Omega^{-}$. 
Since $v$ is $\gamma_1$-even, its normal derivative along $\Fix(\gamma_1)$ vanishes. 
If $v$ vanished identically on $\Fix(\gamma_1)$, the unique continuation principle would imply $v \equiv 0$. 
Thus, there exists an interior point $p \in \Fix(\gamma_1)$ with $v(p) \neq 0$.

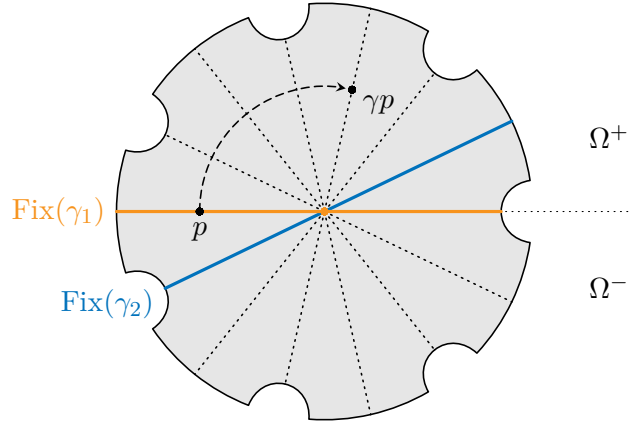
\begin{figure}
\centering
\pgfmathsetmacro{\npar}{7}
\pgfmathsetmacro{\spar}{ceil((\npar)/4)}
\begin{tikzpicture}[line cap=round,line join=round,semithick,scale=2.75]
\pgfmathsetmacro{\nminusone}{\npar-1}
\pgfmathsetmacro{\wA}{60/\npar} 
\foreach[count=\i]\k in {0,...,\npar}{
\coordinate(L\k) at (\k*360/\npar+\wA:1);
\coordinate(R\k) at (\k*360/\npar-\wA:1);
} 
\draw[thin,dotted](-1,0)--(1.5,0)
node[anchor=-70,inner sep=5ex]{$\Omega^+$}
node[anchor=70,inner sep=5ex]{$\Omega^-$}
;
\foreach[count=\i]\k in {0,...,\nminusone}{
\begin{scope}[rotate=\k*360/\npar]
\fill[black!10]
(0:0.85)to[out=90,in=180+\wA](\wA:1)arc(\wA:180/\npar:1)
--(0,0)--cycle;
\fill[black!10,yscale=-1]
(0:0.85)to[out=90,in=180+\wA](\wA:1)arc(\wA:180/\npar:1)
--(0,0)--cycle;
\end{scope}
}
\foreach[count=\i]\k in {0,...,\nminusone}{
\begin{scope}[rotate=\k*360/\npar]
\draw[dotted](0,0)--(0:0.85)
(0,0)--(180/\npar:1);
\draw 
(0:0.85)to[out=90,in=180+\wA](\wA:1)arc(\wA:180/\npar:1);
\draw[yscale=-1]
(0:0.85)to[out=90,in=180+\wA](\wA:1)arc(\wA:180/\npar:1);
\end{scope}
}
\draw[very thick,FarbeB,rotate=180/\npar](-0.85,0)--(1,0)node[pos=0,anchor=15]{$\Fix(\gamma_2)$};
\draw[very thick,FarbeA](0.85,0)--(-1,0)node[left]{$\Fix(\gamma_1)$};
\draw[FarbeA] plot[bullet](0,0);
\draw[] plot[bullet](180:0.6)node[below]{$p$}coordinate(p);
\draw[] plot[bullet](180-360*\spar/\npar:0.6)node[anchor=150]{$\gamma p$};
\draw[-stealth,densely dashed](p)arc(180:180-360*\spar/\npar+3:0.6);
\end{tikzpicture}
\caption{Surface $\Omega$ with an $\dih_k$-action dividing it into $2k$ fundamental domains for $k=\npar$, \\
and a point $p\in\Omega$ rotated by $\gamma=(\gamma_1\gamma_2)^s$ for $s=2$.}
\label{fig:20260820}
\end{figure}

Given an integer $1\leq s<k/2$, let 
$\gamma\vcentcolon=(\gamma_1\gamma_2)^s\in\dih_k$.
On the one hand, $\gamma$ acts on $\Omega$ as the rotation by angle $-2\pi s/k\in\interval{-\pi,0}$. 
Since $p\in \Fix(\gamma_1)$ we have $\gamma p\in\Omega^+$ for any choice of $1\leq s<k/2$ 
(up to relabeling the domains, see Figure~\ref{fig:20260820}). 
Therefore, $u(\gamma p)\neq0$ has the same sign for all integers $1\leq s<k/2$. 
On the other hand, equation \eqref{eqn:Stj} and the definition of $\st_j$ imply 
\begin{align}\label{eqn:20260814}
u(\gamma p) = -\sin\Bigl(\frac{2\pi sj}{k}\Bigr)v(p).
\end{align}
Since $j\geq2$ the right hand side of \eqref{eqn:20260814} changes sign depending on $1\leq s<k/2$. 
Indeed, since $j<k/2$ we have $0<2\pi sj/k<\pi$ for $s=1$ and $\pi\leq 2\pi sj/k<2\pi$ for $s=\lceil\frac{k}{2j}\rceil$. 
This contradiction concludes the proof.
\end{proof}

\section{Surfaces with prismatic symmetry}
\label{sec:conformal}

The symmetry group of the minimal surfaces considered in this section has the structure of a direct product $G=H\times K$. 

\begin{definition}\label{defn:boxtimes}
Let $\rho_1 \colon H \to \Ogroup(V_1)$ and $\rho_2 \colon K \to \Ogroup(V_2)$ be representations of the finite groups $H$ and $K$. We denote by $\rho_1 \boxtimes \rho_2$ the representation of the group $G=H \times K$ on the tensor product space $V_1 \otimes V_2$ defined for all $(h, k) \in H \times K$ by 
\begin{align*} 
(\rho_1 \boxtimes \rho_2)(h, k)(v_1 \otimes v_2) \vcentcolon= \rho_1(h)v_1 \otimes \rho_2(k)v_2.
\end{align*} 
This extends associatively to any finite direct product $H_1\times\cdots\times H_m$ of finite groups and descends to isomorphism classes of representations.
\end{definition}
 
In analogy with the notation introduced in Section~\ref{sec:dihedral} 
we denote the two irreducible representations of the group $\Z_2=\{1, \tau\}$ by 
\begin{align*}
\triv\colon\Z_2&\to\Ogroup(1)\cong\{\pm1\}, 
& 
\Or\colon\Z_2&\to\Ogroup(1), \\
\tau&\mapsto 1, & 
\tau&\mapsto-1. 
\end{align*}
The symmetry group of a regular $k$-gonal prism
(cf.~\cite[Figure~2]{CSWnonuniqueness} and \cite[Table~2]{FranzKetoverSchulz})
can then be formalised as follows.

\begin{definition}[Prismatic symmetry]
\label{defn:prismatic}
Given $2\leq k\in\N$, the \emph{prismatic} group of order $4k$ is the direct product $\dih_k\times\Z_2$ acting on $\R^3$ in the representation 
\begin{align}
\label{eqn:prismatic}
(\st_1\boxtimes\triv)\oplus(\triv\boxtimes\Or)
\colon \dih_k\times\Z_2 
&\to\Ogroup(3).
\end{align}
\end{definition}

In standard Cartesian coordinates $(x_1,x_2,x_3)$ of $\R^3$, 
Definition \ref{defn:prismatic} means that $\dih_k$ acts in the standard representation $\st_1$ on the $x_1x_2$-plane and trivially on the $x_3$-axis, while the nontrivial element $\tau\in\Z_2$ acts as the reflection $(x_1,x_2,x_3)\mapsto(x_1,x_2,-x_3)$ across the $x_1x_2$-plane. 

Given $m\in\N$, every irreducible representation $\varrho$ of $\dih_k\times(\Z_2)^m$ is isomorphic to an external tensor product of an irreducible representation of $\dih_k$ with a copy of $\triv$ or $\Or$ for each $\Z_2$-factor. 
Indeed, for any $\gamma\in\dih_k\times(\Z_2)^m$ and the generator $\tau$ of any factor $\Z_2$ one has $\gamma\tau=\tau\gamma$, so that the operator $\varrho(\gamma)$ preserves the eigenspaces of the orthogonal involution $\varrho(\tau)$. 
As $\varrho$ is irreducible, $\varrho(\tau)=\pm\mathrm{Id}$ for each such $\tau$, which proves the claim. 

The character of $\varrho\boxtimes\varsigma$ with $\varrho\in\IrRep{\dih_k}$ and $\varsigma\in\IrRep{(\Z_2)^m}$ is given by $\chi_{\varrho\boxtimes\varsigma}(\gamma,\gamma')=\chi_\varrho(\gamma)\chi_\varsigma(\gamma')$, 
as the trace of a tensor product of linear maps is the product of their traces.  
Consequently, the pairing \eqref{eqn:character-pairing} factorises, and $\brt{\chi_r,\chi_r}=1$ holds for every $r\in\IrRep{\dih_k\times(\Z_2)^m}$, so that \eqref{eqn:20260318} applies.

Proposition~\ref{prop:rigidity} identifies the energy nullity with $\mathfrak{so}(\cA_\rho)$, which decomposes as follows.

\begin{lemma}
\label{lem:so-Lemma}
Let $\rho\colon G\to\Ogroup(3)$ be the prismatic representation \eqref{eqn:prismatic}. 
Then,   
\[
\mathfrak{so}(\cA_\rho) = (\Or\boxtimes\triv)\oplus(\st_1\boxtimes\Or).
\]
\end{lemma}

\begin{proof} 
By Definition~\ref{defn:so_A_rho}, $\mathfrak{so}(\cA_\rho)$ is the isomorphism class of the representation on $\mathfrak{so}(3)$ given by conjugation $B\mapsto\rho(\gamma)B\rho(\gamma)^{-1}$. 
Identifying $B\in\mathfrak{so}(3)$ with the alternating bilinear form $(u,v)\mapsto\brt{Bu,v}$, this representation becomes the representation $u\wedge v\mapsto\rho(\gamma)u\wedge\rho(\gamma)v$ on the exterior square $\Lambda^2\R^3$ spanned by $e_1\wedge e_2$, $e_1\wedge e_3$, $e_2\wedge e_3$, where $e_1,e_2,e_3$ denotes the standard basis of $\R^3$. 
\begin{itemize}
\item 
The element $\tau\in\Z_2$ preserves $e_1\wedge e_2$, while any $\gamma\in\dih_k$ acts on $e_1\wedge e_2$ by multiplication with $\det\st_1(\gamma)=\Or(\gamma)$, which yields the summand $\Or\boxtimes\triv$. 
\item 
On the span of $e_1\wedge e_3$ and $e_2\wedge e_3$, the map $v\wedge e_3\mapsto v$ is a $\dih_k$-equivariant isomorphism onto the $x_1x_2$-plane, while $\tau$ acts by $-1$, which yields the summand $\st_1\boxtimes\Or$. \qedhere
\end{itemize}
\end{proof}

\subsection{\texorpdfstring{Free boundary minimal $k$-noids in the unit ball}
{Free boundary minimal k-noids in the unit ball}}

In analogy with the complete minimal surfaces in $\R^3$ of genus zero with $k$ catenoidal ends and dihedral symmetry constructed by Jorge and Meeks \cite[§\,5]{JorgeMeeks1983} and called $k$-noids in \cite[(2.3.5)]{Karcher1989}, we also call a compact surface of genus zero with $k$ boundary components, equipped with a Riemannian metric $g$ and an isometric action of a finite group $G$, a \emph{$k$-noid}.
 
In this section we prove Theorem~\ref{thm:index-fbm-k-noid} stating that any properly embedded free boundary minimal $k$-noid in $\B^3$, which is invariant under the prismatic representation \eqref{eqn:prismatic} of $G=\dih_k\times\Z_2$, has Morse index equal to $2k$ and nullity $3$. 
The key idea is to combine Theorem~\ref{thm:equivariant_area_bounds} with the dimensions of the isotypic components of $\DF(\Sigma_k)$ computed in the following theorem. 

\begin{theorem}
\label{thm:k-noidDF}
Given $3\leq k\in\N$, let $\Sigma_k\subset\B^3$ be a properly embedded free boundary minimal $k$-noid which is $G=\dih_k\times\Z_2$-equivariant in the prismatic representation \eqref{eqn:prismatic}. 
Then the induced representation of $G$ on the space $\DF(\Sigma_k)$ of transverse-traceless tensors has the following dimensions of isotypic components. 
\begin{enumerate}[label={\normalfont(\roman*)}]
\item\label{thm:k-noidDF-i} The $\tau$-invariant subspace $\bigoplus_{\varrho\in\IrRep{\dih_k}}\DF^{\varrho\boxtimes\triv}(\Sigma_k)$ has dimension $2k-3$. 
\item\label{thm:k-noidDF-ii} $\dim\DF^{\st_1\boxtimes\Or}(\Sigma_k)=0$.
\item\label{thm:k-noidDF-iii} $\dim\DF^{\triv\boxtimes\Or}(\Sigma_k)=0$.
\item\label{thm:k-noidDF-iv} $\dim\DF^{\triv\boxtimes\triv}(\Sigma_k)=1$.
\end{enumerate}
\end{theorem}

\begin{proof} 
By Lemma~\ref{lem:chi_DF} stated in the appendix, we may compute the character $\chi_{\mathrm{tt}}$ of the representation of $G$ on $\DF(\Sigma_k)$ using the Eichler trace formula stated in Theorem~\ref{thm:Eichler} on the Schottky double $\hat{\Sigma}_k$ of genus $k-1\geq2$. 
As in \eqref{eqn:20260927}, let $\iota\colon\hat{\Sigma}_k\to\hat{\Sigma}_k$ be the involution on the Schottky double. 
We apply Theorem~\ref{thm:Eichler} with $T=\gamma$ if $\gamma\in G$ preserves the orientation of $\Sigma_k$ and with $T=\gamma\iota$ otherwise, as in Remark~\ref{rem:Eichler}. 

The group generators $\gamma_1$, $\gamma_2$ and $\tau$ all reverse the orientation of $\Sigma_k$, because they fix arcs in $\Sigma_k$ by Corollary~\ref{cor:half_knoid}\,\ref{cor:half_knoid-i}--\ref{cor:half_knoid-ii}, while nontrivial orientation-preserving isometries have isolated fixed points. 
Hence, the rotations $(\gamma_1\gamma_2)^s$ and the half-turns $\gamma\tau$ with reflections $\gamma\in\dih_k$ preserve the orientation, while the reflections $\gamma\in\dih_k$, the element $\tau$ and the rotoreflections $(\gamma_1\gamma_2)^s\tau$ reverse it. 

Let $\gamma=\tau$. 
Since $\tau\iota$ exchanges the interiors of $\Sigma_k$ and its copy in $\hat{\Sigma}_k$, its fixed points are the $2k$ points of $\partial\Sigma_k\cap\{x_3=0\}$ given by Corollary~\ref{cor:half_knoid}\,\ref{cor:half_knoid-iii}. 
At each of them, $d\tau$ fixes the inward normal and reverses the orientation, so $\tau\iota$ looks like $z\mapsto-z$ by \eqref{eqn:20260927}. 
Hence, by Theorem~\ref{thm:Eichler},
\[
\chi_{\mathrm{tt}}(\tau)=2k\frac{(-1)^2}{1-(-1)}=k. 
\]
Let $\gamma\in\dih_k$ be a reflection. 
The fixed point set of $\rho(\gamma)$ is a vertical plane consisting of two half-planes bounded by the $x_3$-axis. 
By Corollary~\ref{cor:half_knoid}\,\ref{cor:half_knoid-ii} and $\tau$-invariance, each such half-plane either intersects $\partial\Sigma_k$ in two points and does not intersect $\Sigma_k\cap\{x_3=0\}$, or it does not intersect $\partial\Sigma_k$ and intersects $\Sigma_k\cap\{x_3=0\}$ in a single interior point. 
These two types alternate around the $x_3$-axis, so the two half-planes forming $\Fix(\rho(\gamma))$ are of different type if $k$ is odd and of the same type if $k$ is even. 

The half-turn $\gamma\tau$ is a rotation by $\pi$ around the horizontal line $\Fix(\rho(\gamma))\cap\{x_3=0\}$. 
At each of its fixed points in $\Sigma_k$ and their copies, it looks like $z\mapsto-z$. 
Similarly, $\gamma\iota$ looks like $z\mapsto-z$ around its fixed points, which are the points of $\Fix(\gamma)\cap\partial\Sigma_k$. 
If $k$ is odd, then $\gamma\iota$ has two fixed points and $\gamma\tau$ has one fixed point in $\Sigma_k$ and one on its copy, so that 
\begin{align}
\label{eqn:chitt1}
\chi_{\mathrm{tt}}(\gamma)=\chi_{\mathrm{tt}}(\gamma\tau)=2\frac{(-1)^2}{1-(-1)}=1. 
\end{align}
If $k$ is even, there are two conjugacy classes of reflections in $\dih_k$ (cf.~Figure~\ref{fig:k-gon}). 
If $\Fix(\gamma)$ intersects $\partial\Sigma_k$, then $\gamma\iota$ has $4$ fixed points and $\gamma\tau$ has none, so that $\chi_{\mathrm{tt}}(\gamma)=2$ and $\chi_{\mathrm{tt}}(\gamma\tau)=0$. 
Otherwise, $\gamma\iota$ has no fixed points and $\gamma\tau$ has $4$ fixed points, two on each copy of $\Sigma_k$, so that $\chi_{\mathrm{tt}}(\gamma)=0$ and $\chi_{\mathrm{tt}}(\gamma\tau)=2$.

For $1\leq s<k$, the rotation $(\gamma_1\gamma_2)^s$ by the angle $\theta_s\vcentcolon=2\pi s/k$ has four fixed points: the interior points $q,\tau(q)\in\Sigma_k$ on the $x_3$-axis given by Corollary~\ref{cor:half_knoid}\,\ref{cor:half_knoid-i} and their copies on the double $\hat\Sigma_k$. 
Since the tangent planes of $\Sigma_k$ at $q$ and $\tau(q)$ are horizontal, $(\gamma_1\gamma_2)^s$ has the form $z\mapsto e^{\pm i\theta_s}z$ on $\Sigma_k$ and $z\mapsto e^{\mp i\theta_s}z$ on the copy. 
Hence 
\[
\chi_{\mathrm{tt}}\bigl((\gamma_1\gamma_2)^s\bigr)=4\operatorname{Re}\biggl(\frac{e^{2i\theta_s}}{1-e^{i\theta_s}}\biggr)=-2-4\cos\theta_s. 
\]
Furthermore, $(\gamma_1\gamma_2)^s\tau\iota$ has no fixed points, since $\rho\bigl((\gamma_1\gamma_2)^s\tau\bigr)$ fixes only the origin, which does not lie on $\partial\Sigma_k\subset\partial\B^3$, so that 
\[
\chi_{\mathrm{tt}}\bigl((\gamma_1\gamma_2)^s\tau\bigr)=0.
\]
Finally, $\chi_{\mathrm{tt}}(e)=\dim\DF(\Sigma_k)=3k-6$ by Lemma~\ref{lem:DF_quadratic}. 

Since $\tau$ acts on $\DF(\Sigma_k)$ as an involution, the dimensions $a$ and $b$ of its eigenspaces for the eigenvalues $1$ and $-1$ satisfy 
\begin{align*}
a+b&=\chi_{\mathrm{tt}}(e)=3k-6, &
a-b&=\chi_{\mathrm{tt}}(\tau)=k,
\end{align*}
hence the $\tau$-invariant subspace has dimension $a=2k-3$, which proves~\ref{thm:k-noidDF-i}.

To prove~\ref{thm:k-noidDF-ii} and~\ref{thm:k-noidDF-iii}, we recall from \eqref{eqn:20260318} that the multiplicity of any $r\in\IrRep{G}$ in the representation of $G$ on $\DF(\Sigma_k)$ is given by 
\[
m_r=\frac{1}{4k}\sum_{g\in G}\chi_{\mathrm{tt}}(g)\chi_r(g).
\]
Every element of $G$ is of the form $\gamma$ or $\gamma\tau$ with $\gamma\in\dih_k$, and for $\varrho\in\IrRep{\dih_k}$ we have $\chi_{\varrho\boxtimes\Or}(\gamma)=\chi_\varrho(\gamma)$ and $\chi_{\varrho\boxtimes\Or}(\gamma\tau)=-\chi_\varrho(\gamma)$. 
Hence 
\begin{align}\label{eqn:mult_knoid}
4k\,m_{\varrho\boxtimes\Or}=\sum_{\gamma\in\dih_k}\chi_\varrho(\gamma)\bigl(\chi_{\mathrm{tt}}(\gamma)-\chi_{\mathrm{tt}}(\gamma\tau)\bigr). 
\end{align}
Since $\sum_{s=0}^{k-1}e^{i\theta_s}=0$ and, as $e^{2i\theta_1}\neq1$ for $k\geq3$, also $\sum_{s=0}^{k-1}e^{2i\theta_s}=0$, we have 
\begin{align*}
\sum_{s=1}^{k-1}\cos\theta_s&=-1, &
\sum_{s=1}^{k-1}\cos^2\theta_s&=\sum_{s=1}^{k-1}\frac{1+\cos(2\theta_s)}{2}=\frac{k-2}{2}.
\end{align*}
For $\varrho=\st_1$, only the rotations in $\dih_k$ contribute to \eqref{eqn:mult_knoid}, because $\chi_{\st_1}$ vanishes on reflections and $\chi_{\st_1}\bigl((\gamma_1\gamma_2)^s\bigr)=2\cos\theta_s$. 
Inserting the values of $\chi_{\mathrm{tt}}$ computed above, we obtain 
\[
4k\,m_{\st_1\boxtimes\Or}=2\bigl((3k-6)-k\bigr)+\sum_{s=1}^{k-1}2\cos\theta_s\bigl(-2-4\cos\theta_s\bigr)=(4k-12)+4-4(k-2)=0, 
\]
which proves~\ref{thm:k-noidDF-ii}. 
For $\varrho=\triv$, the rotations contribute 
\[
\bigl((3k-6)-k\bigr)+\sum_{s=1}^{k-1}(-2-4\cos\theta_s)=(2k-6)-2(k-1)+4=0
\]
 to \eqref{eqn:mult_knoid}. 
The reflections contribute $1-1=0$ each if $k$ is odd, while for even $k$ half of them contribute $2-0$ and the other half $0-2$. 
Hence $m_{\triv\boxtimes\Or}=0$, which proves~\ref{thm:k-noidDF-iii}.

Analogously, $4k\,m_{\triv\boxtimes\triv}=\sum_{\gamma\in\dih_k}\bigl(\chi_{\mathrm{tt}}(\gamma)+\chi_{\mathrm{tt}}(\gamma\tau)\bigr)$, where the identity contributes $(3k-6)+k$, the nontrivial rotations contribute $\sum_{s=1}^{k-1}(-2-4\cos\theta_s)=6-2k$, and each of the $k$ reflections contributes $2$. 
Hence $m_{\triv\boxtimes\triv}=1$, which proves~\ref{thm:k-noidDF-iv}.
\end{proof}

\begin{remark}
\label{rem:knoid-halfsphere}
Theorem~\ref{thm:k-noidDF} holds verbatim for any properly embedded free boundary minimal $k$-noid in the half-sphere $\{x\in\Sp^3\mid x_4\geq0\}$ which is $G=\dih_k\times\Z_2$-equivariant in the prismatic representation \eqref{eqn:prismatic}, where $G$ acts trivially on the $x_4$-coordinate, because the proof only relies on Lemma~\ref{lem:fundamental_domain_knoid} and Corollary~\ref{cor:half_knoid}, which extend to this setting by Remark~\ref{rem:half-sphere}. 
\end{remark}

\begin{lemma}
\label{lem:spec_ener_knoid}
Given $3\leq k\in\N$, let $\Sigma_k\subset\B^3$ be any properly embedded free boundary minimal $k$-noid which is $\dih_k\times\Z_2$-equivariant in the prismatic representation \eqref{eqn:prismatic}. 
Then 
\begin{align*}
\Ind_\area(\Sigma_k)&=\Ind_\jacobi(\Sigma_k), &
\Nul_\area(\Sigma_k)&=\Nul_\jacobi(\Sigma_k), &
\ind_\spec(\Sigma_k)&=1. 
\end{align*} 
\end{lemma}

\begin{proof}
Since $\Sigma_k\subset\B^3$ is properly embedded, it is two-sided and admits a global unit normal field $\nu$. 
Being smooth, connected, invariant under the action of the generator $\tau\in\Z_2$, and not a disc, $\Sigma_k$ intersects $\Fix(\rho(\tau))=\{x_3=0\}$ orthogonally, so $\nu$ is preserved by $\rho(\tau)$. 
By Corollary~\ref{cor:half_knoid}\,\ref{cor:half_knoid-i}, the $x_3$-axis intersects $\Sigma_k$ orthogonally at the point $q\in\Sigma_k$, so that $\nu(q)=\pm e_3$. 
Since $q$ and $e_3$ are both fixed by $\rho(\gamma)$ for every $\gamma\in\dih_k$, the unit normal field $\nu$ is also preserved. 
Thus, the relative orientation representation $\epsilon$ (see Definition~\ref{defn:relative_orientation_representation}) is trivial and 
Proposition~\ref{prop:indJ_indA} yields the equivariant isomorphisms $\Ind_\area(\Sigma_k)=\Ind_\jacobi(\Sigma_k)$ and $\Nul_\area(\Sigma_k)=\Nul_\jacobi(\Sigma_k)$.

The spectral quadratic form \eqref{eqn:spectral-index-form} on the free boundary minimal surface $(\Sigma_k,g)$ in $\B^3$ is given by
\begin{align}\label{eqn:20260826}
Q_\spec(u,u) = \int_{\Sigma_k} \abs{du}_g^2 - \int_{\partial \Sigma_k} u^2.
\end{align}
Since $Q_\spec(1,1)< 0$ we have $\ind_\spec(\Sigma_k) \geq 1$.
Let $W \vcentcolon=\{ u \in C^\infty(\Sigma_k) \mid \int_{\partial \Sigma_k} u = 0 \}$ be the subspace of functions with mean value zero on the boundary.
The first Steklov eigenvalue  
\begin{align}\label{eqn:sigma1def}
\sigma_1(\Sigma_k) = \inf_{u \in W \setminus \{0\}} \frac{\int_{\Sigma_k} \abs{du}_g^2}{\int_{\partial \Sigma_k} u^2}
\end{align}
satisfies $\sigma_1(\Sigma_k) = 1$ by \cite[Theorem 4.2]{McGrath2018}, which applies due to Lemma~\ref{lem:fundamental_domain_knoid} stated in the appendix.
(Alternatively, one could apply the recent result~\cite{ChodoshGianocca2026}.) 
Consequently, $Q_\spec(u,u) \geq 0$ for all $u \in W$.
Because $Q_\spec$ is positive semi-definite on the codimension-one subspace $W$, the maximal dimension of any subspace on which $Q_\spec$ is negative definite cannot exceed $1$.
Therefore, $\ind_\spec(\Sigma_k) = 1$ as claimed. 
\end{proof}

To efficiently track the isotypic components associated with the $\Z_2$ factor of the symmetry group, we introduce the following parity notation, following \cite[Example~2.3]{CSWSpectral}.

\begin{definition}[Parity index and nullity] 
\label{defn:parity_single}
Recalling the general notation for isotypic components from Definition~\ref{defn:Isotypic}, we define, for each $\mathord\bullet\in\{\ener,\area,\jacobi\}$, the parity indices
\begin{align*}
\ind^+_\bullet(\Sigma_k)
&\vcentcolon=\sum_{\varrho\in \IrRep{\dih_k}} 
\ind^{\varrho\boxtimes\triv}_\bullet(\Sigma_k), & 
\ind^-_\bullet(\Sigma_k)
&\vcentcolon=\sum_{\varrho\in \IrRep{\dih_k}} 
\ind^{\varrho\boxtimes\Or}_\bullet(\Sigma_k), 
\intertext{and similarly the nullities $\nul^\pm_\bullet(\Sigma_k)$. 
Furthermore, we define the corresponding parity components of the space of transverse-traceless tensors as}
\DF^+(\Sigma_k)
&\vcentcolon=\bigoplus_{\varrho\in \IrRep{\dih_k}} 
\DF^{\varrho\boxtimes\triv}(\Sigma_k), & 
\DF^-(\Sigma_k)
&\vcentcolon=\bigoplus_{\varrho\in \IrRep{\dih_k}} 
\DF^{\varrho\boxtimes\Or}(\Sigma_k).
\end{align*}
\end{definition}

Concretely, $\ind^\pm_\jacobi(\Sigma_k)$ is the maximal dimension of a subspace of functions $f\in C^\infty(\Sigma_k)$ satisfying $f \circ \tau = \pm f$ on which the Jacobi quadratic form $Q_\jacobi$ is negative definite.

\begin{proposition}[Index]
\label{prop:knoid_isotypic}
Given $3\leq k\in\N$, let $\Sigma_k\subset\B^3$ be any properly embedded free boundary minimal $k$-noid which is $\dih_k\times\Z_2$-equivariant in the prismatic representation $\rho$ given in \eqref{eqn:prismatic}. 
Then $\ind_\ener(\Sigma_k)=3$ and $\ind_\area(\Sigma_k)=2k$. 
Moreover,  
\begin{align*}
&\begin{aligned}
\ind_\ener^+(\Sigma_k)&=2, \\
\ind_\area^+(\Sigma_k)&=2k-1, 
\end{aligned}
&
\ind_\area^{\triv\boxtimes\triv}(\Sigma_k)&=1,
&
\ind_\area^{\varrho \boxtimes \Or}(\Sigma_k)&=\begin{cases}
1 & \text{ if $\varrho=\triv$,}\\
0 & \text{ otherwise.}
\end{cases}
\end{align*}
\end{proposition}

\begin{proof}
Since $\ind_\spec(\Sigma_k) = 1$ by Lemma~\ref{lem:spec_ener_knoid}, Corollary~\ref{cor:indE} yields $\Ind_\ener(\Sigma_k)\leq\cA_\rho$. 
Of the two summands in $\rho=(\st_1\boxtimes\triv)\oplus(\triv\boxtimes\Or)$, the first contributes to the parity-even part in the sense of Definition~\ref{defn:parity_single} and has degree $2$, the second to the parity-odd part and has degree $1$.
Consequently, 
\begin{align}
\label{eqn:knoid_ener_index_bounds}
\ind_\ener^+(\Sigma_k)&\leq 2, &
\ind_\ener^{\triv\boxtimes\Or}(\Sigma_k)
\leq \ind_\ener^-(\Sigma_k)
 &\leq 1, & 
\ind_\ener^{\st_1\boxtimes\Or}(\Sigma_k) &= 0. 
\end{align}
By Theorem~\ref{thm:k-noidDF}\,\ref{thm:k-noidDF-i} we have 
\(\dim \DF^+(\Sigma_k) = 2k-3\). 
Theorem~\ref{thm:equivariant_area_bounds} and \eqref{eqn:knoid_ener_index_bounds} then imply 
\begin{align}
\label{eqn:plus_upperbound_knoid}
\ind_\area^+(\Sigma_k)&\leq \ind_\ener^+(\Sigma_k) + \dim \DF^+(\Sigma_k) \leq 2+(2k-3)=2k-1.
\end{align}
The matching lower bound 
$\ind_\area^+(\Sigma_k)\geq 2k-1$  
follows directly from \cite[Proposition~4.2\,(c)]{CSWSpectral}, applied with the $x_3$-axis as $\xi$, and the planes $\{x_2=0\}$, $\{x_3=0\}$ as $\Pi_1, \Pi_\perp$ in the notation therein, and we may conclude 
\begin{align}\label{eqn:plus_equality_knoid}
\ind_\area^+(\Sigma_k)=2k-1.
\end{align}
In particular, we have equality in \eqref{eqn:plus_upperbound_knoid}, which is the sum over all $\varrho\in\IrRep{\dih_k}$ of the estimates 
\begin{align}\label{eqn:varrhoboxtimestriv}
\ind_\area^{\varrho\boxtimes\triv}(\Sigma_k)\leq\ind_\ener^{\varrho\boxtimes\triv}(\Sigma_k)+\dim\DF^{\varrho\boxtimes\triv}(\Sigma_k)
\end{align}
given by Theorem~\ref{thm:equivariant_area_bounds}. 
Thus we also have equality in \eqref{eqn:varrhoboxtimestriv}. 
For $\varrho=\triv$ we have $\ind_\ener^{\triv\boxtimes\triv}(\Sigma_k)=0$ because $\Ind_\ener(\Sigma_k)\leq\cA_\rho$, and $\dim\DF^{\triv\boxtimes\triv}(\Sigma_k)=1$ by Theorem~\ref{thm:k-noidDF}\,\ref{thm:k-noidDF-iv}, hence $\ind_\area^{\triv\boxtimes\triv}(\Sigma_k)=1$.
Similarly, Theorem~\ref{thm:equivariant_area_bounds} implies for any $\varrho \in \IrRep{\dih_k}$ 
\begin{align}
\label{eqn:minus_upperbound_knoid}
\ind_\area^{\varrho \boxtimes \Or}(\Sigma_k)&\leq \ind_\ener^{\varrho \boxtimes \Or}(\Sigma_k) + \dim \DF^{\varrho \boxtimes \Or}(\Sigma_k).
\end{align}
Together with \eqref{eqn:knoid_ener_index_bounds} and Theorem~\ref{thm:k-noidDF}\,\ref{thm:k-noidDF-ii}--\ref{thm:k-noidDF-iii} estimate \eqref{eqn:minus_upperbound_knoid} yields 
\begin{align}
\label{eqn:area_minus_bounds_partial}
\ind_\area^{\triv\boxtimes\Or}(\Sigma_k) &\leq 1, &
\ind_\area^{\st_1\boxtimes\Or}(\Sigma_k) &= 0.
\end{align}

Any function $u \in C^\infty(\Sigma_k)$ satisfying $u \circ \tau = -u$ vanishes along $\Sigma_k \cap \{x_3 = 0\}$. 
Consequently, evaluating $\ind_\area^-(\Sigma_k)$ reduces to counting the negative eigenvalues $\lambda_i$ of the scalar Jacobi form \eqref{eqn:jacobi_form_fb} on $\Sigma_k^+ \vcentcolon= \Sigma_k \cap \{x_3 \ge 0\}$ with Dirichlet boundary conditions along $\Sigma_k \cap \{x_3 = 0\}$. 
By Lemma~\ref{lem:fundamental_domain_knoid} and Corollary~\ref{cor:half_knoid}\,\ref{cor:half_knoid-i}, Proposition~\ref{prop:nodal_dihedral} applies on $\Sigma_k^+$ 
(note that the form \eqref{eqn:jacobi_form_fb} is a special case of \eqref{eqn:spectral-index-form})
and yields that the isomorphism classes of the induced representation of $\dih_k$ on the first two eigenspaces satisfy 
\begin{align}\label{eqn:20260824}
\mathcal{E}_{\lambda_1} &=\triv, & 
\mathcal{E}_{\lambda_2} &=a\triv \oplus b\st_1
\end{align}
for some $a,b\in\N\cup\{0\}$. 
Suppose that $\lambda_2<0$. 
Then by \eqref{eqn:20260824} it suffices to distinguish the following cases: 
\begin{enumerate}[label={(\alph*)}]
\item $\triv\leq\mathcal{E}_{\lambda_2}$. 
Since we also have $\triv\leq\mathcal{E}_{\lambda_1}$ and since the eigenspaces corresponding to $\lambda_1$ and $\lambda_2$ are orthogonal, we would obtain 
$\ind_\area^{\triv \boxtimes \Or}(\Sigma_k)\geq 2$ contradicting \eqref{eqn:area_minus_bounds_partial}. 
\item $\st_1\leq\mathcal{E}_{\lambda_2}$. 
Then, since $\st_1$ has degree $2$, we have $\ind_\area^{\st_1 \boxtimes \Or}(\Sigma_k)\geq2$ contradicting \eqref{eqn:area_minus_bounds_partial}. 
\end{enumerate}
Consequently, $\lambda_2\geq0$ and $\ind_\area^-(\Sigma_k)\leq1$. 
The reverse inequality $\ind_\area^-(\Sigma_k) \geq 1$ follows directly from \cite[Proposition~4.2\,(b)]{CSWSpectral} (applied with parameter $k=2$ in the notation therein). 
Alternatively, $\ind_\area^-(\Sigma_k) \geq 1$ can also be concluded by considering the vertical coordinate function $u=x_3$, which is harmonic on $\Sigma_k$ with the correct boundary condition to yield a negative direction for the scalar Jacobi form \eqref{eqn:jacobi_form_fb} restricted to functions with odd parity.
Hence $\ind_\area^-(\Sigma_k)=1$. 
Adding \eqref{eqn:plus_equality_knoid} proves the claim 
\[
\ind_\area(\Sigma_k)=2k. 
\]
Since $\lambda_1<0\leq\lambda_2$, the only negative eigenvalue is $\lambda_1$, and its eigenspace is $\triv$-isotypic by \eqref{eqn:20260824}. 
Hence the unique negative direction contributes to the component $\triv\boxtimes\Or$ and all others vanish, which proves 
\begin{align}\label{eqn:minus_equality_knoid}
\ind_\area^{\varrho \boxtimes \Or}(\Sigma_k)&=\begin{cases}
1 & \text{ if $\varrho=\triv$,}\\
0 & \text{ otherwise.}
\end{cases}
\end{align}
It remains to show $\ind_\ener(\Sigma_k)=3$. 
By \eqref{eqn:plus_equality_knoid}, all inequalities in \eqref{eqn:plus_upperbound_knoid} are equalities, which implies $\ind_\ener^+(\Sigma_k)=2$.
Similarly, evaluating \eqref{eqn:minus_upperbound_knoid} for $\varrho=\triv$ and recalling $\dim\DF^{\triv\boxtimes\Or}(\Sigma_k)=0$ from Theorem~\ref{thm:k-noidDF}\,\ref{thm:k-noidDF-iii}, equality \eqref{eqn:minus_equality_knoid} implies $\ind_\ener^{\triv\boxtimes\Or}(\Sigma_k)\geq1$. 
Hence all inequalities in \eqref{eqn:knoid_ener_index_bounds} are equalities, and in particular $\ind_\ener^-(\Sigma_k)=1$. 
Thus, $\ind_\ener(\Sigma_k)=\ind_\ener^+(\Sigma_k)+\ind_\ener^-(\Sigma_k)=3$ as claimed.
\end{proof}

\begin{proposition}[Nullity]
\label{prop:knoid_nullity}
Given $3\leq k\in\N$, let $\Sigma_k\subset\B^3$ be any properly embedded free boundary minimal $k$-noid which is $G=\dih_k\times\Z_2$-equivariant in the representation \eqref{eqn:prismatic}. 
Then 
\[
\nul_\area(\Sigma_k)=3.
\] 
\end{proposition}

\begin{proof}
By Lemma~\ref{lem:spec_ener_knoid} we have $\ind_\spec(\Sigma_k)=1$. 
By \cite[Theorem~5.2]{KusnerMcGrath2020}, which applies due to Lemma~\ref{lem:fundamental_domain_knoid}, the first Steklov eigenspace of $\Sigma_k$ coincides with the span of the ambient coordinate functions restricted to $\Sigma_k$, so that $\nul_\spec(\Sigma_k)=3$. 
By Lemma~\ref{lem:knoid_quadratic_independence} (or alternatively Remark~\ref{rem:quadratic_knoid-diag}) the components of the inclusion $\Sigma_k\subset\B^3$ are quadratically independent on $\partial\Sigma_k$. 
Since moreover $\ind_\ener(\Sigma_k)=3$ by Proposition~\ref{prop:knoid_isotypic}, Proposition~\ref{prop:rigidity} applies with $n=2$ and yields 
\[
\Nul_\ener(\Sigma_k)=\mathfrak{so}(\cA_\rho) = (\Or\boxtimes\triv)\oplus(\st_1\boxtimes\Or)
\]
by Lemma~\ref{lem:so-Lemma}. 
Since $\Or$ has degree $1$ and $\st_1$ has degree $2$, we obtain 
\begin{align}
\label{eqn:knoid_ener_nullity}
\nul_\ener^{\Or\boxtimes\triv}(\Sigma_k)&=1, &
\nul_\ener^{\st_1\boxtimes\Or}(\Sigma_k)&=2, &
\nul_\ener^{r}(\Sigma_k)&=0 
\end{align}
for every other $r\in\IrRep{G}$, and in particular $\nul_\ener^+(\Sigma_k)=1$.
Theorem~\ref{thm:equivariant_area_bounds} implies  
\begin{align}\label{eqn:20260409} 
\ind_\area^r(\Sigma_k) + \nul_\area^r(\Sigma_k) 
&\leq 
\ind_\ener^r(\Sigma_k) + \nul_\ener^r(\Sigma_k) + \dim\DF^r(\Sigma_k). 
\end{align}
Summing \eqref{eqn:20260409} over all $r\in\IrRep{G}$ of the form $r=\varrho\boxtimes\triv$, and recalling $\ind_\area^+(\Sigma_k)=2k-1$ and $\ind_\ener^+(\Sigma_k)=2$ from Proposition~\ref{prop:knoid_isotypic} and $\dim\DF^+(\Sigma_k)=2k-3$ from Theorem~\ref{thm:k-noidDF}\,\ref{thm:k-noidDF-i}, estimate \eqref{eqn:20260409} implies 
$\nul_\area^+(\Sigma_k)\leq1$.

Since $\Sigma_k$ has negative Euler characteristic, it does not admit any nonzero Killing field tangent to $\partial\Sigma_k$ by Remark~\ref{rmk:dim_vconf}, because Killing fields are conformal.  
Hence the ambient rotations induce three linearly independent Jacobi fields on $\Sigma_k$.
By Lemma~\ref{lem:spec_ener_knoid} and the existence of the Jacobi field $u$ with even parity induced by 
ambient rotations around the $x_3$-axis, we have $\nul_\area^+(\Sigma_k)=\nul_\jacobi^+(\Sigma_k)\geq1$. 
Therefore
\begin{align}\label{eqn:even_nullity}
\nul_\area^+(\Sigma_k)=1.
\end{align}
It remains to prove $\nul_\area^-(\Sigma_k)=2$. 
The ambient rotations around the axes in the $x_1x_2$-plane induce two linearly independent Jacobi fields with odd parity, hence 
\begin{align}\label{eqn:odd_nullity_lower_bound}
\nul_\area^-(\Sigma_k)\geq2.
\end{align}
Since $\ind_\area^-(\Sigma_k)=1$ by Proposition~\ref{prop:knoid_isotypic}, 
the scalar Jacobi form \eqref{eqn:jacobi_form_fb} on 
$\Sigma_k^+\vcentcolon= \Sigma_k \cap\{x_3 \geq0\}$ with Dirichlet boundary conditions along $\Sigma_k \cap \{x_3=0\}$ has only one negative eigenvalue $\lambda_1<0$ and the 
next eigenvalue is $\lambda_2=0$ by \eqref{eqn:odd_nullity_lower_bound}. 
As in \eqref{eqn:20260824} we have $\mathcal{E}_{\lambda_2}=a\triv\oplus b\st_1$ for some $a,b\in\N\cup\{0\}$ by Proposition~\ref{prop:nodal_dihedral}, and therefore 
$\nul_\area^{\varrho\boxtimes\Or}(\Sigma_k)=0$ 
for every $\varrho\in\IrRep{\dih_k}\setminus\{\triv,\st_1\}$.
We evaluate estimate \eqref{eqn:20260409} in the remaining two cases. 
\begin{itemize} 
\item 
For $r=\triv\boxtimes\Or$, we have $\ind_\area^r(\Sigma_k)=1$ by \eqref{eqn:minus_equality_knoid} and $\ind_\ener^r(\Sigma_k)=1$ by Proposition~\ref{prop:knoid_isotypic}, while $\nul_\ener^r(\Sigma_k)=0$ by \eqref{eqn:knoid_ener_nullity} and $\dim\DF^r(\Sigma_k)=0$ by Theorem~\ref{thm:k-noidDF}\,\ref{thm:k-noidDF-iii}. 
Substituting into \eqref{eqn:20260409} yields $1+\nul_\area^{\triv\boxtimes\Or}(\Sigma_k)\leq1$ and hence $\nul_\area^{\triv\boxtimes\Or}(\Sigma_k)=0$. 
\item 
For $r=\st_1\boxtimes\Or$, we have $\ind_\area^r(\Sigma_k)=0$ by \eqref{eqn:minus_equality_knoid} and $\ind_\ener^r(\Sigma_k)=0$ by \eqref{eqn:knoid_ener_index_bounds}, while $\nul_\ener^r(\Sigma_k)=2$ by \eqref{eqn:knoid_ener_nullity} and $\dim\DF^r(\Sigma_k)=0$ by Theorem~\ref{thm:k-noidDF}\,\ref{thm:k-noidDF-ii}. 
Substituting into \eqref{eqn:20260409} yields $\nul_\area^{\st_1\boxtimes\Or}(\Sigma_k)\leq2$. 
\end{itemize}
Summation yields $\nul_\area^-(\Sigma_k)\leq2$, which complements \eqref{eqn:odd_nullity_lower_bound}. 
Therefore $\nul_\area^-(\Sigma_k)=2$. 
Adding \eqref{eqn:even_nullity} proves the claim. 
\end{proof}
  
\begin{proof}[Proof of Theorem \ref{thm:index-fbm-k-noid}]
The statement is covered by Propositions~\ref{prop:knoid_isotypic} and \ref{prop:knoid_nullity}, where the $G$-equivariant index, i.\,e.~the index restricted to $G$-equivariant normal variations, is $\ind_\area^{\triv\boxtimes\triv}(\Sigma_k)$ in the notation of Definition~\ref{defn:Isotypic}.
\end{proof}

\subsection{The first family of the Lawson surfaces}
\label{sec:Lawson}

Lawson \cite[§\,6]{Lawson1970} proved for any $m,n\in\N$ the existence of a closed, embedded minimal surface $\xi_{m,n}\subset\Sp^3$ of genus $mn$. 
Setting 
\begin{align*}
\Sp^1&\vcentcolon=\{x\in\Sp^3\mid x_3=x_4=0\}, &
\Sp^1_\perp&\vcentcolon=\{x\in\Sp^3\mid x_1=x_2=0\},
\end{align*}
the surface $\xi_{m,n}$ has $m+1$ great spheres of symmetry containing $\Sp^1$, whose union intersects $\Sp^1_\perp$ orthogonally in $2m+2$ equidistant points,
and it has $n+1$ great spheres of symmetry containing $\Sp^1_\perp$, whose union intersects $\Sp^1$ orthogonally in $2n+2$ equidistant points.

The surface $\xi_{1,0}\subset\Sp^3$ is congruent to the equatorial sphere $\Sp^2\subset\Sp^3$ with Morse index equal to $1$, and $\xi_{1,1}\subset\Sp^3$ is congruent to the Clifford torus $\{x\in\Sp^3\mid x_1^2+x_2^2=x_3^2+x_4^2\}$ with Morse index equal to $5$ by \cite{Urbano1990}. 
In this section we prove Theorem~\ref{thm:index-Lawson} concerning the Lawson surfaces $\xi_{1,k-1}$ for $3\leq k\in\N$, which are not given by an explicit parametrisation.

Setting $m=1$ and $n=k-1$, the surface $\xi_{1,k-1}$ has two great spheres of symmetry containing $\Sp^1$ and $k$ great spheres of symmetry containing $\Sp^1_\perp$. 
The latter intersect at angles of $\pi/k$, so the reflections across them generate a dihedral group $\dih_k$ of order $2k$, while the reflections across the former generate a group $\Z_2\times\Z_2$ of order $4$. 
Since the two sets of spheres are mutually orthogonal, these reflections commute, so the symmetry subgroup of $\xi_{1,k-1}$ generated by reflections across great spheres is the direct product 
\[
G=\dih_k\times\Z_2\times\Z_2. 
\]

\begin{remark}
\label{rem:Lawson_symmetry}
The group $G$ has order $8k$, while the full symmetry group of $\xi_{1,k-1}$ has order $16k$ due to the additional orientation-preserving isometries generated by the odd reflections around the great circles contained in the surface, see \cite[Lemma~4.3]{KapouleasWiygul2020}. 
\end{remark}

\paragraph{Extrinsic action.}
In standard Cartesian coordinates $(x_1,x_2,x_3,x_4)$ of $\R^4$, the group $G$ acts on $\Sp^3\subset\R^4\cong\R^2\times\R^2$ in the orthogonal representation
\begin{align}\label{eqn:repR4}
\rho&= \bigl(\st_1 \boxtimes (\triv\boxtimes\triv)\bigr) 
\oplus \bigl(\triv \boxtimes (\Or  \boxtimes\triv)\bigr) 
\oplus \bigl(\triv \boxtimes (\triv\boxtimes\Or  )\bigr).
\end{align}
This means that the $\dih_k$ factor acts in the standard representation $\st_1$ on the $x_1x_2$-plane and trivially on the $x_3x_4$-plane, while the two $\Z_2$ factors act as the reflections across the hyperplanes $\{x_3=0\}$ and $\{x_4=0\}$ respectively.
 
Let $\tau_1,\tau_2$ denote the nontrivial elements generating the $\Z_2\times\Z_2$ factor of $G$, such that 
$\rho(\tau_1)$ fixes the great sphere 
$\Sp^2_\perp=\{x\in\Sp^3\mid x_3=0\}$ while 
$\rho(\tau_2)$ fixes the great sphere 
$\Sp^2=\{x\in\Sp^3\mid x_4=0\}$. 
By Corollary~\ref{cor:Riemann-Hurwitz}, each of these two great spheres divides $\xi_{1,k-1}$ into two isometric, connected surfaces of genus zero with $k$ boundary components, on which the group $(\dih_k\times\Z_2)$ acts in the prismatic representation \eqref{eqn:prismatic}. 
In particular, $\xi_{1,k-1}$ is contained in neither great sphere.

\paragraph{Parity notation for order-two symmetries.} 
Given $k\geq3$ and $G=\dih_k\times\Z_2\times\Z_2$, let $\Phi\colon M\to\Sp^3$ be a $G$-equivariant minimal immersion. 
As in the case of a single $\Z_2$ factor, every irreducible representation $r\in\IrRep{G}$ is of the form 
$r=\varrho\boxtimes\varsigma\boxtimes\varsigma'$, where $\varrho\in \IrRep{\dih_k}$ and $\varsigma,\varsigma'\in\{\triv,\Or\}$. 
To keep track of the isotypic components associated with the two $\Z_2$ factors in $G$, we extend the shorthand notation of Definition~\ref{defn:parity_single}.  
 
\begin{definition}[Parity index and nullity]
\label{defn:parity}
Recalling Definition~\ref{defn:Isotypic}, we define  
\begin{align*}
\ind^{++}_\bullet(\Phi)
&\vcentcolon=\sum_{\varrho\in \IrRep{\dih_k}} 
\ind^{\varrho\boxtimes\triv\boxtimes\triv}_\bullet(\Phi), & 
\ind^{+-}_\bullet(\Phi)
&\vcentcolon=\sum_{\varrho\in \IrRep{\dih_k}} 
\ind^{\varrho\boxtimes\triv\boxtimes\Or}_\bullet(\Phi), 
\\
\ind^{--}_\bullet(\Phi)
&\vcentcolon=\sum_{\varrho\in \IrRep{\dih_k}} 
\ind^{\varrho\boxtimes\Or\boxtimes\Or}_\bullet(\Phi), & 
\ind^{-+}_\bullet(\Phi)
&\vcentcolon=\sum_{\varrho\in \IrRep{\dih_k}} 
\ind^{\varrho\boxtimes\Or\boxtimes\triv}_\bullet(\Phi), 
\end{align*}
and similarly $\nul^{\pm\pm}_\bullet(\Phi)$, for each type  
$\mathord\bullet\in\{\ener,\area,\jacobi\}$. 
Furthermore, we define the corresponding parity components of the space of transverse-traceless tensors as 
\begin{align*}
\DF^{++}(M)
&\vcentcolon=\bigoplus_{\varrho\in \IrRep{\dih_k}} 
\DF^{\varrho\boxtimes\triv\boxtimes\triv}(M), & 
\DF^{+-}(M)
&\vcentcolon=\bigoplus_{\varrho\in \IrRep{\dih_k}} 
\DF^{\varrho\boxtimes\triv\boxtimes\Or}(M), 
\\
\DF^{--}(M)
&\vcentcolon=\bigoplus_{\varrho\in \IrRep{\dih_k}} 
\DF^{\varrho\boxtimes\Or\boxtimes\Or}(M), & 
\DF^{-+}(M)
&\vcentcolon=\bigoplus_{\varrho\in \IrRep{\dih_k}} 
\DF^{\varrho\boxtimes\Or\boxtimes\triv}(M).
\end{align*}
\end{definition}

\begin{remark}
The value $\ind^{\pm\pm}_\jacobi(\Phi)$ is the maximal dimension of a subspace of functions $f\in C^\infty(M)$ satisfying 
$f \circ \tau_1 = \pm f$ and 
$f \circ \tau_2 = \pm f$, 
on which the Jacobi quadratic form \eqref{eqn:jacobi_form_closed} is negative definite. 
This notation is consistent with the one used in \cite{CSWSpectral,CSWstackings}. 
Likewise, $\DF^{\pm\pm}(M)$ is the subspace of transverse-traceless tensors $h \in \DF(M)$ satisfying $\tau_1^* h = \pm h$ and $\tau_2^* h = \pm h$. 
\end{remark}
 
\begin{lemma}
\label{lem:so-Lemma-Lawson}
Let $\rho\colon G=\dih_k\times\Z_2\times\Z_2\to\Ogroup(4)$ be the representation \eqref{eqn:repR4}. 
Then, 
\[
\mathfrak{so}(\cA_\rho) 
=(\Or\boxtimes\triv\boxtimes\triv)
\oplus\bigl(\st_1\boxtimes(\Or\boxtimes\triv\oplus\triv\boxtimes\Or)\bigr)
\oplus(\triv\boxtimes\Or\boxtimes\Or).
\]
\end{lemma}

\begin{proof}
By Definition~\ref{defn:so_A_rho}, the group $G$ acts on $\mathfrak{so}(4)$ by conjugation $B\mapsto\rho(\gamma)B\rho(\gamma)^{-1}$. 
Using the Euclidean metric to identify a skew-symmetric endomorphism $B$ of $\R^4$ with the alternating bilinear form $(u,v)\mapsto\brt{Bu,v}$, we obtain an isomorphism between $\mathfrak{so}(4)$ and the exterior square $\Lambda^2\R^4$ spanned by the bivectors $e_i\wedge e_j$ for $i<j$, where $e_1,e_2,e_3,e_4$ denotes the standard basis of $\R^4$.
Letting $G$ act on $\Lambda^2\R^4$ by $\gamma(u\wedge v)\vcentcolon=\rho(\gamma)u\wedge\rho(\gamma)v$, this 
isomorphism becomes $G$-equivariant since $\rho$ is orthogonal, and we obtain $\mathfrak{so}(\cA_\rho)=\Lambda^2\cA_\rho$. 
Here, $\Lambda^2\cA_\rho$ denotes the isomorphism class of the $G$-representation on $\Lambda^2\R^4$ defined above. 
We decompose it as follows.
\begin{itemize}
\item 
The action of both $\Z_2$-factors preserves $e_1\wedge e_2$, 
while any $\gamma\in\dih_k$ acts on $e_1\wedge e_2$ by multiplication with $\det\st_1(\gamma)=\Or(\gamma)$. 
The representation on the span of $e_1\wedge e_2$ is therefore $\Or\boxtimes\triv\boxtimes\triv$.
\item 
On the span of $e_1\wedge e_3$ and $e_2\wedge e_3$, the map $v\wedge e_3\mapsto v$ is a $\dih_k$-equivariant isomorphism which yields the representation $\st_1\boxtimes\Or\boxtimes \triv$.
\item 
Similarly, on the span of $e_1\wedge e_4$ and $e_2\wedge e_4$ one has a $\dih_k$-equivariant isomorphism $v\wedge e_4\mapsto v$, which yields the representation $\st_1\boxtimes\triv\boxtimes \Or$.
\item 
The element $e_3\wedge e_4$ is preserved by $\dih_k$ while each $\Z_2$-factor reverses it and we obtain the representation $\triv\boxtimes\Or\boxtimes\Or$. \qedhere
\end{itemize}
\end{proof}
 
The following statement is analogous to Lemma~\ref{lem:spec_ener_knoid}. 

\begin{lemma}
\label{lem:spec1}
Given $3\leq k\in\N$, let $M_k\subset\Sp^3$ be any closed, embedded minimal surface of genus $k-1$ which is $G=(\dih_k \times \Z_2 \times \Z_2)$-equivariant in the representation \eqref{eqn:repR4}.  
Then 
\begin{align*}
\Ind_\area(M_k)&=\Ind_\jacobi(M_k), &
\Nul_\area(M_k)&=\Nul_\jacobi(M_k), &
\ind_\spec(M_k)&=1.
\end{align*}
\end{lemma}

\begin{proof}  
Since $M_k\subset\Sp^3$ is closed and embedded, it is necessarily two-sided and thus admits a global unit normal field $\nu$. 
By Lemma~\ref{lem:Riemann-Hurwitz}, the intersection $M_k\cap\Sp^1_\perp$ contains four points, and since $k\geq3$ the unit normal $\nu$ is tangent to $\Sp^1_\perp$ at each of them. 
Consequently, every reflection in $\dih_k$ preserves $\nu$. 
To see that $\tau_1,\tau_2\in\Z_2\times\Z_2$ also preserve $\nu$, we consider the great sphere $\Sp^2_\perp$ fixed by the action of $\tau_1$ and argue that $M_k$ being $\tau_1$-invariant intersects $\Sp^2_\perp$ orthogonally (since $M_k\not\subset\Sp^2_\perp$ by Corollary~\ref{cor:Riemann-Hurwitz}). 
Therefore, $\nu$ is tangent to $\Sp^2_\perp$ along $M_k\cap\Sp^2_\perp$ and thus preserved by $\tau_1$. 
The argument for $\tau_2$ is similar, with $\Sp^2$ in place of $\Sp^2_\perp$. 
Since $G$ is generated by $\dih_k$ together with $\tau_1,\tau_2$, we conclude that $\nu$ is preserved by every generator of $G$, so the relative orientation representation (see Definition~\ref{defn:relative_orientation_representation}) is trivial. 
Proposition~\ref{prop:indJ_indA} then yields the equivariant isomorphisms 
\begin{align*}
\Ind_\area(M_k)&=\Ind_\jacobi(M_k), &
\Nul_\area(M_k)&=\Nul_\jacobi(M_k).
\end{align*}
Since $M_k\subset\Sp^3$ is minimal and $\dim M_k=2$, the associated spectral quadratic form \eqref{eqn:spectral-index-form} simplifies to 
\begin{align*} 
Q_{\spec}(u,u)&=\int_{M_k}\bigl(\abs{du}_g^2-2u^2\bigr)
\end{align*}
as in Remark~\ref{rem:spectral_nullity}. 
Because $G$ is a reflection group of order $8k$ tessellating $\Sp^3$ into tetrahedra and $M_k$ has genus $k-1$, the first nonzero eigenvalue of the Laplace--Beltrami operator on $M_k$ equals $2$ by \cite[Theorem~6.2]{ChoeSoret2009}.  
Hence $Q_\spec(u,u)\geq0$ for every $u$ orthogonal to the constant functions, while $Q_\spec(1,1)<0$. 
Thus, $\ind_\spec(M_k)=1$. 
\end{proof}

\begin{lemma}[{\cite[Theorem 5.4]{KusnerMcGrath2020}}]
\label{lem:spectral_nullity}
Given $3\leq k\in\N$, let $M_k\subset\Sp^3$ be any closed, embedded minimal surface of genus $k-1$ which is $G=(\dih_k\times\Z_2\times\Z_2)$-equivariant in the representation \eqref{eqn:repR4}.  
Then $\nul_\spec(M_k)=4$.     
\end{lemma}

\begin{proof}
The group $G$ is finite and acts on $\R^4$ by reflections in hyperplanes. 
By Remark~\ref{rem:half-sphere} stated in the appendix, \cite[Assumption 5.3]{KusnerMcGrath2020} holds. 
Therefore, \cite[Theorem 5.4]{KusnerMcGrath2020} implies that the eigenspace of the Laplace--Beltrami operator on $M_k$ corresponding to the eigenvalue $2$ is spanned by the ambient coordinate functions restricted to $M_k$. 
Since $M_k$ is linearly full by Lemma~\ref{lem:lawson_quadratic_independence} and Remark~\ref{rem:quadratic_linear}, 
we then obtain $\nul_\spec(M_k)=4$ by Remark~\ref{rem:spectral_nullity}.
\end{proof}

In the closed case, Proposition~\ref{prop:Arho_lb} allows us to take the following shortcut, determining $\Ind_\ener(M_k)$ directly from $\ind_\spec(M_k)=1$. 
As explained in Remark~\ref{rem:fb_energy_unknown}, this route is not available for free boundary minimal surfaces. 
In the proof of Proposition~\ref{prop:knoid_isotypic}, we therefore had to determine the area index first and read off the energy index from the resulting equality in Theorem~\ref{thm:equivariant_area_bounds}.
 
\begin{proposition}[Energy index and nullity]
\label{prop:lawson_energy}
Given $3\leq k\in\N$, let $M_k\subset\Sp^3$ be any closed, embedded minimal surface of genus $k-1$ which is $G=(\dih_k\times\Z_2\times\Z_2)$-equivariant in the representation $\rho$ given in \eqref{eqn:repR4}. 
Then 
\begin{align*}
\Ind_\ener(M_k)&=\cA_\rho, & 
\Nul_\ener(M_k)&=\mathfrak{so}(\cA_\rho).
\end{align*}
In particular, in the notation of Definition~\ref{defn:parity}, 
\begin{align*}
\ind_\ener^{++}(M_k)&=2, &
\ind_\ener^{-+}(M_k)&=1=\ind_\ener^{+-}(M_k), &
\ind_\ener^{--}(M_k)&=0, 
\\
\nul_\ener^{++}(M_k)&=1, &
\nul_\ener^{-+}(M_k)&=2=\nul_\ener^{+-}(M_k), &
\nul_\ener^{--}(M_k)&=1.
\end{align*}
\end{proposition}

\begin{proof}
By Lemma~\ref{lem:spec1}, $\ind_\spec(M_k)=1$ and  Proposition~\ref{prop:Arho_lb} yields $\Ind_\ener(M_k)=\cA_\rho$ as claimed. 
Since $\nul_\spec(M_k)=4$ by Lemma~\ref{lem:spectral_nullity} and the components of the inclusion $M_k\subset\Sp^3$ are quadratically independent by Lemma~\ref{lem:lawson_quadratic_independence}, Proposition~\ref{prop:rigidity} applies with $n=3$ and yields 
\[
\Nul_\ener(M_k)=\mathfrak{so}(\cA_\rho).
\]
We recall that irreducible representations $r\in\IrRep{G}$ are of the form 
$r=\varrho\boxtimes\varsigma\boxtimes\varsigma'$, where $\varrho\in \IrRep{\dih_k}$ and $\varsigma,\varsigma'\in\{\triv,\Or\}$ have degree $1$. 
In particular, $r$ has the same degree as $\varrho$, while its parity is determined by the pair $(\varsigma,\varsigma')$. 
Recalling \eqref{eqn:repR4}, the three summands of 
\[
\rho= \bigl(\st_1\boxtimes(\triv\boxtimes\triv)\bigr) 
\oplus\bigl(\triv\boxtimes(\Or  \boxtimes\triv)\bigr) 
\oplus\bigl(\triv\boxtimes(\triv\boxtimes\Or  )\bigr)
\]
therefore have parities $++$, $-+$, $+-$ and degrees $2$, $1$, $1$ respectively, and no summand has parity $--$, which yields the stated energy indices. 
Likewise, the four summands of $\mathfrak{so}(\cA_\rho)$ in Lemma~\ref{lem:so-Lemma-Lawson} have parities $++$, $-+$, $+-$, $--$ and degrees $1$, $2$, $2$, $1$ respectively, which yields the stated energy nullities.
\end{proof} 
 
\begin{proposition}
\label{prop:ind++}
Given $3\leq k\in\N$, let $M_k\subset\Sp^3$ be any closed, embedded minimal surface of genus $k-1$ which is $G=(\dih_k\times\Z_2\times\Z_2)$-equivariant in the representation $\rho$ from \eqref{eqn:repR4}.  
Then, in the notation of Definition~\ref{defn:parity},
\begin{align*} 
\ind_\area^{++}(M_k)&=2k-1, & 
\nul_\area^{++}(M_k)&=1.
\end{align*}
\end{proposition}

\begin{proof}   
By Corollary~\ref{cor:Riemann-Hurwitz}, the great sphere $\Sp^2$ fixed pointwise by $\rho(\tau_2)$ divides $M_k$ into two congruent $k$-noids. 
The $\tau_2$-even transverse-traceless tensors on $M_k$ are naturally identified with the transverse-traceless tensors on the $k$-noid $\Sigma_k\vcentcolon=M_k\cap\{x_4\geq0\}$, which is a free boundary minimal $k$-noid in the half-sphere $\{x\in\Sp^3\mid x_4\geq0\}$ with $\dih_k\times\Z_2$ acting as in Remark~\ref{rem:knoid-halfsphere}. 
Hence Theorem~\ref{thm:k-noidDF}\,\ref{thm:k-noidDF-i} and Remark~\ref{rem:knoid-halfsphere} yield
\begin{align}
\label{eqn:DF++}
\dim \DF^{++}(M_k) = \dim \DF^{+}(\Sigma_k) = 2k-3.
\end{align}
Together with $\ind_\ener^{++}(M_k)=2$ from Proposition~\ref{prop:lawson_energy}, Theorem~\ref{thm:equivariant_area_bounds} implies 
\begin{align}
\label{eqn:++upperbound}
\ind_\area^{++}(M_k) \leq \ind_\ener^{++}(M_k) + \dim \DF^{++}(M_k) = 2+(2k-3)=2k-1.
\end{align}
For the matching lower bound, we consider the Jacobi field $u$ on $M_k$ induced by the ambient rotations of $\Sp^3$ in the $x_1x_2$-plane, which fix $\Sp^1_\perp$ pointwise and commute with $\rho(\tau_1)$ and $\rho(\tau_2)$. 
Hence $u$ is even with respect to $\tau_1$ and $\tau_2$, and odd with respect to the reflections in $\dih_k$. 
In particular, $u$ vanishes on the fixed set of the reflections in $\dih_k$ and thus has at least $2k$ nodal domains in $M_k$. 
(We have $u\not\equiv0$ by Remark~\ref{rmk:dim_vconf}, which implies that no nonzero Killing field is everywhere tangent to $M_k$ because $M_k$ has negative Euler characteristic.) 
By \cite[Corollary~3.2\,(i)]{CSWSpectral} applied with $n=2k$, we obtain $\ind_\jacobi^{++}(M_k)\geq2k-1$. 
Since $\ind_\jacobi^{++}(M_k)=\ind_\area^{++}(M_k)$ by Lemma~\ref{lem:spec1}, we conclude 
\begin{align}
\label{eqn:++equality}
\ind_\area^{++}(M_k)=2k-1.
\end{align}
It remains to determine $\nul_\area^{++}(M_k)$. 
The Jacobi field $u$ from above is an eigenfunction of $Q_\jacobi$ with eigenvalue $0$, so $\nul_\area^{++}(M_k)=\nul_\jacobi^{++}(M_k)\geq1$ by Lemma~\ref{lem:spec1}. 
Conversely, Theorem~\ref{thm:equivariant_area_bounds} implies 
\begin{align*} 
\ind_\area^{++}(M_k) + \nul_\area^{++}(M_k) \leq 
\ind_\ener^{++}(M_k) + \nul_\ener^{++}(M_k) + \dim \DF^{++}(M_k), 
\end{align*}
and inserting \eqref{eqn:DF++}, \eqref{eqn:++equality} as well as $\ind_\ener^{++}(M_k)=2$ and $\nul_\ener^{++}(M_k)=1$ from Proposition~\ref{prop:lawson_energy} yields $\nul_\area^{++}(M_k)\leq1$. 
Therefore, $\nul_\area^{++}(M_k)=1$ as claimed. 
\end{proof}

\begin{proposition}
\label{prop:ind+-}
Given $3\leq k\in\N$, let $M_k\subset\Sp^3$ be any closed, embedded minimal surface of genus $k-1$ which is $G=(\dih_k\times\Z_2\times\Z_2)$-equivariant in the representation $\rho$ from \eqref{eqn:repR4}.  
Then, in the notation of Definition~\ref{defn:parity},
\begin{align*} 
\ind_\area^{-+}(M_k)=\ind_\area^{+-}(M_k)&=1, & 
\nul_\area^{-+}(M_k)=\nul_\area^{+-}(M_k)&=2. 
\end{align*}
\end{proposition}

\begin{proof}
By Theorem~\ref{thm:equivariant_area_bounds} we have for any $r\in\IrRep{G}$
\begin{align}
\label{eqn:iso_area_bound}
\ind_\area^r(M_k) \leq \ind_\ener^r(M_k) + \dim \DF^r(M_k),
\end{align}
which we apply to the isomorphism classes of the representations 
$\triv \boxtimes \Or \boxtimes \triv$ and 
$\st_1 \boxtimes \Or \boxtimes \triv$. 
Via Corollary~\ref{cor:Riemann-Hurwitz}, cutting $M_k$ along the great sphere $\Sp^2$ fixed pointwise by $\rho(\tau_2)$ identifies the $\tau_2$-even transverse-traceless tensors on $M_k$ with those on the resulting $(\dih_k\times\Z_2)$-invariant $k$-noid $\Sigma_k$, as in the proof of Proposition~\ref{prop:ind++}.  
Therefore, by Theorem~\ref{thm:k-noidDF}\,\ref{thm:k-noidDF-ii}--\ref{thm:k-noidDF-iii} and Remark~\ref{rem:knoid-halfsphere} 
\begin{align}
\label{eqn:DF_isotypic_bounds}
\dim \DF^{\triv \boxtimes \Or \boxtimes \triv}(M_k) &= 0, &
\dim \DF^{\st_1 \boxtimes \Or \boxtimes \triv}(M_k) &= 0.
\intertext{The only summand of $\rho$ with $-+$ parity is $\triv\boxtimes(\Or\boxtimes\triv)$. 
Hence, by Proposition~\ref{prop:lawson_energy}}
\label{eqn:ener_isotypic_bounds}
\ind_\ener^{\triv \boxtimes \Or \boxtimes \triv}(M_k) &= 1, & 
\ind_\ener^{\st_1 \boxtimes \Or \boxtimes \triv}(M_k) &= 0.
\intertext{Substituting \eqref{eqn:DF_isotypic_bounds} and \eqref{eqn:ener_isotypic_bounds} into \eqref{eqn:iso_area_bound} yields} 
\label{eqn:20260331}
\ind_\area^{\triv\boxtimes\Or\boxtimes\triv}(M_k)&\leq 1, &
\ind_\area^{\st_1\boxtimes\Or\boxtimes\triv}(M_k)&=0.
\end{align}
Any function $u\in C^\infty(M_k)$ satisfying $u\circ \tau_1=-u$ vanishes on the great sphere $\Sp^2_\perp$ fixed pointwise by $\rho(\tau_1)$. 
Consequently, evaluating $\ind_\area^{-+}(M_k)$ reduces to the mixed eigenvalue problem for the Jacobi form $Q_\jacobi$ on $\Omega\vcentcolon=M_k\cap\{x_3,x_4\geq0\}$ with Dirichlet conditions along $M_k\cap\Sp^2_\perp$ and Neumann conditions along $M_k\cap\Sp^2$. 
Let $\lambda_1<\lambda_2$ denote the first two distinct eigenvalues of this problem and $\mathcal{E}_{\lambda_1},\mathcal{E}_{\lambda_2}$ the isomorphism classes of the induced representations of $\dih_k$ on the corresponding eigenspaces. 
Since $Q_\jacobi$ is the quadratic form \eqref{eqn:spectral-index-form} with the $\dih_k$-invariant potentials $q=\abs{\II}^2+2$ and $\sigma=0$, and since the action of $\dih_k$ on $\Omega$ satisfies the remaining assumptions of Proposition~\ref{prop:nodal_dihedral} by Corollary~\ref{cor:half_knoid} and Remark~\ref{rem:half-sphere}, we obtain
\begin{align}\label{eqn:20260331b}
\mathcal{E}_{\lambda_1} &= \triv, & 
\mathcal{E}_{\lambda_2} &=a\triv \oplus b\st_1
\end{align}
for some $a,b\in\N\cup\{0\}$. 
Denoting the standard orthonormal basis of $\R^4$ by $e_1,e_2,e_3,e_4$, and the unit normal vector field on $M_k$ by $\nu$, the $\dih_k$-invariant function $f=\brt{e_3,\nu}$ satisfies 
\begin{align*}
f\circ\tau_1&=-f, &
f\circ\tau_2&=f, & 
Q_\jacobi(f,f)&=-2\int_{M_k}f^2<0  
\end{align*} 
by \cite[Lemma~5.1.4]{Simons1968}, since $f\not\equiv0$ as $M_k$ is not totally geodesic. 
Hence $\ind_\area^{\triv \boxtimes \Or \boxtimes \triv}(M_k)\geq1$ and $\lambda_1<0$, which combined with \eqref{eqn:20260331} yields  
\begin{align}
\label{eqn:20260331c}
\ind_\area^{\triv \boxtimes \Or \boxtimes \triv}(M_k)=1
\end{align}
and $\lambda_2\geq0$. 
By \eqref{eqn:20260331b}, the only negative eigenvalue $\lambda_1$ has a $\triv$-isotypic eigenspace. 
Therefore, $\ind_\area^{\varrho\boxtimes \Or \boxtimes \triv}(M_k)=0$ for every $\varrho\in\IrRep{\dih_k}\setminus\{\triv\}$ and $\ind_\area^{-+}(M_k)=1$ as claimed. 

It remains to compute $\nul_\area^{-+}(M_k)$. 
The skew-symmetric endomorphisms of $\R^4$ anticommuting with $\rho(\tau_1)$ and commuting with $\rho(\tau_2)$ span the two-dimensional subspace of $\mathfrak{so}(4)$ generated by the rotations in the $x_1x_3$-plane and in the $x_2x_3$-plane. 
The normal components of the corresponding Killing fields provide two linearly independent Jacobi fields on $M_k$. 
(Indeed, since $M_k$ has negative Euler characteristic for $k\geq3$, Remark~\ref{rmk:dim_vconf} implies that no nonzero Killing field is everywhere tangent to $M_k$.)
Consequently
\begin{align}
\label{eqn:-+nullity_lower}
\nul_\area^{-+}(M_k)\geq2.
\end{align}
In particular, zero is an eigenvalue of the mixed eigenvalue problem on $\Omega$. 
Since $\lambda_1$ is the only negative eigenvalue by \eqref{eqn:20260331c} and \eqref{eqn:20260331b}, we conclude $\lambda_2=0$. 
By \eqref{eqn:20260331b} we therefore have $\nul_\area^{\varrho\boxtimes\Or\boxtimes\triv}(M_k)=0$ for every $\varrho\in\IrRep{\dih_k}\setminus\{\triv,\st_1\}$, and it remains to evaluate 
\begin{align}
\label{eqn:-+nullity_bound}
\ind_\area^{r}(M_k) + \nul_\area^{r}(M_k) \leq 
\ind_\ener^{r}(M_k) + \nul_\ener^{r}(M_k) + \dim \DF^{r}(M_k)
\end{align}
from Theorem~\ref{thm:equivariant_area_bounds} in the remaining two cases. 
\begin{itemize}
\item 
For $r=\triv\boxtimes\Or\boxtimes\triv$, we have $\ind_\area^r(M_k)=1$ by \eqref{eqn:20260331c} and $\ind_\ener^r(M_k)=1$ by \eqref{eqn:ener_isotypic_bounds}, while $\dim\DF^r(M_k)=0$ by \eqref{eqn:DF_isotypic_bounds} and $\nul_\ener^r(M_k)=0$, since the $-+$ summand of $\mathfrak{so}(\cA_\rho)$ in Lemma~\ref{lem:so-Lemma-Lawson} is $\st_1\boxtimes\Or\boxtimes\triv$. 
Substituting into \eqref{eqn:-+nullity_bound} yields $\nul_\area^{\triv\boxtimes\Or\boxtimes\triv}(M_k)=0$. 
\item 
For $r=\st_1\boxtimes\Or\boxtimes\triv$, we have $\ind_\area^r(M_k)=0$ by \eqref{eqn:20260331} and $\ind_\ener^r(M_k)=0$ by \eqref{eqn:ener_isotypic_bounds}, while $\dim\DF^r(M_k)=0$ by \eqref{eqn:DF_isotypic_bounds} and $\nul_\ener^r(M_k)=2$ by Lemma~\ref{lem:so-Lemma-Lawson}. 
Substituting into \eqref{eqn:-+nullity_bound} yields $\nul_\area^{\st_1\boxtimes\Or\boxtimes\triv}(M_k)\leq2$. 
\end{itemize}
Summing over all $\varrho\in\IrRep{\dih_k}$, we obtain $\nul_\area^{-+}(M_k)\leq2$, which complements \eqref{eqn:-+nullity_lower}. 
Therefore, $\nul_\area^{-+}(M_k)=2$ as claimed. 
The same arguments apply in the case of $+-$ parity by interchanging the roles of $\tau_1$ and $\tau_2$, respectively of the coordinates $x_3$ and $x_4$. 
\end{proof}

Unlike the two preceding Propositions~\ref{prop:ind++} and~\ref{prop:ind+-}, the following statement does not rely on Theorem~\ref{thm:equivariant_area_bounds}, but instead on a comparison between Dirichlet and mixed boundary conditions.

\begin{proposition}
\label{prop:ind--}
Given $3\leq k\in\N$, let $M_k\subset\Sp^3$ be any closed, embedded minimal surface of genus $k-1$ which is $G=(\dih_k\times\Z_2\times\Z_2)$-equivariant in the representation $\rho$ from \eqref{eqn:repR4}.  
Then, in the notation of Definition~\ref{defn:parity},
\begin{align*} 
\ind_\area^{--}(M_k)&=0, & 
\nul_\area^{--}(M_k)&=1.
\end{align*}
\end{proposition}

\begin{proof}
Any function $u\in C^\infty(M_k)$ satisfying $u \circ \tau_1 = -u$ and $u \circ \tau_2 = -u$ vanishes on both $M_k\cap\Sp^2_\perp$ and $M_k\cap\Sp^2$. 
Consequently, evaluating $\ind_\area^{--}(M_k)$ reduces to the Dirichlet eigenvalue problem on a connected component $\Omega$ of $M_k\setminus(\Sp^2_\perp\cup\Sp^2)$. 

By contrast, functions $u\in C^\infty(M_k)$ satisfying $u\circ\tau_1=-u$ and $u\circ\tau_2=u$ vanish on $\partial\Omega\cap\Sp^2_\perp$ while satisfying the Neumann boundary condition on $\partial\Omega\cap\Sp^2$. 
Let $\lambda_1^{--}<\lambda_2^{--}\leq\ldots$ denote the eigenvalues of the scalar Jacobi quadratic form \eqref{eqn:jacobi_form_closed} on $H^1_0(\Omega)$, and let $\lambda_1^{-+}<\lambda_2^{-+}\leq\ldots$ denote the eigenvalues of the same form on $\Omega$ with the mixed boundary conditions described above. 
By the variational characterisation of eigenvalues, $\lambda_j^{-+}\leq\lambda_j^{--}$ for every $j\in\N$. 

Let $w\neq0$ be a linear combination of the first two Dirichlet eigenfunctions which is $L^2$-orthogonal to the first mixed eigenfunction. 
Then $\lambda_2^{-+}\int_\Omega w^2\leq Q_\jacobi(w,w)\leq\lambda_2^{--}\int_\Omega w^2$. 
If $\lambda_2^{-+}=\lambda_2^{--}$ then $w$ is a mixed eigenfunction which vanishes along $\partial\Omega\cap\Sp^2$ together with its normal derivative, contradicting the unique continuation principle. 
Therefore, $\lambda_2^{-+}<\lambda_2^{--}$.
Since $\ind_\area^{-+}(M_k)=1$ by Proposition~\ref{prop:ind+-}, we have $\lambda_1^{-+}<0\leq\lambda_2^{-+}<\lambda_2^{--}$ and therefore 
\begin{align}\label{eqn:20260401}
\ind_\area^{--}(M_k) + \nul_\area^{--}(M_k)\leq 1.
\end{align}
For the matching lower bound, we consider the rotation $\rotation\in\mathfrak{so}(4)$ in the $x_3x_4$-plane, which anticommutes with both $\rho(\tau_1)$ (which reflects $x_3$) and $\rho(\tau_2)$ (which reflects $x_4$). 
Furthermore, it commutes with the action of $\dih_k$ on the $x_1x_2$-plane. 
Therefore, the one-dimensional subspace of $\mathfrak{so}(4)$ spanned by $\rotation$ is invariant under the conjugation action of $G$, and the induced representation of $G$ on this subspace is isomorphic to $\triv\boxtimes\Or\boxtimes\Or$ (cf.~Lemma~\ref{lem:so-Lemma-Lawson}). 
Since $M_k$ has negative Euler characteristic, Remark~\ref{rmk:dim_vconf} implies that no nonzero Killing field is everywhere tangent to $M_k$, so the normal component of the corresponding Killing field defines a nontrivial Jacobi field on $M_k$ with $--$ parity.
Hence, $\nul_\area^{--}(M_k)\geq1$. 
Combined with \eqref{eqn:20260401} this yields $\ind_\area^{--}(M_k)=0$ and $\nul_\area^{--}(M_k)=1$ as claimed. 
\end{proof}

\begin{proof}[Proof of Theorem \ref{thm:index-Lawson}]
Let $3\leq k\in\N$ and let $M_k\subset\Sp^3$ be a closed, embedded minimal surface of genus $k-1$ which is $(\dih_k\times\Z_2\times\Z_2)$-equivariant in the representation \eqref{eqn:repR4}.
By Propositions~\ref{prop:ind++}, \ref{prop:ind+-} and  \ref{prop:ind--},  
\begin{align*}
  \ind_\area(M_k) 
&=\ind_\area^{++}(M_k) 
 +\ind_\area^{+-}(M_k)
 +\ind_\area^{-+}(M_k)
 +\ind_\area^{--}(M_k)
= 2k+1, 
\\[1.5ex]
  \nul_\area(M_k) 
&=\nul_\area^{++}(M_k) 
 +\nul_\area^{+-}(M_k)
 +\nul_\area^{-+}(M_k)
 +\nul_\area^{--}(M_k)
 = 6. 
 \qedhere
\end{align*}
\end{proof}

\pagebreak[2]
\section{Free boundary minimal punctured tori in the unit ball}
\label{sec:g1b1}

In this section we prove Theorem~\ref{thm:index-fbm-torus} stating that any properly embedded free boundary minimal surface $M_1\subset \B^3$ 
of genus $1$ with connected boundary and antiprismatic symmetry group of order $8$ has Morse index equal to $5$ and nullity $3$.
In this setting, sharp lower bounds are already available, as the following lemma shows. 
The remainder of the section is therefore devoted to matching upper bounds, for which we rely on the theory developed in Sections~\ref{sec:equivariant-harmonic}--\ref{sec:dihedral}.

\begin{lemma}
\label{lem:index_lower_bound_genus}
Any properly embedded free boundary minimal surface $\Sigma\subset\B^3$ with nonzero genus satisfies  
$\ind_{\area}(\Sigma)\geq 5$ and 
$\nul_{\area}(\Sigma)\geq 3$.  
\end{lemma}

\begin{proof}
Since $\Sigma$ is not a flat disc, $\ind_\area(\Sigma)\geq 4$ by 
\cite[Proposition 6.2]{Devyver2019}. 
Moreover, since $\Sigma$ is embedded with nonzero genus, $\ind_\area(\Sigma)\neq 4$ by \cite[Corollary~7.2]{Devyver2019}  respectively \cite[Corollary~3.10]{Tran2020}. 
Since $\Sigma$ has negative Euler characteristic, the same argument as in the proof of Proposition~\ref{prop:knoid_nullity} shows that the ambient rotations induce three linearly independent Jacobi fields on $\Sigma$, which implies $\nul_\area(\Sigma)\geq3$. 
\end{proof}

The symmetry group of a regular $m$-gonal antiprism can be formalised as follows (cf.~\cite[Figure~2]{CSWnonuniqueness} and \cite[(2.4)]{CSWstackings}). 

\begin{definition}[{Antiprismatic symmetry}]
\label{defn:antiprismatic}
Given $2\leq m\in\N$, the \emph{antiprismatic} group of order $4m$ is $\dih_{2m}$ acting on $\R^3$ in the representation 
\begin{align} 
\label{eqn:antiprismatic}
\rho=\st_1\oplus\m_1\colon\dih_{2m}\to\Ogroup(3).
\end{align}
\end{definition}

Throughout this section, we specialise Definition~\ref{defn:antiprismatic} to the case $m=2$, where $\dih_4$ acts on $M_1\subset\B^3$ as the antiprismatic group of order $8$.

\paragraph{Extrinsic action.}
In standard Cartesian coordinates $(x_1, x_2, x_3)\in\R^3$, the generator $\gamma_1$ of $\dih_4$ acts as the ambient rotation by angle $\pi$ around the $x_1$-axis, while the generator $\gamma_2$ acts as the reflection across the vertical plane $\{x_1 = x_2\}$ in the sense that 
\begin{align}\label{eqn:extrinsicD4action}
\rho(\gamma_1)&=\begin{pmatrix}
1 & \hphantom{-}0 & \hphantom{-}0 \\
0 & -1 & \hphantom{-}0 \\
0 & \hphantom{-}0 & -1
\end{pmatrix}, &
\rho(\gamma_2)&=\begin{pmatrix}
0 & \hphantom{-}1 & \hphantom{-}0 \\
1 & \hphantom{-}0 & \hphantom{-}0 \\
0 & \hphantom{-}0 & \hphantom{-}1
\end{pmatrix}.
\end{align}
Indeed, \eqref{eqn:extrinsicD4action} restricts to the standard representation $\st_1$ of $\dih_4$ on the $x_1x_2$-plane as defined in \eqref{eqn:St-matrix}, while on the $x_3$-axis, $\gamma_1$ acts by $-1$ and $\gamma_2$ acts by $+1$, which is the definition of $\m_1$.

\paragraph{Intrinsic action.}
Because $M_1\subset\B^3$ is invariant under $\rho(\dih_4)$, the ambient isometries restrict to an intrinsic isometric action of $\dih_4$ on the abstract Riemannian surface $(M_1,g)$. 
By Lemma~\ref{lem:fundamental_domain_torus}\,\ref{lem:fundamental_domain_torus-i}, the relative orientation representation is $\epsilon=\m_1$ which means that $\epsilon(\gamma_1)=-1$ and $\epsilon(\gamma_2)=1$. 
Since $\det\rho(\gamma_1)=1$ and $\det\rho(\gamma_2)=-1$, both $\gamma_1$ and $\gamma_2$ act as orientation-reversing isometries on $(M_1,g)$, while $\gamma_1\gamma_2$ acts as an orientation-preserving isometry of order~$4$.
The resulting configuration of fixed point sets is visualised in Figure~\ref{fig:D4torus}.

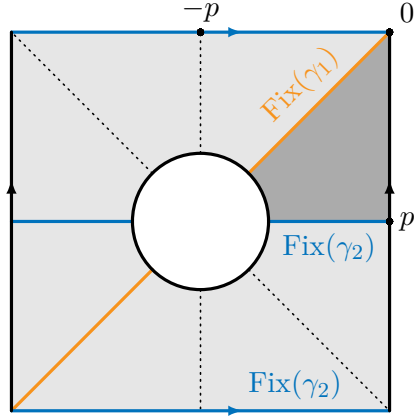
\begin{figure}
\centering
\pgfmathsetmacro{\rpar}{0.18} 
\begin{tikzpicture}[line cap=round,line join=round,scale=5,baseline={(0,0)}]  
\fill[fill=black!10](-1/2,-1/2)rectangle(1/2,1/2)(0,0)circle(\rpar);
\fill[black!33](0:\rpar)arc(0:45:\rpar)--(1/2,1/2)|-cycle;
\draw[dotted,semithick]
(0:\rpar)--(1/2,0)
(45:\rpar)--(1/2,1/2)
(90:\rpar)--(0,1/2)
(135:\rpar)--(-1/2,1/2)
(180:\rpar)--(-1/2,0)
(225:\rpar)--(-1/2,-1/2)
(270:\rpar)--(0,-1/2)
(-45:\rpar)--(1/2,-1/2)
; 
\draw[very thick,FarbeA]
(45:\rpar)--(1/2,1/2)node[midway,sloped,above]{$\Fix(\gamma_1)$}
(225:\rpar)--(-1/2,-1/2)
;
\draw[very thick,FarbeB]
(0:\rpar)--(1/2,0)node[midway,below]{$\Fix(\gamma_2)$}
(180:\rpar)--(-1/2,0)
(-1/2,-1/2)--++(1,0)node[near end,above]{$\Fix(\gamma_2)$}
(-1/2, 1/2)--++(1,0)
;
\draw[very thick]
(-1/2,-1/2)--++(0,1)
( 1/2,-1/2)--++(0,1)
(0,0)circle(\rpar);
\begin{scope}[thick]
\draw[-latex](1/2,0.1)--++(0,0.01);
\draw[-latex](-1/2,0.1)--++(0,0.01);
\draw[-latex,FarbeB](0.1,-1/2)--++(0.01,0);  
\draw[-latex,FarbeB](0.1,1/2)--++(0.01,0);  
\end{scope} 
\draw plot[bullet](1/2,1/2)node[above right]{$0$};
\draw plot[bullet](1/2,0)node[right]{$p$};
\draw plot[bullet](0,1/2)node[above]{$-p$};
\end{tikzpicture}
\caption{Schematic picture of the intrinsic $\dih_4$-action on $M_1$, which is topologically a torus with a disc removed. 
$M_1$ is divided into $8$ congruent domains (cf.~Lemma~\ref{lem:fundamental_domain_torus}).}
\label{fig:D4torus}
\end{figure}
 
The following statement is the analogue of Theorem~\ref{thm:k-noidDF} for punctured tori. 

\begin{theorem}
\label{thm:DF_punctured_torus}
Let $M_1\subset\B^3$ be a properly embedded free boundary minimal surface of genus one with connected boundary which is $\dih_4$-equivariant in the antiprismatic representation \eqref{eqn:antiprismatic}. 
Then the induced representation of $\dih_4$ on the space $\DF(M_1)$ of transverse-traceless tensors decomposes into irreducible representations of $\dih_4$ as
\[
\triv \oplus \m_1 \oplus \m_2.
\]
In particular, the dimensions of the isotypic components are:
\begin{enumerate}[label={\normalfont(\roman*)}]
\item $\dim \DF^{\triv}(M_1) = \dim \DF^{\m_1}(M_1) = \dim \DF^{\m_2}(M_1) = 1$,
\item $\dim \DF^{\Or}(M_1) = \dim \DF^{\st_1}(M_1) = 0$.
\end{enumerate}
\end{theorem}

\begin{proof}
By Lemma~\ref{lem:chi_DF}, we may compute the character $\chi_{\mathrm{tt}}$ of the representation of $\dih_4$ on $\DF(M_1)$ using the Eichler trace formula stated in Theorem~\ref{thm:Eichler} on the Schottky double $\hat{M}_1$ of genus~$2$. 
As in the proof of Theorem~\ref{thm:k-noidDF}, we apply Theorem~\ref{thm:Eichler} with $T=\gamma$ if $\gamma\in\dih_4$ preserves the orientation of $M_1$ and with $T=\gamma\iota$ otherwise, where $\iota\colon \hat{M}_1\to\hat{M}_1$ is the involution \eqref{eqn:20260927}. 
As shown above, $\gamma_1$ and $\gamma_2$ reverse the orientation of $M_1$, so the four reflections in $\dih_4$ reverse the orientation while the rotations $(\gamma_1\gamma_2)^s$ preserve it. 
Since characters are constant on conjugacy classes, it suffices to consider $\gamma_1$, $\gamma_2$, $\gamma_1\gamma_2$, $(\gamma_1\gamma_2)^2$ and $e$. 

Let $\gamma=\gamma_1$. 
Since $\rho(\gamma_1)$ fixes exactly the $x_1$-axis, the fixed points of $\gamma_1\iota$ are the two endpoints of the arc $\xi_1\subset M_1$ from Lemma~\ref{lem:fundamental_domain_torus}\,\ref{lem:fundamental_domain_torus-i}. 
As in the proof of Theorem~\ref{thm:k-noidDF}, $\gamma_1\iota$ looks like $z\mapsto-z$ around them, so that 
\[
\chi_{\mathrm{tt}}(\gamma_1)=2\frac{(-1)^2}{1-(-1)}=1. 
\]
Let $\gamma=\gamma_2$. 
Since $\gamma_2$ reverses the orientation of $M_1$ and hence of the closed curve $\partial M_1$, it fixes exactly two points of $\partial M_1$, which are the fixed points of $\gamma_2\iota$. 
Around them, $\gamma_2\iota$ again looks like $z\mapsto-z$, so that 
\[
\chi_{\mathrm{tt}}(\gamma_2)=1. 
\]
The rotation $\gamma_1\gamma_2\in\dih_4$ has two fixed points: the origin, which lies in the interior of $M_1$ since $\xi_1\subset M_1$, and its copy on the double $\hat{M}_1$. 
(Indeed, $\rho(\gamma_1\gamma_2)$ fixes only the origin and exchanges $\pm p$.)
Since $\xi_1\cup\xi_2\subset M_1$, the tangent plane of $M_1$ at the origin is horizontal, so $\gamma_1\gamma_2$ has the form $z\mapsto e^{\pm i\pi/2}z$ on $M_1$ and $z\mapsto e^{\mp i\pi/2}z$ on the copy. 
Hence, as in the proof of Theorem~\ref{thm:k-noidDF}, 
\[
\chi_{\mathrm{tt}}(\gamma_1\gamma_2)=2\operatorname{Re}\biggl(\frac{e^{i\pi}}{1-e^{i\pi/2}}\biggr)=-1-2\cos\frac{\pi}{2}=-1. 
\]
The rotation $(\gamma_1\gamma_2)^2$ by the angle $\pi$ around the $x_3$-axis fixes the three interior points $\{-p,0,p\}=M_1\cap\xi_3$ from Lemma~\ref{lem:fundamental_domain_torus}\,\ref{lem:fundamental_domain_torus-0} and their copies on $\hat{M}_1$. 
At all six fixed points, $(\gamma_1\gamma_2)^2$ looks like $z\mapsto-z$, so that 
\[
\chi_{\mathrm{tt}}\bigl((\gamma_1\gamma_2)^2\bigr)=6\frac{(-1)^2}{1-(-1)}=3. 
\]
Finally, $\chi_{\mathrm{tt}}(e)=\dim\DF(M_1)=6-6+3=3$ by Lemma~\ref{lem:DF_quadratic}. 
In summary, $\chi_{\mathrm{tt}}$ is given in the last row of the character table \eqref{eqn:char_table_M1}, whose columns correspond to the conjugacy classes of $\dih_4$. 
\begin{align}\label{eqn:char_table_M1}
\begin{array}{l||x{4.5em}|x{4.5em}|x{4.5em}|x{4.5em}|x{4.5em}}
& e & (\gamma_1\gamma_2)^{\pm1} & (\gamma_1\gamma_2)^2 & \gamma_1,~\gamma_2\gamma_1\gamma_2 & \gamma_2,~\gamma_1\gamma_2\gamma_1\\
\hline\hline
\chi_{\triv} & 1 & \hphantom{-}1 & 1 & \hphantom{-}1 & \hphantom{-}1\\
\chi_{\m_1} & 1 & -1 & 1 & -1 & \hphantom{-}1\\
\chi_{\m_2} & 1 & -1 & 1 & \hphantom{-}1 & -1\\
\hline
\chi_{\triv\oplus\m_1\oplus\m_2} & 3 & -1 & 3 & \hphantom{-}1 & \hphantom{-}1\\
\hline
\chi_{\mathrm{tt}} & 3 & -1 & 3 & \hphantom{-}1 & \hphantom{-}1
\end{array}
\end{align}
Since the last two rows coincide, $\DF(M_1)\cong\triv\oplus\m_1\oplus\m_2$ by \eqref{eqn:20260318}.
\end{proof}

The following statement is the analogue of Lemma~\ref{lem:spec_ener_knoid} for punctured tori. 

\begin{lemma}
\label{lem:spec_ener_torus}
Let $M_1\subset\B^3$ be any properly embedded free boundary minimal surface of genus one with connected boundary which is $\dih_4$-equivariant in the representation \eqref{eqn:antiprismatic}. 
Then 
\begin{align*}
\Ind_\area(M_1)&=\Ind_\jacobi(M_1)\otimes\m_1, &
\Nul_\area(M_1)&=\Nul_\jacobi(M_1)\otimes\m_1, &
\ind_\spec(M_1)&=1.
\end{align*} 
In particular, $\ind_\jacobi(M_1)=\ind_\area(M_1)$ and $\nul_\jacobi(M_1)=\nul_\area(M_1)$.
\end{lemma}

\begin{proof}
Lemma~\ref{lem:fundamental_domain_torus}\,\ref{lem:fundamental_domain_torus-i} in the appendix states that 
the relative orientation representation (see Definition~\ref{defn:relative_orientation_representation}) of $M_1\subset\B^3$ is necessarily given by $\epsilon=\m_1$. 
Proposition~\ref{prop:indJ_indA} therefore yields the equivariant isomorphisms 
\begin{align*}
\Ind_\area(M_1)&=\Ind_\jacobi(M_1)\otimes\m_1, &
\Nul_\area(M_1)&=\Nul_\jacobi(M_1)\otimes\m_1.
\end{align*}
Recalling the definitions from Section~\ref{sec:dihedral} we obtain 
\begin{align*}
\triv\otimes\m_1&\cong\m_1, &
\m_1\otimes\m_1&\cong\triv, &
\m_2\otimes\m_1&\cong\Or, &
\Or\otimes\m_1&\cong\m_2.
\end{align*}
Since $\st_1\otimes\m_1$ is again an irreducible representation of $\dih_4$ of degree $2$, and $\st_1$ is the only such representation of $\dih_4$ up to isomorphism, we conclude $\st_1\otimes\m_1\cong\st_1$. 
(This relation does not generalise to dihedral groups of higher order.)
Hence, $\varrho\mapsto\varrho'\vcentcolon=\varrho\otimes\m_1$ is a bijection $\IrRep{\dih_4}\to\IrRep{\dih_4}$ and we obtain
\begin{align*}
\ind_\jacobi(M_1)
&=\sum_{\varrho\in\IrRep{\dih_4}}\ind_\jacobi^{\varrho}(M_1)
=\sum_{\varrho\in\IrRep{\dih_4}}\ind_\area^{\varrho\otimes\m_1}(M_1)
=\sum_{\varrho'\in\IrRep{\dih_4}}\ind_\area^{\varrho'}(M_1)
=\ind_\area(M_1)
\end{align*}
and likewise $\nul_\jacobi(M_1)=\nul_\area(M_1)$ as claimed.

It remains to prove $\ind_\spec(M_1)=1$. 
As in \eqref{eqn:20260826} we have $Q_\spec(1,1)<0$ and thus $\ind_\spec(M_1)\geq1$.  
Lemma~\ref{lem:fundamental_domain_torus}\,\ref{lem:fundamental_domain_torus-ii}--\ref{lem:fundamental_domain_torus-iii} stated in the appendix verifies \cite[Assumption~5.1]{KusnerMcGrath2020} for $M_1$. 
Thus, the first Steklov eigenvalue on $M_1$, defined as in \eqref{eqn:sigma1def}, satisfies $\sigma_1(M_1)=1$
by \cite[Theorem~5.2]{KusnerMcGrath2020}. 
Therefore, $Q_\spec(u,u)\geq0$ for all $u\in C^\infty(M_1)$ with $\int_{\partial M_1}u=0$. 
Hence, $Q_\spec$ is positive semi-definite on a subspace of codimension one, which implies the reverse inequality $\ind_\spec(M_1)\leq1$ and completes the proof.
\end{proof}

\begin{proposition}
\label{prop:torus_isotypic}
Let $M_1\subset\B^3$ be any properly embedded free boundary minimal surface of genus one with connected boundary which is $\dih_4$-equivariant in the representation \eqref{eqn:antiprismatic}. 
Then $\ind_\ener(M_1)=3$ and 
$\ind_\area(M_1)=5$. 
Moreover the isotypic indices are given by
\begin{align*}
\ind_\ener^{\triv}(M_1)&=0, &
\ind_\ener^{\Or}(M_1)&=0, &
\ind_\ener^{\m_1}(M_1)&=1, &
\ind_\ener^{\m_2}(M_1)&=0, &
\ind_\ener^{\st_1}(M_1)&=2,
\\
\ind_\area^{\triv}(M_1)&=1, &
\ind_\area^{\Or}(M_1)&=0, &
\ind_\area^{\m_1}(M_1)&=2, &
\ind_\area^{\m_2}(M_1)&=0, &
\ind_\area^{\st_1}(M_1)&=2.
\end{align*}
\end{proposition}

\begin{proof}
Since $\ind_\spec(M_1)=1$ by Lemma~\ref{lem:spec_ener_torus}, Corollary~\ref{cor:indE} yields
$\Ind_\ener(M_1)\leq\cA_\rho=\st_1\oplus\m_1$ and therefore
\begin{align}
\label{eqn:torus_ener_index_bounds}
\ind_\ener^{\m_1}(M_1)&\leq1, &
\ind_\ener^{\st_1}(M_1)&\leq2, &
\ind_\ener^{\varrho}(M_1)&=0 \text{ for any }\varrho\in\{\triv,\Or,\m_2\}.
\end{align}
By Theorem~\ref{thm:equivariant_area_bounds}, for every $\varrho\in\IrRep{\dih_4}$, the isotypic indices satisfy 
\begin{align*}
\ind_\area^{\varrho}(M_1)
&\leq
\ind_\ener^{\varrho}(M_1)
+\dim\DF^{\varrho}(M_1).   
\end{align*} 
Inserting \eqref{eqn:torus_ener_index_bounds} and 
recalling from Theorem~\ref{thm:DF_punctured_torus} that the representation class of $\DF(M_1)$ is $\triv\oplus\m_1\oplus\m_2$, we obtain
\begin{align}
\label{eqn:torus_area_index_bounds}
\ind_\area^{\triv}(M_1)&\leq1, &
\ind_\area^{\Or}(M_1)&=0, &
\ind_\area^{\m_1}(M_1)&\leq2, &
\ind_\area^{\m_2}(M_1)&\leq1, &
\ind_\area^{\st_1}(M_1)&\leq2,
\end{align}
and in particular $\ind_\area(M_1)\leq6$.

We claim $\ind_\jacobi^{\Or}(M_1)=\nul_\jacobi^{\Or}(M_1)=0$.
By Lemma~\ref{lem:fundamental_domain_torus}\,\ref{lem:fundamental_domain_torus-iii}, the union $X$ of the sets fixed by one of the reflections in $\dih_4$ divides $M_1$ into $8$ congruent domains (each mapped homeomorphically onto the quadrilateral $\pi_G(M_1)$).
A function $u\in C^\infty(M_1)$ satisfying $u=\Or(\gamma)\,(u\circ\gamma^{-1})=-u\circ\gamma^{-1}$ for every reflection $\gamma\in\dih_4$ vanishes on all of $X$.
Consequently, any such $u$ has at least $8$ nodal domains.
If $u$ were an eigenfunction of $Q_\jacobi$ with nonpositive eigenvalue, 
then \cite[Corollary~3.2\,(i)]{CSWSpectral} applied with $n=8$ would imply $\ind_\jacobi(M_1)\geq7$, contradicting $\ind_\jacobi(M_1)=\ind_\area(M_1)\leq6$.  
This proves the claim, so by Lemma~\ref{lem:spec_ener_torus} and $\Or\otimes\m_1\cong\m_2$,
\begin{align}
\label{eqn:torus_m2_vanishes}
\ind_\area^{\m_2}(M_1)&=0, & \nul_\area^{\m_2}(M_1)&=0.
\end{align}

Substituting \eqref{eqn:torus_m2_vanishes} into \eqref{eqn:torus_area_index_bounds} yields $\ind_\area(M_1)\leq5$, while $\ind_\area(M_1)\geq5$ by Lemma~\ref{lem:index_lower_bound_genus}.
Hence, $\ind_\area(M_1)=5$ and all the inequalities in \eqref{eqn:torus_ener_index_bounds}--\eqref{eqn:torus_area_index_bounds} (except for $\m_2$) must be equalities.
This proves the statement about the isotypic indices and, summing \eqref{eqn:torus_ener_index_bounds}, also $\ind_\ener(M_1)=3$.
\end{proof}

\begin{proposition}
\label{prop:torus_nullity}
Let $M_1\subset\B^3$ be any properly embedded free boundary minimal surface of genus one with connected boundary which is $\dih_4$-equivariant in the representation \eqref{eqn:antiprismatic}. 
Then 
\[
\nul_\ener(M_1)=3=\nul_\area(M_1).
\]
\end{proposition}

\begin{proof}
By Lemma~\ref{lem:fundamental_domain_torus}\,\ref{lem:fundamental_domain_torus-ii}--\ref{lem:fundamental_domain_torus-iii} stated in the appendix, $M_1$ satisfies the hypotheses \cite[Assumption~5.1]{KusnerMcGrath2020} 
and \cite[Theorem~5.2]{KusnerMcGrath2020} applies. 
Therefore, the first Steklov eigenspace of $M_1$ coincides with the span of the ambient coordinate functions restricted to $M_1$ and $\sigma_1(M_1)=1$. 
Recalling Remark~\ref{rem:spectral_nullity}, this yields $\nul_\spec(M_1)=3$, since the coordinate functions span the Steklov eigenspace corresponding to the eigenvalue $1$. 
Together with $\ind_\spec(M_1)=1$ from Lemma~\ref{lem:spec_ener_torus} and $\ind_\ener(M_1)=3$ from Proposition~\ref{prop:torus_isotypic}, and since the components of the inclusion $\Phi\colon M_1\to\B^3$ 
are quadratically independent on $\partial M_1$  
by Lemma~\ref{lem:torus_quadratic_independence},
Proposition~\ref{prop:rigidity} applies with $n=2$ and yields 
\begin{align}
\label{eqn:torus_ener_nullity}
\nul_\ener(M_1)=\tfrac{1}{2}n(n+1)=3.
\end{align}
Theorem~\ref{thm:equivariant_area_bounds} implies, for every $\varrho\in\IrRep{\dih_4}$, 
\begin{align}\label{eqn:20260827a}
\ind_\area^{\varrho}(M_1)+\nul_\area^{\varrho}(M_1)
\leq
\ind_\ener^{\varrho}(M_1)+\nul_\ener^{\varrho}(M_1)+\dim\DF^{\varrho}(M_1).
\end{align}
Proposition~\ref{prop:torus_isotypic} and Theorem~\ref{thm:DF_punctured_torus} yield 
\begin{align}\label{eqn:20260827b}
\begin{array}{r|x{2em}x{2em}x{2em}x{2em}x{2em}}
\varrho\in\IrRep{\dih_4} & {\triv} & \Or & \m_1 & \m_2 & \st_1
\\\hline\hline
\ind_\area^{\varrho}(M_1) & 1 & 0 & 2 & 0 & 2
\\\hline
\ind_\ener^{\varrho}(M_1) & 0 & 0 & 1 & 0 & 2
\\
\dim\DF^{\varrho}(M_1) & 1 & 0 & 1 & 1 & 0
\end{array}
\end{align}
Inserting the values in table \eqref{eqn:20260827b} columnwise into \eqref{eqn:20260827a}
yields $\nul_\area^{\varrho}(M_1)\leq\nul_\ener^{\varrho}(M_1)$
for every $\varrho\in\{\triv, \Or, \m_1, \st_1\}$, 
while $\nul_\area^{\m_2}(M_1)=0$ by \eqref{eqn:torus_m2_vanishes}. 
Summing over all $\varrho\in\IrRep{\dih_4}$ and recalling \eqref{eqn:torus_ener_nullity} yields $\nul_\area(M_1)\leq3$.
The reverse inequality $\nul_\area(M_1)\geq3$ holds by Lemma~\ref{lem:index_lower_bound_genus}, which completes the proof. 
\end{proof}

\begin{proof}[Proof of Theorem~\ref{thm:index-fbm-torus}]
The statement is covered by Propositions~\ref{prop:torus_nullity} and \ref{prop:torus_isotypic}. 
\end{proof}

\appendix

\section{Decomposition of symmetric bilinear forms}
\label{sec:decomposition}

Throughout this section, let $(M,g)$ be a compact Riemannian surface, possibly with boundary, and let $\eta$ denote the outward unit normal along $\partial M$. 
We consider the spaces  
\begin{align*}
\mathcal{H}&\vcentcolon=L^2\bigl(S^2_0T^*M\bigr), &
\mathcal{V}&\vcentcolon=\{X\in H^1(TM)\mid g(X,\eta)=0 \text{ on } \partial M\}, 
\end{align*}
where $S^2_0T^*M$ denotes the bundle of trace-free symmetric bilinear forms, and the boundary condition is understood in the Sobolev
trace sense. 
We define the bounded operator
\begin{align*}
L\colon\mathcal{V}&\to\mathcal{H}, 
\\
X&\mapsto\mathcal{L}_Xg-(\divop_gX)g
\end{align*}
(extending the operator from Lemma~\ref{lem:L_operator}) 
and set 
$\mathcal{V}_0\vcentcolon=\mathcal{V}\cap\VConf^\perp$, where $\VConf=\ker L$ and the orthogonal complement is taken with respect to the $L^2$-inner product. 
We write 
\[
\delta\vcentcolon=-\divop_g
\]
for the divergence operator acting on symmetric bilinear forms.
The main purpose of this appendix is to prove the following theorem used for \eqref{eq:metric_decomp}. 

\begin{theorem}\label{thm:surface-tensor-decomposition}
Every smooth symmetric bilinear form $h$ on $M$ has an orthogonal decomposition
\[
h=\tfrac{1}{2}(\tr_g h)g+LX+h^{\mathrm{tt}}
\]
into pure trace, pure divergence and transverse traceless parts, where $X$ is smooth and tangent to $\partial M$, and
\begin{align*}
\tr_g h^{\mathrm{tt}}&=0, &
\delta h^{\mathrm{tt}}&=0, &
h^{\mathrm{tt}}(\eta,\cdot)|_{T\partial M}&=0.
\end{align*}
The three tensor summands are pairwise $L^2$-orthogonal and uniquely determined by $h$, while $X$ is unique up to an element of $\VConf(M,g)$.
\end{theorem}

This theorem is well-known for closed manifolds \cite{BergerEbin1969, York1974}. 
In the case of surfaces, the transverse traceless part can be identified with the space of real parts of holomorphic quadratic differentials, see e.\,g.~\cite{BuzanoRupflin2017}, and Lemma~\ref{lem:DF_quadratic} below makes this precise for surfaces with boundary.
We were not able to find a proof of Theorem~\ref{thm:surface-tensor-decomposition} 
for surfaces with boundary
in the literature, and we therefore include one for completeness.

We first establish the estimate used to construct the decomposition. It should be compared with Korn's inequality from mathematical elasticity and our proof closely follows that from~\cite{BauerPauly2016}.

\begin{lemma}
\label{lem:surface-L-coercivity}
There is a constant $C>0$ such that for any $X\in \mathcal V$ 
\begin{align}
\label{app:LM_prelim}
\nm{X}_{H^1}^2&\leq 
C\bigl(\nm{LX}_{L^2}^2+\nm{X}_{L^2}^2\bigr).
\intertext{In particular, the space $\VConf$ is finite-dimensional, and there is $c>0$ such that for any $X\in \mathcal V_0$}
\label{app:L-coercivity}
\nm{LX}_{L^2}^2&\geq c\nm{X}_{H^1}^2.
\end{align}
\end{lemma}

\begin{proof}
Let $X\in\mathcal{V}$ be smooth and set $\alpha\vcentcolon=X^\flat$. 
With the convention $\II(\xi,\zeta)=g(\nabla_\xi\eta,\zeta)$ for the second fundamental form of $\partial M$, the integrated Weitzenböck formula reads 
\begin{align*}
\nm{d\alpha}_{L^2}^2+\nm{\delta\alpha}_{L^2}^2
&=\nm{\nabla X}_{L^2}^2
+\int_M K_g\abs{X}^2_g
+\int_{\partial M}\II(X,X), 
\end{align*}
where $K_g$ denotes the Gauß curvature of $(M,g)$. 
Note that the boundary term takes this form because $X$ is tangent to $\partial M$. 
The orthogonal splitting of $\nabla\alpha$ into its trace-free symmetric,
pure trace, and skew-symmetric parts gives
\[
\abs{\nabla X}_g^2
=\tfrac{1}{4}\abs{LX}_g^2+\tfrac{1}{2}\abs{\delta\alpha}_g^2+\tfrac{1}{2}\abs{d\alpha}_g^2.
\]
Here the norm on $2$-forms is the standard exterior-algebra norm, given in an orthonormal frame by $\abs{\omega}_g^2=\sum_{i<j}\abs{\omega_{ij}}^2$. 
Thus the skew-symmetric part of $\nabla\alpha$, whose components are $\tfrac{1}{2}(d\alpha)_{ij}$, has squared tensor norm $\tfrac{1}{2}\abs{d\alpha}_g^2$. 
Combining the preceding identities yields
\[
\nm{\nabla X}_{L^2}^2
=\frac{1}{2}\nm{LX}_{L^2}^2+\int_M K_g\abs{X}_g^2+\int_{\partial M}\II(X,X).
\]
The curvature terms are bounded by a constant times $\nm{X}_{L^2(M)}^2+\nm{X}_{L^2(\partial M)}^2$, and the trace inequality 
\[
\nm{X}_{L^2(\partial M)}^2\leq\varepsilon\nm{\nabla X}_{L^2(M)}^2+C_\varepsilon\nm{X}_{L^2(M)}^2
\]
therefore proves~\eqref{app:LM_prelim} by absorption, first for smooth $X$ and then for all $X\in\mathcal{V}$ by density. 

On $\VConf$, estimate~\eqref{app:LM_prelim} makes the $H^1$- and
$L^2$-norms equivalent. 
Since the embedding $H^1\hookrightarrow L^2$ is compact, the $H^1$-unit ball in $\VConf$ is compact, which implies that $\VConf$ is finite-dimensional.

To prove~\eqref{app:L-coercivity} we argue by contradiction. 
Suppose there exists a sequence $X_j\in\mathcal V_0$
with $\nm{X_j}_{H^1}=1$ and $LX_j\to 0$ in $L^2$.
After passing to a subsequence, $X_j\rightharpoonup X$ weakly in $H^1$ and $X_j\to X$ in $L^2$. 
Since the operator $L$ is bounded, $LX_j\rightharpoonup LX$, and thus $LX=0$. 
Since $\mathcal{V}_0$ is weakly closed, we have $X\in\mathcal{V}_0$, while $LX=0$ implies $X\in\VConf$. 
Hence, $X\in\mathcal{V}_0\cap\VConf=\{0\}$, so $X_j\to0$ in $L^2$ and \eqref{app:LM_prelim} implies $X_j\to0$ in $H^1$, contradicting $\nm{X_j}_{H^1}=1$. 
\end{proof}

Lemma~\ref{lem:surface-L-coercivity} allows us to construct the $L^2$-version of the decomposition in Theorem~\ref{thm:surface-tensor-decomposition}. 
Let $h_0\in\mathcal H$. 
On $\mathcal V_0$, the bilinear form
\[
a(X,Y)\vcentcolon=\int_M\brt{LX,LY}_g
\]
is bounded and coercive by Lemma~\ref{lem:surface-L-coercivity}.
The Lax--Milgram theorem therefore yields a unique $X\in\mathcal V_0$ such that for any $Y\in\mathcal V_0$  
\begin{align}
\label{app:Lax-Milgram}
\int_M\brt{LX,LY}_g&=\int_M\brt{h_0,LY}_g.
\end{align}
Since $LY=0$ for $Y\in\VConf$, identity \eqref{app:Lax-Milgram} holds trivially on $\VConf$. 
As $\VConf$ is finite-dimensional by Lemma~\ref{lem:surface-L-coercivity}, we have $\mathcal{V}=\mathcal{V}_0\oplus\VConf$, and \eqref{app:Lax-Milgram} therefore holds for every $Y\in\mathcal{V}$. 
Consequently, $h^{\mathrm{tt}}\vcentcolon=h_0-LX\in(L\mathcal{V})^\perp$.
Applying this construction to $h_0=h-\frac12(\tr_g h)g$ gives the decomposition
\begin{align}
\label{dec:L2}
h&=\tfrac{1}{2}(\tr_g h)g + LX + h^{\mathrm{tt}}.
\end{align}
It remains to show that $h^{\mathrm{tt}}$ is indeed transverse traceless and to establish the regularity of the terms in the decomposition~\eqref{dec:L2}.

\begin{lemma}
\label{lem:surface-TT-regularity}
The orthogonal complement $(L\mathcal{V})^\perp\subset\mathcal{H}$ consists of the smooth trace-free symmetric bilinear forms $k$ satisfying 
\begin{align*}
\delta k&=0, & k(\eta,\cdot)\vert_{T\partial M}&=0.
\end{align*}
\end{lemma}

\begin{proof}
Let $k\in(L\mathcal V)^\perp$. 
In local conformal coordinates with $g=e^{2\omega}(dx^2+dy^2)$ we write 
\begin{align}
\label{eqn:k_in_local_coordinates}
k=a\,(dx\otimes dx-dy\otimes dy)+b\,(dx\otimes dy+dy\otimes dx),
\end{align}
where $a,b\in L^2_{\mathrm{loc}}$. 
For $Y=\varphi\partial_x+\psi\partial_y$ supported in the chart, a
direct computation gives
\begin{align}
\label{app:surface-local-pairing}
0=\frac{1}{2}\int_M\brt{k,LY}_g
=\int\bigl(a(\varphi_x-\psi_y)+b(\varphi_y+\psi_x)\bigr)\,dx\,dy.
\end{align}
Here we used the conformal covariance $L_{e^{2\omega}g}=e^{2\omega}L_g$ of the operator $L$.
Integration by parts in \eqref{app:surface-local-pairing} yields
\begin{align}
\label{app:surface-CR}
a_x+b_y&=0, & 
b_x-a_y&=0
\end{align}
in the weak sense, so $\Delta a=\Delta b=0$ and Weyl's lemma yields interior smoothness. 
Equivalently, $a-ib$ is holomorphic.

To establish regularity up to the boundary, we apply the Schwarz reflection principle in its weak form. 
At a boundary point we consider a conformal boundary chart given by a half-disc $D^+=\{(x,y)\mid x^2+y^2<1,~y\geq0\}$ with boundary arc $I=\{y=0\}$ and $g=e^{2\omega}(dx^2+dy^2)$, see e.\,g.~\cite{YangZhou2021}. 
In a boundary chart, identity~\eqref{app:surface-local-pairing} holds
for all smooth $\varphi,\psi$ supported away from the semicircular
part of $\partial D^+$ as long as $\psi|_I=0$.
Extending $a$ evenly and $b$ oddly across $I$, we obtain functions $\tilde a,\tilde b\in L^2(D)$ on the full disc $D$. 
Given $\zeta\in C_c^\infty(D)$, we apply \eqref{app:surface-local-pairing} on $D^+$ to the pairs 
\begin{align*}
(\varphi,\psi)&=\bigl(\zeta(x,y)+\zeta(x,-y),0\bigr), &
(\varphi,\psi)&=\bigl(0,\zeta(x,y)-\zeta(x,-y)\bigr), 
\end{align*}
which satisfy $\psi\vert_I=0$. 
Changing variables on the reflected half-disc gives
\begin{align*}
\int_D\bigl(\tilde a\,\zeta_x+\tilde b\,\zeta_y\bigr)\,dx\,dy&=0, &
\int_D\bigl(\tilde b\,\zeta_x-\tilde a\,\zeta_y\bigr)\,dx\,dy&=0. 
\end{align*}
Thus, $\tilde a,\tilde b$ satisfy~\eqref{app:surface-CR} on $D$
in a weak sense and are smooth by Weyl's lemma.
In particular, $k$ is smooth up to $I$, and oddness of $\tilde b$ gives $b\vert_I=0$. 
In these coordinates,
\[
\delta k=-e^{-2\omega}
\bigl((a_x+b_y)\,dx+(b_x-a_y)\,dy\bigr)= 0,
\]
by \eqref{app:surface-CR}, 
and $k(\eta,\partial_x)$ is a nonzero smooth multiple of $b$ along $I$.
Since $b\vert_I=0$, both asserted conditions follow, without assuming a classical boundary trace for $k$ in advance. 

Conversely, if $k$ is smooth, trace-free and satisfies these conditions, Green's formula yields
\[
\frac{1}{2}\int_M\brt{k,LY}_g
=\int_M(\delta k)(Y)+\int_{\partial M}k(\eta,Y)
=0
\]
for any $Y\in \mathcal V$, since $\delta k=0$ and $Y$ is tangent to $\partial M$. 
\end{proof}

\begin{lemma}
\label{lem:surface-L-regularity}
If $X\in\mathcal{V}$ and $LX$ is smooth, then $X$ is smooth up to the boundary. 
In particular, every element of $\VConf$ is smooth, so $\VConf$ coincides with the space of conformal vector fields tangent to $\partial M$ from Lemma~\ref{lem:L_operator}. 
\end{lemma}

\begin{proof}
In a conformal chart we write $X=u\partial_x+v\partial_y$. 
By conformal covariance of $L$, 
\[
LX=e^{2\omega}\bigl((u_x-v_y)(dx\otimes dx-dy\otimes dy)+(u_y+v_x)(dx\otimes dy+dy\otimes dx)\bigr),  
\]
so that $u_x-v_y=A$ and $u_y+v_x=B$ for smooth functions $A,B$, and consequently 
\begin{align*}
u_{xx}+u_{yy}&=A_x+B_y, & v_{xx}+v_{yy}&=B_x-A_y.
\end{align*}
Interior regularity gives smoothness of $u,v$ in interior charts.
In a boundary chart, tangency of $X$ means $v|_I=0$ in the Sobolev trace sense. 
Local regularity for the Dirichlet boundary problem~\cite[§\,8.3]{GilbargTrudinger1983} gives smoothness of $v$ up to $I$. 
Finally, $u_x=A+v_y$ and $u_y=B-v_x$ give the regularity of $u$ up to the boundary. 
\end{proof}

\begin{proof}[Proof of Theorem~\ref{thm:surface-tensor-decomposition}]
Let $h$ be a smooth symmetric bilinear form decomposed as in \eqref{dec:L2}. 
By Lemma~\ref{lem:surface-TT-regularity}, the tensor $h^{\mathrm{tt}}$ is smooth and transverse traceless, so $LX=h-\tfrac{1}{2}(\tr_gh)g-h^{\mathrm{tt}}$ is smooth and $X$ is smooth up to the boundary by Lemma~\ref{lem:surface-L-regularity}.   
Given two such decompositions, their difference yields $LX\in(L\mathcal{V})^\perp\cap L\mathcal{V}=\{0\}$, so $h^{\mathrm{tt}}$ is unique and $X$ is unique up to $\ker L=\VConf(M,g)$. 
\end{proof}

We conclude this appendix with a characterisation of the space 
\[
\DF(M)=\{ h \in \mathcal{S}(M) \mid \tr_g h = 0,~\divop_g h = 0,~h(\eta, \cdot)\vert_{T\partial M} = 0 \}
\]
of transverse traceless symmetric bilinear forms in terms of holomorphic quadratic differentials. 
We assume that $(M,g)$ is a compact, connected and oriented Riemannian surface with $\beta\geq1$ boundary components.
We choose orientation-preserving local complex coordinates $z=x+iy$ on $M$ which are conformal in the sense that $g=e^{2\omega}(dx^2+dy^2)$ for some function $\omega$. 
Near boundary points we choose such coordinates in an upper half-disc such that they are real valued on $\partial M$. 

\paragraph{Schottky double.}
Let $\hat{M}$ denote the \emph{Schottky double} of $M$, i.\,e.~the closed Riemann surface obtained by gluing $M$ and a copy of $M$ with reversed orientation along their boundaries \cite[§\,II.4.5]{Farkas1992}. 
Local coordinates on $\hat{M}$ are given by orientation-preserving conformal complex coordinates $z$ on $M$, by their complex conjugates $\overline{z}$ on the copy, and near $\partial M$ by combining the boundary coordinates with their conjugates, which is consistent by the Schwarz reflection principle. 
The map 
\begin{align}
\label{eqn:20260927}
\begin{aligned}
\iota\colon\hat{M}&\to\hat{M}\\
z&\mapsto\overline{z}
\end{aligned}
\end{align}
exchanging the two copies of $M$ is an anti-holomorphic involution fixing $\partial M$ pointwise. 
Since the Euler characteristic of $\hat{M}$ is twice the Euler characteristic of $M$, we have 
\[
\operatorname{genus}(\hat{M})=2\operatorname{genus}(M)+\beta-1.
\] 

\paragraph{Quadratic differentials.}
A \emph{holomorphic quadratic differential} $\psi$ on $\hat{M}$ is a holomorphic section of the symmetric square of the holomorphic cotangent bundle. 
In each local coordinate $z$ it is given by $\psi=f(z)\,dz^2$ for some holomorphic function $f$.  
The holomorphic quadratic differentials form a complex vector space $\mathcal{Q}(\hat{M})$ of complex dimension $3\operatorname{genus}(\hat{M})-3$ if $\operatorname{genus}(\hat{M})\geq2$, see  \cite[Proposition~III.5.2]{Farkas1992}.
Given $\psi\in\mathcal{Q}(\hat{M})$, the complex conjugate $\overline{\iota^*\psi}$ of its pullback by the map \eqref{eqn:20260927} is again a holomorphic quadratic differential, given by $\overline{f(\overline{z})}\,dz^2$ in local coordinates in which $\iota(z)=\overline{z}$. 
The map $\psi\mapsto\overline{\iota^*\psi}$ is a complex anti-linear involution, and we denote its fixed point set by 
\begin{align}
\label{eqn:Q_real}
\mathcal{Q}(\hat{M})^\iota\vcentcolon=\{\psi\in\mathcal{Q}(\hat{M})\mid\overline{\iota^*\psi}=\psi\}.
\end{align}
Then $\mathcal{Q}(\hat{M})=\mathcal{Q}(\hat{M})^\iota\oplus i\,\mathcal{Q}(\hat{M})^\iota$, because $\psi=\frac{1}{2}\bigl(\psi+\overline{\iota^*\psi}\bigr)+i\,\frac{1}{2i}\bigl(\psi-\overline{\iota^*\psi}\bigr)$ for every $\psi\in\mathcal{Q}(\hat{M})$, where 
$\frac{1}{2}\bigl(\psi+\overline{\iota^*\psi}\bigr),\,\frac{1}{2i}\bigl(\psi-\overline{\iota^*\psi}\bigr)\in\mathcal{Q}(\hat{M})^\iota$. 
In particular, 
\(
\dim_\R\mathcal{Q}(\hat{M})^\iota=\dim_\C\mathcal{Q}(\hat{M}).
\)

\begin{lemma}
\label{lem:DF_quadratic}
Given $h\in\DF(M)$, extended complex-bilinearly, the quadratic differential   
\[
\psi_h\vcentcolon=2h(\partial_z,\partial_z)\,dz^2
\]
on $M$ extends to a unique element of $\mathcal{Q}(\hat{M})^\iota$ and the map $h\mapsto\psi_h$ is a real-linear isomorphism 
\[
\DF(M)\to\mathcal{Q}(\hat{M})^\iota
\]
with inverse $\psi\mapsto\operatorname{Re}(\psi\vert_M)$. 
If $2\operatorname{genus}(M)+\beta\geq3$ then 
$\dim\DF(M)=6\operatorname{genus}(M)-6+3\beta$. 
\end{lemma}

\begin{proof}
In the given local coordinates $z=x+iy$ any trace-free $h\in\mathcal{S}(M)$ takes the form 
\[
h=a\,(dx\otimes dx-dy\otimes dy)+b\,(dx\otimes dy+dy\otimes dx),
\] 
as in \eqref{eqn:k_in_local_coordinates}, 
and a direct computation using $\partial_z=\frac{1}{2}(\partial_x-i\partial_y)$ and $dz=dx+i\,dy$ yields
\begin{align*}
2h(\partial_z,\partial_z)&=a-ib, &
h&=\operatorname{Re}\bigl((a-ib)\,dz^2\bigr).
\end{align*}
Recalling $g=e^{2\omega}(dx^2+dy^2)$ we also compute 
\[
\divop_g h=e^{-2\omega}\bigl((a_x+b_y)\,dx+(b_x-a_y)\,dy\bigr),
\]
which vanishes if and only if $a-ib$ satisfies the Cauchy--Riemann equations \eqref{app:surface-CR}, i.\,e.~if and only if $\psi_h$ is holomorphic in the interior of $M$. 
In boundary coordinates, the outward unit normal is $\eta=-e^{-\omega}\partial_y$ and $T\partial M$ is spanned by $\partial_x$, so $h(\eta,\cdot)\vert_{T\partial M}=0$ if and only if $b=0$ on the real axis, i.\,e.~if and only if $a-ib$ is real there. 
By the Schwarz reflection principle, the latter holds if and only if $\psi_h$ extends holomorphically to $\hat{M}$ by $\overline{\iota^*\psi_h}$ on the copy of $M$, which proves the first two claims. 
Since $\operatorname{genus}(\hat{M})=2\operatorname{genus}(M)+\beta-1\geq2$ by assumption, the dimension formula follows from
\[
\dim\DF(M)=\dim_\R\mathcal{Q}(\hat{M})^\iota=\dim_\C\mathcal{Q}(\hat{M})=3\operatorname{genus}(\hat{M})-3.
\qedhere
\]
\end{proof}

For the following lemma we recall the induced representation on $\DF(M)$ from Section~\ref{sub:equivariant-area} and the notion of characters from \eqref{eqn:character}. 

\begin{lemma}\label{lem:chi_DF}
Let $G$ be a finite group acting on $(M,g)$ by isometries. 
Then each $\gamma\in G$ extends to the Schottky double $\hat{M}$ such that 
\begin{align}\label{eqn:action_Q}
\gamma\cdot\psi\vcentcolon=
\begin{cases}
(\gamma^{-1})^*\psi & \text{if $\gamma$ preserves the orientation of $M$,}\\
(\iota\gamma^{-1})^*\psi & \text{if $\gamma$ reverses the orientation of $M$}
\end{cases}
\end{align}
defines a representation of $G$ on $\mathcal{Q}(\hat{M})$ by $\C$-linear maps. 
Its character coincides with the character $\chi_{\mathrm{tt}}$ of the representation $h\mapsto(\gamma^{-1})^*h$ of $G$ on $\DF(M)$. 
\end{lemma}

\begin{proof}
Each $\gamma\in G$ is a conformal diffeomorphism of $M$ which either preserves or reverses the orientation. 
Since the map $\iota$ from \eqref{eqn:20260927} fixes $\partial M$ pointwise, $\gamma$ extends continuously to the Schottky double $\hat{M}$ by setting it to be equal to $\iota\gamma\iota$ on the copy. 
By the Schwarz reflection principle, the extension, again denoted by $\gamma$, is holomorphic respectively anti-holomorphic and commutes with $\iota$. 

Since $\gamma^{-1}$ respectively $\iota\gamma^{-1}$ is holomorphic and commutes with $\iota$, and since $\iota^2=\mathrm{id}$, \eqref{eqn:action_Q} defines a representation of $G$ on $\mathcal{Q}(\hat{M})$ by $\C$-linear maps which preserve $\mathcal{Q}(\hat{M})^\iota$ defined in \eqref{eqn:Q_real}. 
Moreover, in the notation of Lemma~\ref{lem:DF_quadratic}, 
\[
\psi_{(\gamma^{-1})^*h}=\gamma\cdot\psi_h
\]
for all $h\in\DF(M)$, because both sides belong to $\mathcal{Q}(\hat{M})^\iota$ and their real parts on $M$ coincide. 
Indeed, if $\gamma$ reverses the orientation, then $\iota^*\psi_h=\overline{\psi_h}$ implies 
\[
(\iota\gamma^{-1})^*\psi_h=(\gamma^{-1})^*\iota^*\psi_h=(\gamma^{-1})^*\overline{\psi_h}
\]
on $M$, and its real part is $(\gamma^{-1})^*h$. 
Hence, the isomorphism $h\mapsto\psi_h$ from Lemma~\ref{lem:DF_quadratic} identifies the representation of $G$ on $\DF(M)$ with the restriction of \eqref{eqn:action_Q} to $\mathcal{Q}(\hat{M})^\iota$. 
Since $\mathcal{Q}(\hat{M})=\mathcal{Q}(\hat{M})^\iota\oplus i\,\mathcal{Q}(\hat{M})^\iota$, any real basis of $\mathcal{Q}(\hat{M})^\iota$ is also a complex basis of $\mathcal{Q}(\hat{M})$, with respect to which both representations are given by the same real matrices and hence have the same traces. 
\end{proof}
The character in Lemma~\ref{lem:chi_DF} can be computed with the following classical result. 
\begin{theorem}[Eichler trace formula {\cite[Theorem~V.2.9]{Farkas1992}}]
\label{thm:Eichler}
Let $T\neq\mathrm{id}$ be a holomorphic automorphism of finite order of a closed Riemann surface of genus at least $2$. 
Then $T$ has finitely many fixed points, and the trace of the complex-linear map $\psi\mapsto(T^{-1})^*\psi$ on the space of holomorphic quadratic differentials is given by 
\[
\sum_{p\in\Fix(T)}\frac{\zeta_p^2}{1-\zeta_p}, 
\]
where $T^{-1}\colon z\mapsto\zeta_p z$ in a suitable local holomorphic coordinate centred at $p$, so that $\zeta_p\in\C\setminus\{1\}$ is a root of unity and equals the derivative of $T^{-1}$ at $p$ in any local coordinate centred at $p$. 
\end{theorem}

\begin{remark}
\label{rem:Eichler}
In the setting of Lemma~\ref{lem:chi_DF} with $2\operatorname{genus}(M)+\beta\geq3$, Theorem~\ref{thm:Eichler} applies to $\hat{M}$, and the representation \eqref{eqn:action_Q} of $\gamma\neq e$ is $\psi\mapsto(T^{-1})^*\psi$ for $T=\gamma$ if $\gamma$ preserves the orientation, and for $T=\gamma\iota$ if $\gamma$ reverses the orientation.    
\end{remark}

\section{Structure of surfaces with dihedral symmetry}

The following statement formulates \cite[Lemma 2.9]{BuzanoNguyenSchulz2021} in the spherical case, where claim and proof are analogous. 
We emphasise that it does not assume the surface in question to be minimal. 

\begin{lemma}
\label{lem:Riemann-Hurwitz}
Given $3\leq k\in\N$, let $M_k \subset \Sp^3$ be any closed, connected, embedded, $\Z_k$-equivariant surface of genus $\alpha\in\{1,\ldots,k-1\}$, where $\Z_k$ acts as the cyclic subgroup of $\dih_k$ in standard representation. 
If $M_k$ is disjoint from the great circle $\Sp^1_\perp\subset\Sp^3$ fixed by the action of $\Z_k$, then $\alpha=1$. 
If $M_k$ intersects $\Sp^1_\perp$, then the number of intersections is $4$ and $M_k$ has genus $\alpha=k-1$.  
\end{lemma}

\begin{corollary}
\label{cor:Riemann-Hurwitz}
Given $3\leq k\in\N$, let $M_k \subset \Sp^3$ be any closed, embedded, $G$-equivariant minimal surface of genus $k-1$, where $G = \dih_k \times \Z_2 \times \Z_2$ acts via the representation \eqref{eqn:repR4}. 
Then each of the two great spheres $\Sp^2=\{x\in\Sp^3\mid x_4=0\}$ and $\Sp^2_\perp=\{x\in\Sp^3\mid x_3=0\}$ divides $M_k$ into two isometric, $\dih_k \times \Z_2$-equivariant, connected surfaces of genus zero with $k$ boundary components, i.\,e.~$k$-noids. 
\end{corollary}
 
\begin{proof}
The Frankel \cite{Frankel1966} property of closed minimal surfaces in $\Sp^3$ implies that $M_k\setminus\Sp^2$ has exactly two connected components, denoted $M_k^\pm$. 
The symmetry of $M_k$ implies that the two surfaces $M_k^+$ and $M_k^-$ are congruent; each of them has genus $\alpha\geq0$ and $\beta\geq1$ boundary components. 
The Euler characteristics of $M_k$ and $M_k^+$ satisfy 
\begin{align}\label{eqn:Eulerchar}
2-2(k-1)=\chi(M_k)=2\chi(M_k^+)=4-4\alpha-2\beta
\end{align}
which simplifies to $k=2\alpha+\beta$. 
Since all variables are nonnegative integers, 
$\beta\in\{1,\ldots,k\}$ and $0\leq\alpha\leq(k-1)/2$. 
We claim that necessarily $\alpha=0$ and $\beta=k$. 

As before, let $\Z_k$ denote the cyclic subgroup of $\dih_k$. 
By Lemma~\ref{lem:Riemann-Hurwitz}, $M_k$ intersects the great circle $\Sp^1_\perp\subset\Sp^3$ fixed by the action of $\Z_k$ in exactly $4$ points. 
In particular, at most two of the intersections could be on $\Sp^2$. 
Therefore, $M_k^+$ also intersects $\Sp^1_\perp$ at least once.  
Moreover, $M_k^+$ inherits the $\Z_k$-equivariance from $M_k$.  
We may then $\Z_k$-equivariantly close up each of the boundary components $M_k^+$ by a topological disc to obtain a closed, embedded, $\Z_k$-equivariant surface $\tilde M_k^+\subset\Sp^3$. 
(We refer e.\,g.~to the proof of \cite[Corollary~2.10]{BuzanoNguyenSchulz2021} for more details on this construction.)
By definition, this procedure preserves the genus $\alpha$. 
If $\alpha\geq1$ then Lemma \ref{lem:Riemann-Hurwitz} applied to $\tilde M_k^+$ yields $\alpha=k-1$, which contradicts \eqref{eqn:Eulerchar} for $\beta\geq1$ and $k\geq3$. 
Consequently, $\alpha=0$ and \eqref{eqn:Eulerchar} implies $\beta=k$ as claimed. 
The proof is analogous in the case where $\Sp^2$ is replaced by $\Sp^2_\perp$. 
\end{proof}

\subsection{Fundamental domains}
 
\begin{definition}[Fundamental domain]
\label{defn:fundamental_domain}
Given a finite group $G$ acting on a set $\Omega$, a closed connected subset $D\subset\Omega$ is called a \emph{fundamental domain} for $\Omega$ if $G(x)\cap D$ contains exactly one point for every $x\in\Omega$, where $G(x)$ denotes the orbit of $x$ under $G$.
\end{definition}

The following lemma is a variant of \cite[Lemma 7.3]{GirouardLagace2021} stated for a different group action. 

\begin{lemma}
\label{lem:fundamental_domain_knoid}
Given $3\leq k\in\N$, let $\Sigma_k\subset\B^3$ be any properly embedded, free boundary minimal surface of genus zero with $k$ boundary components which is $G=(\dih_k \times \Z_2)$-equivariant in the representation \eqref{eqn:prismatic}. 
Then $\Sigma_k$ has a simply connected fundamental domain $D$ with piecewise smooth boundary $\partial D$ consisting of four edges and four corners:  
Three corners have interior angle $\frac{\pi}{2}$, while the fourth is on the $x_3$-axis with interior angle $\frac{\pi}{k}$. 
\end{lemma}

\begin{proof}
We can realise a fundamental domain $W$ for the action of $G$ on $\B^3$ as a tetrahedral wedge bounded by $\partial\B^3$, by two adjacent planes of reflection symmetry $\Pi_1,\Pi_2$ of $\dih_k$ containing the $x_3$-axis, and by the plane $\Pi_3\vcentcolon=\{x_3=0\}$ fixed by the generator $\tau\in\Z_2$.  
Since $\Sigma_k$ is necessarily connected by \cite[Lemma~2.4]{FraserLi2014}, the set $D=\Sigma_k\cap W$ is connected by \cite[Lemma~7.2]{GirouardLagace2021}. 
By reflection symmetry and the free boundary condition, $D$ meets $\partial W$ orthogonally, so $\partial D$ is piecewise smooth with corners where it meets the edges of $W$, and the exterior angles along $\partial D$ are given by the angles between the corresponding faces of $\partial W$. 

Let $j$ and $\ell$ denote the numbers of exterior angles along $\partial D$ with values $\frac{\pi}{2}$ and $\pi-\frac{\pi}{k}$ respectively.
Since $W$ is a fundamental domain and $\partial\Sigma_k\neq\emptyset$, we have $\partial D\cap\partial\B^3\neq\emptyset$. 
Moreover, no connected component of $\partial\Sigma_k$ is fully contained in $W$, since its $G$-orbit would otherwise consist of $4k$ distinct boundary components of $\Sigma_k$. 
Hence $\partial D\cap\partial\B^3$ is a nonempty union of arcs whose endpoints are right-angled corners of $\partial D$, which implies $j\geq2$. 

Since $\Sigma_k$ has genus zero with $k$ boundary components, its Euler characteristic is $\chi(\Sigma_k)=2-k$. 
Denoting the Gauß curvature by $K$ and the geodesic curvature of the boundary by $\kappa$, the Gauß--Bonnet theorem yields 
\begin{align*}
2\pi(2-k)
=\int_{\Sigma_k} K+\int_{\partial\Sigma_k}\kappa
=\abs{G}\biggl(\int_{D}K+\int_{\partial D}\kappa\biggr)
=4k\Bigl(2\pi\chi(D)-j\tfrac{\pi}{2}-\ell\bigl(\pi-\tfrac{\pi}{k}\bigr)\Bigr), 
\end{align*}
which simplifies to 
\begin{align}
\label{eqn:diophantine-knoid}
4k\,\chi(D)&=kj+2\ell(k-1)+2-k. 
\end{align}
Any connected surface with boundary satisfies $\chi(D)\leq1$. 
Since $j\geq2$, the right hand side of \eqref{eqn:diophantine-knoid} is at least $k+2>0$. 
Hence $\chi(D)=1$ and equation \eqref{eqn:diophantine-knoid} simplifies to
\begin{align}
\label{eqn:diophantine-knoid-2}
kj+2\ell(k-1)&=5k-2. 
\end{align}
If $\ell=0$, then $k$ divides $2$, which is impossible for $k\geq3$. 
If $\ell=1$, then $kj=3k$ and thus $j=3$. 
If $\ell=2$, then $kj=k+2$, so again $k$ divides $2$, which is impossible. 
If $\ell=3$, then $kj=4-k$, contradicting $j\geq2$ and $k\geq3$. 
If $\ell\geq4$, the left hand side of \eqref{eqn:diophantine-knoid-2} is at least $2k+8(k-1)=10k-8>5k-2$. 
Hence $(j,\ell)=(3,1)$ is the only possibility, and $\partial D$ has $j+\ell=4$ corners as claimed. 
\end{proof}
 
\begin{corollary}
\label{cor:half_knoid}
In the setting of Lemma~\ref{lem:fundamental_domain_knoid}, let $H_1,\ldots,H_{2k}\subset\B^3$ denote the half-discs bounded by the $x_3$-axis which are contained in the planes of reflection of $\dih_k$, in cyclic order and suitably labelled.
Then the following holds. 
\begin{enumerate}[label={\normalfont(\roman*)}]
\item\label{cor:half_knoid-i} 
$\Sigma_k$ intersects the $x_3$-axis orthogonally in exactly two points $q$ and $\tau(q)$, and $\Sigma_k^+\vcentcolon=\Sigma_k\cap\{x_3\geq0\}$ is a disc whose boundary consists of $2k$ right-angled corners and $2k$ edges, alternating between $k$ arcs contained in $\Sigma_k\cap\{x_3=0\}$ and $k$ arcs contained in $\partial\Sigma_k$. 
\item\label{cor:half_knoid-ii} 
If $j$ is odd, then $\Sigma_k\cap H_j\cap\{x_3\geq0\}$ is an arc in $\{x_3>0\}$ from $q$ to a point in $\partial\Sigma_k$, whose other points lie in the interior of $\Sigma_k$. 
If $j$ is even, then $\Sigma_k\cap H_j\cap\{x_3\geq0\}$ is an arc in the interior of $\Sigma_k$ from $q$ to a point in $\{x_3=0\}$, whose other points lie in $\{x_3>0\}$.
\item\label{cor:half_knoid-iii} 
$\partial\Sigma_k\cap\{x_3=0\}$ consists of exactly $2k$ points. 
\end{enumerate}
\end{corollary}

\begin{proof}
Let $W$, $\Pi_1$, $\Pi_2$ and $\Pi_3=\{x_3=0\}$ be as in the proof of Lemma~\ref{lem:fundamental_domain_knoid}. 
Swapping the labels of $\Pi_1$ and $\Pi_2$ if necessary, the edges of the fundamental domain $D$ of $\Sigma_k$ are contained in $\Pi_1$, $\Pi_2$, $\Pi_3$, $\partial\B^3$ in cyclic order.
The images of $W\subset\B^3$ under $\dih_k$ are the $2k$ wedges between consecutive half-discs $H_j$, and we label them such that $H_j\cap\{x_3\geq0\}$ is an image of the face $\Pi_1\cap W$ of $W$ for odd $j$ and of the face $\Pi_2\cap W$ for even $j$. 
Since every $G$-orbit meets $W$, all claims follow by reflecting $D$ across the faces of $W$. 
\end{proof}

\begin{remark}\label{rem:half-sphere}
The statements and proofs of Lemma~\ref{lem:fundamental_domain_knoid} and 
Corollary~\ref{cor:half_knoid} extend verbatim to free boundary minimal surfaces $\Sigma_k$ in
spherical caps (see Definition~\ref{defn:spherical_cap_case}) including 
the half-sphere 
$\{x\in\Sp^3\mid x_4\geq0\}$ in place of $\B^3$.
\end{remark}

The following statement verifies \cite[Assumption~5.1]{KusnerMcGrath2020} with $G=2\mathord\ast2$ and $H=\mathord\ast22$. 

\begin{lemma} 
\label{lem:fundamental_domain_torus}
Let $\Sigma\subset\B^3$ be any compact, connected, properly embedded surface of genus $1$ with connected boundary which is $\dih_4$-equivariant in the antiprismatic representation \eqref{eqn:antiprismatic}.
Let $\pi_G$ and $\pi_H$ denote the projections onto the quotients of $\Sigma$ over $G=\dih_4$ respectively over the subgroup $H\vcentcolon=\brt{\gamma_2,\gamma_1\gamma_2\gamma_1}$. 
For $i\in\{1,2,3\}$ let $\xi_i$ denote the restriction of the $x_i$-axis to $\B^3$.  
Then the following properties hold. 
\begin{enumerate}[label={\normalfont(\roman*)}]
\item\label{lem:fundamental_domain_torus-0}
$\Sigma\cap\xi_3=\{-p,0,p\}$ consists of three interior points for some $p=(0,0,p_3)$ with $p_3\in\interval{0,1}$. 
\item\label{lem:fundamental_domain_torus-i}
$\xi_1\cup\xi_2\subset\Sigma$ and the relative orientation representation (see Definition~\ref{defn:relative_orientation_representation}) is $\epsilon=\m_1$.
\item\label{lem:fundamental_domain_torus-ii} 
$\pi_H(\Sigma)$ is simply connected and separated by $\pi_H(\xi_1)$. 
\item\label{lem:fundamental_domain_torus-iii} 
$\pi_G(\Sigma)$ is a quadrilateral whose edges are, in cyclic order, contained in $\pi_G(\Fix(\gamma_1))$, $\pi_G(\partial\Sigma)$, $\pi_G(\Fix(\gamma_2))$ and again $\pi_G(\Fix(\gamma_2))=\pi_G(\Fix(\gamma_1\gamma_2\gamma_1))$. 
\end{enumerate} 
\end{lemma}

\begin{proof}
Being properly embedded in $\B^3$, the surface $\Sigma$ admits a global unit normal vector field $\nu$. 
Recalling the relative orientation representation $\epsilon\colon\dih_4\to\{\pm1\}$ from Definition~\ref{defn:relative_orientation_representation}, we have 
\[
\epsilon\bigl((\gamma_1\gamma_2)^2\bigr)=\bigl(\epsilon(\gamma_1\gamma_2)\bigr)^2=1,
\]
so $(\gamma_1\gamma_2)^2(x_1,x_2,x_3)=(-x_1,-x_2,x_3)$ preserves $\nu$. 
Since also $\det\rho\bigl((\gamma_1\gamma_2)^2\bigr)=1$, the involution $(\gamma_1\gamma_2)^2$ preserves the orientation of $\Sigma$ and $\partial\Sigma$, but does not act as the identity, as 
\[
\Fix\bigl((\gamma_1\gamma_2)^2\bigr)=\xi_3.
\]
Consequently, $(\gamma_1\gamma_2)^2$ does not fix any point on $\partial\Sigma$, and 
$\Sigma\cap\xi_3$ consists of finitely many interior points where the intersection is orthogonal by \cite[Lemma~3.4\,(2)]{Ketover2016Equivariant}.   
Let $j\in\N\cup\{0\}$ denote their number and let $F\cong\Z_2$ denote the subgroup of $H$ generated by $(\gamma_1\gamma_2)^2$. 
The quotient $\Sigma/F$ is an orientable topological surface with connected boundary. 
Hence, its Euler characteristic is $\chi(\Sigma/F)=1-2g$ for some integer $g\geq0$. 
A variant of the Riemann--Hurwitz formula (see e.\,g.~\cite[§ IV.3]{Freitag2011}) implies 
\[
-1=\chi(\Sigma)=2\chi(\Sigma/F)-j 
\]
which simplifies to $j=3-4g$. 
Since both variables are nonnegative integers, we obtain $g=0$ and $j=3$. 
This means that $\Sigma/F$ is a topological disc and that $\Sigma$ intersects $\xi_3$ in exactly three points. 
Since the action of $\gamma_1$ restricts to $x_3\mapsto-x_3$ on $\xi_3$, and since an involution of a set with three elements has a fixed point, we conclude $\Sigma\cap\xi_3=\{-p,0,p\}$ for some $p=(0,0,p_3)$ with $p_3\in\interval{0,1}$, as stated in \ref{lem:fundamental_domain_torus-0}.

The tangent space $T_0\Sigma\subset\R^3$ is a $\dih_4$-invariant plane. 
Since the representation \eqref{eqn:antiprismatic} decomposes into the irreducible representation $\st_1$ of degree $2$ on $\{x_3=0\}$ and the $1$-dimensional irreducible representation $\m_1$ on $\xi_3$, the only $\dih_4$-invariant plane in $\R^3$ is $\{x_3=0\}$. 
Consequently $d\gamma_1\vert_{T_0\Sigma}$ and $d\gamma_2\vert_{T_0\Sigma}$ are the reflections $(x_1,x_2)\mapsto(x_1,-x_2)$ and $(x_1,x_2)\mapsto(x_2,x_1)$ respectively, both orientation-reversing. 
Together with $\det\rho(\gamma_1)=1$ and $\det\rho(\gamma_2)=-1$ we obtain $\epsilon(\gamma_1)=-1$ and $\epsilon(\gamma_2)=1$, that is, $\epsilon=\m_1$ as defined in Section~\ref{sec:dihedral}. 
Recalling \eqref{eqn:extrinsicD4action}, the $\Z_2$-action of $\gamma_1$ respectively $\gamma_2\gamma_1\gamma_2$ fixes $\xi_1$ respectively $\xi_2$, and both axes are tangential to $\Sigma$ at the origin. 
Hence, $\xi_1\cup \xi_2\subset\Sigma$ by \cite[Lemma~3.4\,(2)]{Ketover2016Equivariant} as claimed in \ref{lem:fundamental_domain_torus-i}.  

It remains to prove \ref{lem:fundamental_domain_torus-ii} and \ref{lem:fundamental_domain_torus-iii}.  
Being generated by the reflections $\gamma_2$ and $\gamma_1\gamma_2\gamma_1$ 
across the orthogonal planes $\Pi_1=\{x_1=x_2\}$ and $\Pi_2=\{x_1=-x_2\}$, 
we have $H\cong\Z_2\times\Z_2$ and $F$ generated by $(\gamma_1\gamma_2)^2$ is a normal subgroup of $H$. 
Hence 
\begin{align}\label{eqn:doublequotient}
\Sigma/H=(\Sigma/F)\big/(H/F).
\end{align}
As shown above, $\Sigma/F$ is a topological disc and thus simply connected. 
$H/F\cong\Z_2$ is generated by the class of $\gamma_2$ which has fixed points, for instance the image of the origin. 
By \cite{Armstrong1968}, the quotient \eqref{eqn:doublequotient} is therefore simply connected as well. 
Being compact with nonempty boundary, $\Sigma/H$ is a topological disc. 

Since $\epsilon=\m_1$ and $\det\rho(\gamma_2)=\det\rho(\gamma_1\gamma_2\gamma_1)=-1$, the reflections $\gamma_2$ and $\gamma_1\gamma_2\gamma_1$ reverse the orientation of $\Sigma$ and therefore also the orientation of the simple closed curve $\partial\Sigma$.
Consequently, $\gamma_2$ and $\gamma_1\gamma_2\gamma_1$ each fix exactly two points of $\partial\Sigma$.
These four points are pairwise distinct and divide $\partial\Sigma$ into four arcs which are permuted simply transitively by the action of $H$, so $\pi_H(\partial\Sigma)$ is a single arc with one endpoint in $\Pi_1$ and the other in $\Pi_2$. 
Moreover, since $\gamma_2$ and $\gamma_1\gamma_2\gamma_1$ act on $\Sigma$ as orientation-reversing isometric involutions, their fixed point sets $\Sigma\cap\Pi_1$ and $\Sigma\cap\Pi_2$ are one-dimensional, intersecting at the three points $\{-p,0,p\}=\Sigma\cap\xi_3$.
Since $\Sigma/H$ is a topological disc, its boundary therefore consists of five edges, namely $\pi_H(\partial\Sigma)$ together with two edges in $\pi_H(\Sigma\cap\Pi_1)$ and two edges in $\pi_H(\Sigma\cap\Pi_2)$, alternating at the corners $\pi_H(-p)$, $\pi_H(0)$ and $\pi_H(p)$.

Because $\xi_1\subset\Sigma$ is fixed as a set by $(\gamma_1\gamma_2)^2$ (via reflection through the origin) and mapped to $\xi_2$ by $\gamma_2$, the image $\pi_H(\xi_1)$ is a properly embedded curve in $\Sigma/H$ and thus separates the 
topological disc $\Sigma/H$ into two connected components, which completes the proof of \ref{lem:fundamental_domain_torus-ii}. 

Since $H$ is a normal subgroup of index two in $G$ and the class of $\gamma_1$ generating $G/H\cong\Z_2$ fixes the point $\pi_H(0)$, the same arguments as for \eqref{eqn:doublequotient} prove that $\Sigma/G=(\Sigma/H)\big/(G/H)$ is a topological disc.
Its boundary consists of four edges: 
the line $\pi_G(\xi_1)=\pi_G(\Fix(\gamma_1))$ starting from $\pi_G(0)$, followed by  $\pi_G(\partial\Sigma)$, 
and one curve in each of $\pi_G(\Fix(\gamma_1\gamma_2\gamma_1))$ and $\pi_G(\Fix(\gamma_2))$. 
This proves~\ref{lem:fundamental_domain_torus-iii}.
\end{proof}

\subsection{Quadratic independence}
\label{sec:quadratic_independence}

When computing the nullity, we rely on Proposition~\ref{prop:rigidity} which has quadratic independence in the sense of Definition~\ref{defn:quadratic_independence} as an assumption.   
Below we verify this condition in the relevant settings. 

\begin{lemma}
\label{lem:knoid_quadratic_independence}
Given $3\leq k\in\N$, let $\Sigma_k\subset\B^3$ be any properly embedded free boundary minimal $k$-noid which is $\dih_k\times\Z_2$-equivariant in the prismatic representation \eqref{eqn:prismatic} and let $\Phi\colon\Sigma_k\to\B^3$ denote the inclusion map. 
Then the components of $\Phi$ are quadratically independent on $\partial\Sigma_k$. 
\end{lemma}

\begin{proof}
Let $B=(b_{ij})$ be a symmetric $3\times3$ matrix such that 
$\brt{Bx,x}=0$ for all $x\in\partial\Sigma_k$.
Given $x=(x_1,x_2,x_3)\in\partial\Sigma_k$ we also have
\begin{align*}
\gamma_1x &= (x_1,-x_2,x_3)\in\partial\Sigma_k, & 
\tau x  &= (x_1,x_2,-x_3)\in\partial\Sigma_k, & 
\gamma_1\tau x &= (x_1,-x_2,-x_3)\in\partial\Sigma_k
\end{align*}
by assumption. 
Evaluating $\brt{Bx,x}=0$ in each case leads to the four equations  
\begin{align}
\label{eqn:quadratic_knoid-I}
b_{11}x_1^2+b_{22}x_2^2+b_{33}x_3^2 + 2b_{12}x_1x_2 + 2b_{13}x_1x_3 + 2b_{23}x_2x_3 &= 0, 
\\
\label{eqn:quadratic_knoid-II}
b_{11}x_1^2+b_{22}x_2^2+b_{33}x_3^2 - 2b_{12}x_1x_2 + 2b_{13}x_1x_3 - 2b_{23}x_2x_3 &= 0, 
\\
\label{eqn:quadratic_knoid-III}
b_{11}x_1^2+b_{22}x_2^2+b_{33}x_3^2 + 2b_{12}x_1x_2 - 2b_{13}x_1x_3 - 2b_{23}x_2x_3 &= 0, 
\\
\label{eqn:quadratic_knoid-IV}
b_{11}x_1^2+b_{22}x_2^2+b_{33}x_3^2 - 2b_{12}x_1x_2 - 2b_{13}x_1x_3 + 2b_{23}x_2x_3 &= 0.
\end{align}
Summing \eqref{eqn:quadratic_knoid-I}--\eqref{eqn:quadratic_knoid-IV} cancels all three cross-terms, while the combinations $\text{\eqref{eqn:quadratic_knoid-I}}-\text{\eqref{eqn:quadratic_knoid-II}}+\text{\eqref{eqn:quadratic_knoid-III}}-\text{\eqref{eqn:quadratic_knoid-IV}}$, respectively $\text{\eqref{eqn:quadratic_knoid-I}}+\text{\eqref{eqn:quadratic_knoid-II}}-\text{\eqref{eqn:quadratic_knoid-III}}-\text{\eqref{eqn:quadratic_knoid-IV}}$ and $\text{\eqref{eqn:quadratic_knoid-I}}-\text{\eqref{eqn:quadratic_knoid-II}}-\text{\eqref{eqn:quadratic_knoid-III}}+\text{\eqref{eqn:quadratic_knoid-IV}}$, isolate the individual cross-terms. 
Hence 
\begin{align}
\label{eqn:quadratic_knoid-diag}
b_{11}x_1^2+b_{22}x_2^2+b_{33}x_3^2&=0, 
\\
\label{eqn:quadratic_knoid-cross}
b_{12}x_1x_2
=b_{13}x_1x_3
=b_{23}x_2x_3
&=0 
\end{align}
for every $x\in\partial\Sigma_k$. 
Since $\partial\Sigma_k$ is not contained in the union of any two coordinate planes, \eqref{eqn:quadratic_knoid-cross} yields $b_{12}=b_{13}=b_{23}=0$. 
Since $k\geq3$, averaging \eqref{eqn:quadratic_knoid-diag} over the images of $x$ under the cyclic subgroup $\Z_k$ (generated by the rotation of angle $2\pi/k$ in the $x_1x_2$-plane) replaces the coefficients $b_{11}$ and $b_{22}$ by their mean and yields 
\(
\tfrac{1}{2}(b_{11}+b_{22})(x_1^2+x_2^2)+b_{33}x_3^2=0.
\)
As $\partial\Sigma_k\subset\Sp^2$, this simplifies to 
\begin{align}
\label{eqn:20260913}
\tfrac{1}{2}(b_{11}+b_{22})+\bigl(b_{33}-\tfrac{1}{2}(b_{11}+b_{22})\bigr)x_3^2=0.
\end{align} 
Recalling that the action of the generator $\tau\in\Z_2$ preserves every component of $\partial\Sigma_k$ while reversing $x_3$, we conclude that $x_3$ vanishes somewhere on $\partial\Sigma_k$ but not identically.  
Consequently, \eqref{eqn:20260913} implies $b_{11}+b_{22}=0$ and $b_{33}=0$. 
Then \eqref{eqn:quadratic_knoid-diag} reduces to $b_{11}(x_1^2-x_2^2)=0$, and since $\partial\Sigma_k$ is not contained in the union of the planes $\{x_1=x_2\}$ and $\{x_1=-x_2\}$, we obtain $-b_{22}=b_{11}=0$ and thus $B=0$ as claimed. 
\end{proof}

\begin{remark}
\label{rem:quadratic_knoid-diag}
An alternative argument for Lemma~\ref{lem:knoid_quadratic_independence} is contained in the proof of \cite[Proposition~9.6]{KarpukhinKusnerMcGrathStern}, where the symmetry assumption is replaced by the requirement that the surface is embedded by first Steklov eigenfunctions of multiplicity three. 
In either case the number of boundary components must differ from $2$, as the boundary of the critical catenoid lies on two parallel planes.
\end{remark} 

\begin{lemma}
\label{lem:lawson_quadratic_independence}
Given $3\leq k\in\N$ let $M_k\subset\Sp^3$ be any closed, embedded minimal surface 
of genus $k-1$ which is $G=(\dih_k\times\Z_2\times\Z_2)$-equivariant in the representation \eqref{eqn:repR4}, and let $\Phi\colon M_k\to\Sp^3$ denote the inclusion map. 
Then the components of $\Phi$ are quadratically independent. 
\end{lemma}

\begin{proof}
Let $B=(b_{ij})$ be a symmetric $4\times4$ matrix such that $\brt{Bx,x}=0$ for all $x\in M_k$. 
Recalling \eqref{eqn:repR4}, the generator $\gamma_1\in\dih_k$ together with $\tau_1$ and $\tau_2$ generate a subgroup of $G$ of order eight, acting on $\R^4$ by the independent sign changes 
\begin{align*}
x=(x_1,x_2,x_3,x_4)&\mapsto(x_1,\pm x_2,\pm x_3,\pm x_4), 
\end{align*}
under which $M_k$ is invariant. 
Evaluating $\brt{Bx,x}=0$ at the eight images of $x$ and forming suitable linear combinations as in the proof of Lemma~\ref{lem:knoid_quadratic_independence}, we therefore obtain
\begin{align}
\label{eqn:quadratic_lawson-diag}
b_{11}x_1^2+b_{22}x_2^2+b_{33}x_3^2+b_{44}x_4^2&=0, 
\\
\label{eqn:quadratic_lawson-cross}
 b_{12}x_1x_2
=b_{13}x_1x_3
=b_{14}x_1x_4
=b_{23}x_2x_3
=b_{24}x_2x_4
=b_{34}x_3x_4
&=0
\end{align}
for every $x\in M_k$. 
Since $M_k$ is smooth of nonzero genus, it is not contained in the union of any two great spheres and \eqref{eqn:quadratic_lawson-cross} yields $b_{ij}=0$ for all integers $1\leq i<j\leq 4$. 

Since $k\geq3$, averaging \eqref{eqn:quadratic_lawson-diag} over the images of $x$ under the cyclic subgroup $\Z_k$ of $\dih_k$ replaces the coefficients $b_{11}$ and $b_{22}$ by their mean $a\vcentcolon=\tfrac{1}{2}(b_{11}+b_{22})$ and yields 
\(
a(x_1^2+x_2^2)+b_{33}x_3^2+b_{44}x_4^2=0.
\)
As $M_k\subset\Sp^3$, this simplifies to 
\begin{align}
\label{eqn:quadratic_lawson-sphere}
a+(b_{33}-a)x_3^2+(b_{44}-a)x_4^2&=0.
\end{align}
Recalling $\Sp^2_\perp=\{x\in\Sp^3\mid x_3=0\}$, Corollary~\ref{cor:Riemann-Hurwitz} 
states that the intersection $M_k\cap\Sp^2_\perp$ consists of $k$ disjoint curves. 
Along these we have $x_3=0$, so \eqref{eqn:quadratic_lawson-sphere} reduces to 
\begin{align}
\label{eqn:reduced_equation}
a+(b_{44}-a)x_4^2=0 
\end{align}
for all $x\in M_k\cap\Sp^2_\perp$. 
If $b_{44}\neq a$, then $x_4^2$ would be constant on $M_k\cap\Sp^2_\perp$, which would therefore be contained in the set $\{x\in\Sp^3\mid x_3=0,~x_4^2=-a/(b_{44}-a)\}$ consisting of at most two circles, contradicting $k\geq3$. 
Hence $b_{44}=a$ and \eqref{eqn:reduced_equation} yields $a=0$. 
Interchanging the roles of $x_3$ and $x_4$ and arguing along the $k$ circles of $M_k\cap\Sp^2$ likewise yields $b_{33}=a=0$. 
Recalling $a=\tfrac{1}{2}(b_{11}+b_{22})$, equation \eqref{eqn:quadratic_lawson-diag} reduces to $b_{11}(x_1^2-x_2^2)=0$. 
Since $M_k$ is not contained in the union of the great spheres $\{x_1=x_2\}$ and $\{x_1=-x_2\}$, we conclude $-b_{22}=b_{11}=0$ and thus $B=0$ as claimed. 
\end{proof}

\begin{lemma}
\label{lem:torus_quadratic_independence}
Let $M_1\subset\B^3$ be as in Theorem~\ref{thm:index-fbm-torus} and let $\Phi\colon M_1\to\B^3$ denote the inclusion map.
Then the components of $\Phi$ are quadratically independent on $\partial M_1$.
\end{lemma}

\begin{proof}
Let $B=(b_{ij})$ be a symmetric $3\times3$ matrix such that for all $x=(x_1,x_2,x_3)\in\partial M_1$
\begin{align}\label{eqn:quadratic}
\brt{Bx,x} = b_{11}x_1^2 + b_{22}x_2^2 + b_{33}x_3^2 + 2b_{12}x_1x_2 + 2b_{13}x_1x_3 + 2b_{23}x_2x_3 = 0.
\end{align}
Lemma~\ref{lem:fundamental_domain_torus}\,\ref{lem:fundamental_domain_torus-i} 
implies $(\pm1,0,0),\,(0,\pm1,0)\in\partial M_1$ and thus $b_{11}=b_{22}=0$.
Evaluating \eqref{eqn:quadratic} at any $x\in\partial M_1$ and its images under the symmetries $\gamma_1,(\gamma_1\gamma_2)^2,\gamma_1(\gamma_1\gamma_2)^2 \in\dih_4$ we obtain 
\begin{align}
\label{eqn:quadratic-I}
x&=(x_1,x_2,x_3) &&\Rightarrow & 
b_{33}x_3^2 + 2b_{12}x_1x_2 + 2b_{13}x_1x_3 + 2b_{23}x_2x_3 &= 0, 
\\
\label{eqn:quadratic-II}
\gamma_1x &= (x_1,-x_2,-x_3) &&\Rightarrow & 
b_{33}x_3^2 - 2b_{12}x_1x_2 - 2b_{13}x_1x_3 + 2b_{23}x_2x_3 &= 0, 
\\
\label{eqn:quadratic-III}
(\gamma_1\gamma_2)^2 x  &= (-x_1,-x_2,x_3) &&\Rightarrow & 
b_{33}x_3^2 + 2b_{12}x_1x_2 - 2b_{13}x_1x_3 - 2b_{23}x_2x_3 &= 0, 
\\
\label{eqn:quadratic-IV}
\gamma_1(\gamma_1\gamma_2)^2 x &= (-x_1,x_2,-x_3) &&\Rightarrow & 
b_{33}x_3^2 - 2b_{12}x_1x_2 + 2b_{13}x_1x_3 - 2b_{23}x_2x_3 &= 0.
\end{align}
Summing equations \eqref{eqn:quadratic-I}--\eqref{eqn:quadratic-IV} cancels all three cross-terms and yields $4b_{33}x_3^2 = 0$. 
Since $M_1$ has genus $1$, it is not contained in the plane $\{x_3=0\}$, so the harmonic coordinate function $x_3$ attains a strictly positive maximum on $\partial M_1$ and we obtain $b_{33}=0$. 
Adding equation \eqref{eqn:quadratic-I} to equations \eqref{eqn:quadratic-II}--\eqref{eqn:quadratic-IV} then yields 
\begin{align*}
4b_{23}x_2x_3 &= 0, &
4b_{12}x_1x_2 &= 0, & 
4b_{13}x_1x_3 &= 0 
\end{align*}
for every $x\in\partial M_1$. 
Since $\partial M_1$ is a smooth connected curve which cannot be contained in any of the coordinate planes,
$b_{23}=b_{12}=b_{13}=0$ follows, so necessarily $B=0$ as claimed.
\end{proof}

 
\setlength{\parskip}{1ex plus 1pt minus 2pt}
\bibliography{fbms-bibtex}

\printaddress

\end{document}